\documentclass[a4paper]{article}      

\usepackage{amsmath,amssymb,amsfonts,amsthm,mathrsfs,tikz} 
\usepackage{graphicx}
\usepackage{color}
\usepackage{verbatim} 
\usepackage[margin=1cm,font=small]{caption} 
\usepackage{layout} 

\newtheorem{Theorem}{Theorem}
\newtheorem{Proposition}{Proposition}
\newtheorem{Corollary}{Corollary}
\newtheorem{Lemma}{Lemma}

\numberwithin{equation}{section}

\newcommand{\rd}{\color{red}}

\newcommand{\R}{\mathbb{R}}

\newcommand{\Z}{\mathbb{Z}}
\newcommand{\N}{\mathbb{N}}

\renewcommand{\P}{\mathbb{P}}
\newcommand{\E}{\mathbb{E}}

\newcommand{\dd}{\mathrm{d}}

\newcommand{\ed}{\mathrm{e}}

\newcommand{\Const}{C}
\newcommand{\const}{c}
\newcommand{\8}{\infty}

\newcommand{\vart}{{\mathsf{t}}}
\newcommand{\varc}{{\mathsf{c}}}

\begin{document}

\title{Random walk driven by the simple exclusion process}

\author{
Fran\c cois Huveneers, 
Fran\c cois Simenhaus
}
\date{\today}

\maketitle 

\begin{abstract}
We prove a strong law of large numbers and an annealed invariance principle for a random walk in a one-dimensional dynamic random environment evolving as the simple exclusion process with jump parameter $\gamma$. 
First, we establish that if the asymptotic velocity of the walker is non-zero in the limiting case ``$\gamma = \infty$", where the environment gets fully refreshed between each step of the walker,
then, for $\gamma$ large enough, the walker still has a non-zero asymptotic velocity in the same direction. 
Second, we establish that if the walker is transient in the limiting case $\gamma = 0$, then, for $\gamma$ small enough but positive, the walker has a non-zero asymptotic velocity in the direction of the transience. 
These two limiting velocities can sometimes be of opposite sign. 
In all cases, we show that the fluctuations are normal. 
\end{abstract}

\section{Introduction}\label{Section: Introduction}
The question of the evolution of a random walk in a disordered environment has attracted a lot of attention in both the mathematical and the physical communities over the past few decades. 
The first studies were concerned with static random environments. 
In this set-up, anomalous slowdowns are expected in comparison with the homogeneous case, as the environment may create traps where the walker gets stuck for long times. 
This effect is particularly strong in one dimension, where it is by now well understood (see \cite{Zeitouni} for background). 
In dynamical random environments instead, the transition probabilities of the walker evolve with time too. 
If the environment has good space-time mixing properties, one expects the trapping phenomenon to disappear, and the walker to behave very much as if the medium was homogeneous. 
The study of this case has recently led to intense researches;
see for example \cite{Dolgopyat Keller Liverani,Redig Vollering} and references therein, as well as \cite{Avena} for an overview and further references. 

However, examples of dynamical environments with slow relaxation times occur naturally.  
Indeed, in the presence of a macroscopically conserved quantity, the environment may evolve diffusively. 
Time correlations then only decay as $t^{-d/2}$ in dimension $d$. 
{Especially for $d=1$, correlations decay so slowly that the results developed for fast mixing environments do not apply
(see e.g.\@ \cite{Avena den Hollander Redig,Bricmont_Kupiainen_09,Dolgopyat Keller Liverani,Rassoul,Redig Vollering})}.
{Slowly mixing} environments form an intermediate class of models, which is still far for being well understood at the present time \cite{Avena,Avena Thomann}. 

In this paper, we are interested in the asymptotic behavior as $t\to \8$ of $X_t/t$, 
where $X_t$ denotes the position at time $t$ of a walker driven by the one-dimensional simple exclusion process, at equilibrium with density $0 < \rho <1$.
The walker evolves in discrete time: 
if he sits on a particle at the moment of jumping, he moves to the right with probability $\alpha$ and to the left with probability $1-\alpha$ for some $0< \alpha < 1$, 
while, if he sits on a vacant site, these probabilities become respectively $\beta$ and $1 - \beta$ for some $0 < \beta < 1$ (see Figure \ref{figure: Walker}, as well as Section \ref{subsec:model} below for a more rigorous description). 
Of particular interest is the case $\alpha < \frac{1}{2} < \beta$ or vice versa, as we then see reappearing the possibility of trapping mechanisms. 

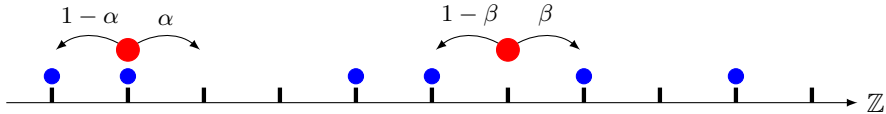
\begin{figure}[h!]
\begin{center}
\begin{tikzpicture}[scale=1]

\draw[->,>=latex] (-0.6,0) -- (10.6,0);

\draw (10.6,0) node[right] {$\Z$};

\draw [ultra thick] (0,0) -- (0,0.2);
\draw [ultra thick] (1,0) -- (1,0.2);
\draw [ultra thick] (2,0) -- (2,0.2);
\draw [ultra thick] (3,0) -- (3,0.2);
\draw [ultra thick] (4,0) -- (4,0.2);
\draw [ultra thick] (5,0) -- (5,0.2);
\draw [ultra thick] (6,0) -- (6,0.2);
\draw [ultra thick] (7,0) -- (7,0.2);
\draw [ultra thick] (8,0) -- (8,0.2);
\draw [ultra thick] (9,0) -- (9,0.2);
\draw [ultra thick] (10,0) -- (10,0.2);

\draw [blue,fill=blue] (0,0.35) circle (0.1);

\draw [->,>=latex] (1,0.7) to [bend left =40] (1.95,0.7);
\draw [->,>=latex] (1,0.7) to [bend right =40] (0.05,0.7);

\draw [blue,fill=blue] (1,0.35) circle (0.1);
\draw [red,fill=red] (1,0.7) circle (0.15);

\draw (1.5,0.9) node[above] {\small $\alpha$};
\draw (0.5,0.9) node[above] {\small $1-\alpha$};

\draw [blue,fill=blue] (4,0.35) circle (0.1);

\draw [blue,fill=blue] (5,0.35) circle (0.1);

\draw [->,>=latex] (6,0.7) to [bend left =40] (6.95,0.7);
\draw [->,>=latex] (6,0.7) to [bend right =40] (5.05,0.7);

\draw [red,fill=red] (6,0.7) circle (0.15);

\draw (6.5,0.9) node[above] {\small $\beta$};
\draw (5.5,0.9) node[above] {\small $1-\beta$};

\draw [blue,fill=blue] (7,0.35) circle (0.1);

\draw [blue,fill=blue] (9,0.35) circle (0.1);

\end{tikzpicture}
\end{center}
\caption{
\label{figure: Walker}
Transition probabilities for the walker (red) evolving on top of the particles of the simple exclusion process (blue). 
}
\end{figure}

Let $\gamma$ be the jump rate of the particles of the exclusion dynamics. 
Two regimes, depending on the value of $\gamma$, are considered in this article. 
We first deal with the fluid regime $\gamma \gg 1$. 
In the formal limit ``$\gamma=\infty$", the walker evolves as in an homogeneous medium: 
he jumps to the right with the homogenized probability $\tilde p = \rho \alpha + (1 - \rho) \beta$ and to the left with probability $1 - \tilde p$. 
We assume $\tilde p \ne \frac{1}{2}$, and say $\tilde p > \frac{1}{2}$ to fix things.
The walker drifts thus to the right in the limiting case ``$\gamma=\infty$" and the fluctuations are normal. 
Let us now take $1 \ll\gamma<\infty$. 
In the time interval left between each step of the walker, the exclusion process mixes the particles in boxes of size $\gamma^{1/2}$. 
Therefore, most of the time, the walk will behave in the same way as the limiting homogeneous walk.   
However, the walker will eventually enter some regions where the density of particles is anomalously high or small, and behave then differently. 
But regions of size $l$ with an anomalous density of particles (with respect to the equilibrium measure) are typically at distance $\ed^{c l}$ from the origin for some $c>0$, 
while only a time of order $l^2$ is needed for the dynamics to disaggregate them. 
Therefore, these regions do not act as efficient barriers to the evolution of the walker, which exhibits thus a positive asymptotic velocity. 
This is made precise in Theorem~\ref{Theorem: Drift to the right}, where it is also shown that fluctuations are normal.


We next deal with the quasi-static regime $\gamma \ll 1$. 
The dynamics is then no longer dominated by the homogenized probability $\tilde p$. 
Instead of the condition $\tilde p \ne \frac{1}{2}$, we now assume that, in the limiting static case $\gamma = 0$, 
the walk is transient, say to the right ({see the criterium \eqref{static transience} below, as well as Chapter 2.1 in \cite{Zeitouni}}). 
Under this hypothesis, when $\gamma = 0$, it is known that $X_t$ {behaves as} $t^\delta$ as $t\to\infty$ for some $0 < \delta \le 1$. 
{The sub-ballistic behavior corresponding to cases where $\delta < 1$ appears when $\rho \frac{1-\alpha}{\alpha}+(1-\rho)\frac{1-\beta}{\beta}> 1$ and is due to the fluctuations of the environment (see \cite {KKS} for precise results on this regime).}
Let us now take $\gamma > 0$. 
The environment does not evolve significantly over a time $t=\tau\gamma^{-1}$ for some small $\tau >0$. 
Therefore, taking $\gamma$ small enough, the walker will move to the right by an amount of order $t^\delta$ if started in a typical environment with respect to the Gibbs measure. 
Thus, as long as the walk evolves in such configurations of the environment, it drifts to the right. 
Nevertheless, when entering a region where the density of particles is anomalous, its progression may be greatly slowed. 
However, since $\gamma$ is positive, we may just reuse the argument developed for large $\gamma$ and show that traps are irrelevant: 
large traps (of size $l$) disappear on much shorter time scales ($\gamma^{-1}l^2$) 
than the time needed for the walker to see a next trap of that size (a time at least of order $l^{K}$ for some $K\gg 1$, see Section \ref{Section: small gamma}). 
We conclude that the walker has a positive velocity to the right, as stated in Theorem \ref{Theorem: Drift to the left}.
Again, we show also that fluctuations are normal. 


From these two results, we conclude that the value of the limiting velocity $v(\gamma)$ may change drastically between the large an small $\gamma$ regimes. 
Indeed, for some values $\rho$, $\alpha$ and $\beta$, the walk is transient to the {right} in a static environment while the homogenized drift $2\tilde{p}-1$ is negative
(take e.g.\@ $\rho=9/10$, $\alpha=1/4$ and $\beta=1-\epsilon$ with $\epsilon$ small enough in \eqref{static transience} and \eqref{drift condition} below).
In this case, the asymptotic velocity in a static environment is necessarily zero (Jensen's inequality implies $\rho \frac{1-\alpha}{\alpha}+(1-\rho)\frac{1-\beta}{\beta}> 1$).
Therefore, it holds that $v(0)=0$, $v(\gamma) > 0$ for small enough $\gamma > 0$, and $v(\gamma) < 0$ for large enough $\gamma$.

We did not show that $v(\gamma)$ converges to $v(0)$ as $\gamma$ goes to $0$, nor to $2\tilde{p}-1$ as $\gamma$ goes to $\8$, but we expect both of these limits to hold.
While the latter could be shown using the techniques of the present article, at the cost of slightly more involved estimates (control on both excess and shortage of particles), 
some more work should be needed to check the continuity at $0$.  

Theorems \ref{Theorem: Drift to the right} and \ref{Theorem: Drift to the left} are shown in essentially the same way.
In a first step, by means of a multi-scale analysis, we show recursively that the walker can pass larger and larger traps, allowing for a set of initial environments of larger and larger measure.
The introduction of renormalization techniques to study random walks in random environments goes back to \cite{Bricmont_Kupiainen_91} and \cite{Bricmont_Kupiainen_09}, from where our strategy is inspired. 
A similar method is used in \cite{Hilario et al} (see also \cite{den Hollander Kesten Sidoravicius} and \cite{dos_Santos} for a slightly different approach).
From this first step, one concludes that the walker drifts almost surely to the left or to the right. 

The law of large numbers and the invariance principle are deduced in a second step. 
As the environment evolves only on diffusive time scales, a ballistic walker discovers fresh randomness most of the time. 
From this observation, it is possible to build up a renewal structure, allowing to cut the full trajectory into pieces that are mutually independent.
This idea was first introduced in \cite{SZ} for the case of a static i.i.d.\@ environment, 
and further adapted by \cite{Comets Zeitouni} to deal with the case of static environments with good mixing properties.
It was exploited in \cite{Avena den Hollander Redig} to obtain a law of large numbers for dynamic environments with good mixing rates, and then in \cite{Avena dos Santos Vollering,Berard Ramirez,Hilario et al}
{for a one-dimensional diffusive environment}.
As this method is rather delicate and model dependent, we had to perform specific constructions and estimates (see Section \ref{Section: Proof of law of large numbers}). 

Let us now discuss some existing works that are directly related to our results. 
A rather comprehensive study of the model studied here was initiated in \cite{Avena Thomann}, 
where several conjectures, based on numerical computations and some heuristics, were presented.
Our Theorem \ref{Theorem: Drift to the left} answers negatively one of the ``key open questions" asked in (3.8) in \cite{Avena Thomann} 
(for $\beta=1-\alpha$ and $\rho \ne 1/2$, the velocity of the walker is non-zero for all $\gamma > 0$ small enough).
Moreover, two dimensional analogs of our model have been studied in the physics literature. 
In \cite{Basu Maes}, the differential mobility of a tagged particle driven by an external field is shown, by means of numerical computations, 
to undergo a transition from a quasi-static to a fluid regime as $\gamma$ is increased. 
These two regimes are somehow analogous to the ones described by our Theorem \ref{Theorem: Drift to the left} and Theorem \ref{Theorem: Drift to the right} respectively. 
In addition, the same model was studied in \cite{Benichou} as a way to probe the glassy transition in liquids, and the possibility of anomalous fluctuations was alluded.
Our results suggest that no anomalous fluctuation should be observed. 

Finally, in a recent work \cite{Hilario et al}, a law of large numbers and an invariance principle for the fluctuations of a walker were obtained in a set-up close to ours. 
The authors consider indeed a random walk driven by a set of non-interacting particles at equilibrium with density $0 < \rho < + \infty$.
The transition probability of the walker differs if he sits on a vacant site or on a site occupied by at least one particle.
Their results hold then for all $\rho$ large enough, assuming that the limiting velocity is non-zero in the limiting case ``$\rho=\infty$". 
This is thus a situation analogous to the one described by our Theorem \ref{Theorem: Drift to the right}, which proof is moreover based on a similar architecture as their.
Nevertheless, we stress that, even for Theorem \ref{Theorem: Drift to the right}, we developed our strategy independently, and that a closer look at the details shows that many steps cannot be simply taken over.

\subsection{Model}\label{subsec:model}
An environment is a function $(\omega(t,x))_{t\ge 0,x\in \Z}$ with values in the interval $]0,1[$; we refer to $t\in \R_+$ as the (continuous) time, and to $x\in\Z$ as the position. 
Given such a function $\omega$ we define, for any space-time point $(n,x) \in \N \times \Z$, the Markovian (discrete time) law $P^{\omega}_{n,x}$ by
\begin{align}
&P^{\omega}_{n,x}(X_0=x)=1 \; \textrm{and, for all }\; k\geq 0 \; \textrm{and}\; z\in \Z,\\
&P^{\omega}_{n,x}(X_{k+1}=z+1|X_k=z)=\omega(n+k,z),\\
&P^{\omega}_{n,x}(X_{k+1}=z-1|X_k=z)=1-\omega(n+k,z).
\end{align}
Given the discrete time process $(X_n)_{n\in\N}$, we also define, with a slight abuse of notation, the continuous time process $(X_t)_{t\ge 0}:=(X_{\lfloor t \rfloor})_{t\ge 0}$ with $t\in \R_+$.

Consider the simple exclusion process on $\Z$ ,
\begin{equation*}
(\eta(t))_{t\ge 0} \; = \;  (\eta_{x}(t))_{t\ge 0,x\in\Z} \; \in \; \{ 0,1\}^\Z,
\end{equation*}
defined by its generator 
\begin{equation}\label{generator simple exclusion}
\mathcal L f \; = \; \gamma \sum_{x\in \Z} \big( f\circ \sigma_{x,x+1} - f \big),
\end{equation}
where $\gamma>0$ is the jump rate of the particles (the intensity of the process), and where $\sigma_{x,y}$ is defined for any $x,y\in \Z$ by
\begin{equation*}
\sigma_{x,y} (\eta) \; = \; (\dots , \eta_y, \dots , \eta_x, \dots ) \qquad \text{if} \qquad \eta \; =  \; (\dots , \eta_x, \dots , \eta_y, \dots).
\end{equation*}
If $\eta_x(t)=1$, the site $x$ is said to be occupied by a particle at time $t$, while it is said to be vacant if $\eta_x (t)=0$.  
For any $0 \le \rho \le 1$, the probability measure $\pi_\rho=(\rho\delta_1+(1-\rho)\delta_0)^{\otimes \Z}$  is invariant for the simple exclusion process. 
Remark that for any $x$, $\pi_\rho(\eta_x=1)=\rho$. We denote by $\P_\rho$ the law of the process with initial condition distributed according to $\pi_\rho$. 

Let $0\leq \alpha, \beta\leq 1$. Given a realization of $\eta$ we define an environment $\omega$ as a function of $\eta$ by
\begin{align}\label{defomega}
\omega(\eta) (t,x)= 
\begin{cases} 
\alpha &\mbox{if } \eta_x(t)=1, \\
\beta & \mbox{if } \eta_x(t)=0 .
\end{cases}
\end{align}
We define $\P$, a law on the space of environments, as the push-forward of $\P_\rho$ through this function. 
To fix the ideas, we will from now on assume that the drift corresponding to empty sites is larger than the drift corresponding to occupied sites:
\begin{equation*}
0\leq \alpha < \beta \leq 1.
\end{equation*}
We define for any space-time point $(n,x)\in \N\times \Z$ the {annealed} law by
\begin{equation}
P_{n,x}=\P\times P^{\omega}_{n,x}.
\end{equation}
Remark that the law of $(X_k-x)_{k\geq 0}$ under $P_{n,x}$ is the same as the law of $(X_k)_{k\geq 0}$ under $P_{0,0}$, that we will often denote simply by $P_0$. 



\subsection{Results}
We assume that the environment is {elliptic}: 
\begin{equation}\label{ellipticity}
0 \; < \alpha , \beta < \; 1. 
\end{equation} 

For our first result, we assume that the walker would drift to the right if, at each of his step, the environment would be entirely refreshed according to the equilibrium measure (the limiting case ``$\gamma = \infty$"). 
We assume thus that $\rho$, $\alpha$ and $\beta$ are such that 
\begin{equation}\label{drift condition}
E_0(X_1) \; = \; \E(2 \omega - 1) \; = \; \rho (2\alpha-1) + (1-\rho) (2\beta-1) \; > \; 0,
\end{equation}
where $E_0$ is the expectation with respect to $P_0$ and $\E$ the expectation with respect to $\P$.
We show that the drift to the right will still be observed if, instead of refreshing the environment at each step of the walker, 
it evolves according to the dynamics generated by \eqref{generator simple exclusion} with initial condition distributed as $\P_{\rho}$, provided that $\gamma$ is taken large enough once the parameters $\rho$, $\alpha$ and $\beta$ have been fixed.
The following theorem may therefore be seen as a perturbative result around the trivial  case ``$\gamma = \infty$":
\begin{Theorem}\label{Theorem: Drift to the right}
Assume that the ellipticity condition \eqref{ellipticity} and the drift condition \eqref{drift condition} hold. There exists $v_*>0$ so that, for $\gamma$ large enough,
\begin{enumerate}
\item  there exists $v(\gamma) \ge v_*$ so that 
\begin{equation*}
\lim_{t\to \infty} \frac{X_t}{t} \; = \; v (\gamma) \qquad P_0-a.s.,
\end{equation*}
\item the annealed central limit theorem holds: under $P_0$
\begin{equation*}
\left(\frac{X_{nt}-ntv(\gamma)}{\sqrt{n}}\right)_{t\geq 0}\implies (B_{t})_{t\geq 0},
\end{equation*}
where $(B_{t})_{t\geq 0}$ is a non-degenerate Brownian motion, and where the convergence in law is in the Skorohod topology.
\end{enumerate}
\end{Theorem}

For our second result, we assume that the walker is transient to the right in a static environment ($\gamma=0$), i.e.\@ that the parameters $\rho$, $\alpha$ and $\beta$ are such that 
\begin{equation}\label{static transience}
\E\left(\ln\frac{1-\omega}{\omega}\right) \; = \; \rho \ln \frac{1-\alpha}{\alpha} + (1-\rho) \ln \frac{1-\beta}{\beta} \; < \;  0.
\end{equation}
For background about static random walks in random environments, see e.g.\@ \cite{Zeitouni}.
We show that, for strictly positive and small enough $\gamma$, the walker has a positive velocity: 
\begin{Theorem}\label{Theorem: Drift to the left}
Assume that the ellipticity condition \eqref{ellipticity}, as well as the condition \eqref{static transience} for transience to the right in a static environment, hold. 
Then, for $\gamma>0$ small enough,
\begin{enumerate}
\item  there exists $v(\gamma)>0$ such that
\begin{equation*}
\lim_{t\to \infty} \frac{X_t}{t} \; = \; v (\gamma) \qquad P_0-a.s.,
\end{equation*}
\item the annealed central limit theorem holds: under $P_0$
\begin{equation*}
\left(\frac{X_{nt}-ntv(\gamma)}{\sqrt{n}}\right)_{t\geq 0}\implies (B_{t})_{t\geq 0},
\end{equation*}
where $(B_{t})_{t\geq 0}$ is a non-degenerate Brownian motion, and where the convergence in law is in the Skorohod topology.
\end{enumerate}
\end{Theorem}

We stress that in the case where \eqref{drift condition} holds while the left hand side of \eqref{static transience} is strictly positive,
the asymptotic velocity of the walker can be either positive or negative according of the value of $\gamma$. 
We can thus exclude that any of our two theorems can be valid as such for all strictly positive $\gamma$.

\subsection{Outline of the paper}
The rest of the article is organized in four sections. 
In Section \ref{Section: Dissipation of traps}, we control the time of dissipation of zones with high density of particles. 
In Section \ref{Section: Drift to the right}, we use these results on the environment together with a renormalization procedure to derive that, if the assumption \eqref{drift condition} holds and if $\gamma$ is large enough, the walk is ballistic to the right.
In Section \ref{Section: small gamma}, the same ideas are used to prove that the walk is ballistic to the right if the assumption \eqref{static transience} holds and if $\gamma >0$ is small enough. 
Finally, in Section \ref{Section: Proof of law of large numbers}, we build a renewal structure to show that the bounds obtained in Sections \ref{Section: Drift to the right} and \ref{Section: small gamma} imply
the law of large numbers and the annealed invariance principle stated in Theorems \ref{Theorem: Drift to the right} and \ref{Theorem: Drift to the left}.

\section{Dissipation of traps}\label{Section: Dissipation of traps}

The walker can be slowed in the places where the concentration of particles is too high with respect to the expected density $\rho$.  
These locally anomalous configurations of the environment are called traps. 
Here we make precise the idea that traps disappear on diffusive time scales. Concretely, we establish that, if in a box of size $L$ around a point $x$, the density of particles is very close to the density $\rho$ for a given initial profile $\eta\in \{0,1\}^\Z$, 
then, waiting a time bigger than $L^2$, the density in smaller boxes around $x$ becomes close to $\rho$ as well, with high probability with respect to the evolution of the process. 
The section is divided into three parts: 
we first state our results, then show some technical lemmas, and finally give the proof of our propositions.  
The technical estimates in Section \ref{subsection: Some lemmas}
are very close to some results obtained in \cite{dos_Santos} (see Lemma 5.3 there). 

Before starting, let us introduce an extra notation, to be in use mainly in Sections \ref{Section: Dissipation of traps}-\ref{Section: small gamma}.
For $\eta \in \{0,1\}^\Z$, we define the law $\P^\eta$ of the simple exclusion process, defined by \eqref{generator simple exclusion}, for the deterministic initial condition $\eta$. 
We define then
\begin{equation}
P^{\eta}=\P^{\eta}\times P^{\omega}_{0,0},
\end{equation}
where $\omega$ is the environment built from $(\eta(t,x))_{t\geq 0,x\in\Z}$ (see \eqref{defomega}).

\subsection{Statement of the results}
Given $x\in \Z$, $L \in \N$ and $\eta \in \N^\Z$, let
\begin{equation}\label{density in a box}
\langle \eta \rangle_{x,L} \; = \; \frac{1}{|\mathrm B(x,L)|} \sum_{y\in \mathrm B(x,L)} \eta_y 
\end{equation}
be the empirical density of particles in a box of radius $L$ around the point $x$. 
In \eqref{density in a box}, we have used the notations
\begin{equation*}
\mathrm B(x,L) \; = \; [x-L,x+L] \cap \Z \qquad  \text{and}  \qquad |\mathrm B(x,L)| \; = \; 2L+1.
\end{equation*}
Let $(\epsilon_L)_{L\ge 0}$ be some decreasing sequence of numbers in $] 0,1 ]$.
The numbers $\epsilon_L$ will serve to control the difference between the density $\rho$ and the empirical density in a box of size $L$.
Given $\eta\in \{ 0,1 \}^\Z$ and $L\in \N$, we define the set of {good} sites $G(\eta,L) \subset \Z$ as follows: 
we say that $x \in G(\eta,L)$ if
\begin{equation}\label{definition good set}
\langle\eta\rangle_{x,L'} \; \le \; (1 + \epsilon_L) \rho \qquad \text{for all} \qquad L' \ge L.  
\end{equation}

The main result of this section is contained in the following proposition, where we use the assumption $\gamma t \ge L^3$ instead of the more natural assumption $\gamma t \ge C L^2$ for some large constant $C$, in order to avoid the introduction of too many constants. 
\begin{Proposition}\label{Proposition: good set at a later time}
There exist some constants $\Const < + \infty$ and $c> 0$ such that,
given an initial profile $\eta \in \{ 0,1 \}^\Z$, and given $x\in\Z$, $t\ge 0$ and $J,L\in \N$, 
the conditions 
\begin{equation*}
x\in G\big(\eta,L\big), \qquad \gamma t\; \ge \; L^{3}, \qquad J \; \le \; L, \qquad  L (\epsilon_J - \epsilon_L) \; \ge \; \Const,
\end{equation*}
imply that 
\begin{equation}\label{result Proposition good set at a later time}
P^{\eta} \Big(x \,\notin\, G \big(\eta (t),J \big) \Big) 
\; \le \;  
\Const \, \frac{\exp \big( - c J \rho^2 (\epsilon_J - \epsilon_L)^2  \big)}{\rho^2 (\epsilon_J - \epsilon_L)^2}.
\end{equation}
\end{Proposition}

\noindent
This proposition can only be applied if, given a profile $\eta$, one waits a time $\gamma t \ge L^3$. 
The next proposition furnishes a control that holds for short times too, 
\begin{Proposition}\label{Proposition: conservation of good sets}
There exist a constant $c> 0$ such that,
given an initial profile $\eta \in \{ 0,1 \}^\Z$, and given $\epsilon > 0$, $t\ge 0$, $x\in\Z$ and $L\in \N$, 
it holds that, if
\begin{equation*}
y\in G\big(\eta,L\big) \quad \forall y \in \mathrm B(x,L) \qquad \text{and} \qquad  \epsilon \ge \epsilon_L,
\end{equation*}
then
\begin{equation*}
P^{\eta} \Big( \langle \eta(t) \rangle_{x,L} \ge (1+\epsilon)\rho \Big) 
\; \le \;  
\exp \big( - c L \rho^2 (\epsilon - \epsilon_L)^2  \big).
\end{equation*}
\end{Proposition}

\noindent
One observes that $\gamma$ plays no role in Proposition \ref{Proposition: conservation of good sets}.

\subsection{Some lemmas: heat equation properties and concentration}\label{subsection: Some lemmas}
We let 
\begin{equation}\label{mean evolution of the field}
\overline{\eta}(t) \; = \; E^\eta \big(\eta(t)\big)
\end{equation}
be the mean value of the field after a time $t$, starting from the initial field $\eta$, where we have used $E^{\eta}$ for the expectation with respect to $P^\eta$.
The mean evolution $\overline{\eta}(t)$ solves the discrete heat equation $\partial_t \overline{\eta} = \gamma \Delta \overline{\eta}$ with initial condition $\overline{\eta}(0) = \eta$.
The operator $\Delta$ appearing here is the discrete Laplacian defined by $\Delta u (x) = u(x+1) - 2 u(x) + u(x-1)$ for $u:\Z\rightarrow \R$.

We find it convenient to introduce three closely related kernels.
Let first $p:\R_+\times \Z \rightarrow \R$ be the heat kernel associated to the Laplacian $\gamma\Delta$: 
$p$ solves the initial value problem 
\begin{equation}\label{definition of p kernel}
p(0,\cdot)= \delta_0(\cdot), \qquad  \partial_t p = \gamma\Delta p.
\end{equation}
So, for $x\in \Z$ and $t\ge 0$, $p(t,x)$ represents the probability that a free particle jumping with rate $\gamma$ starting at origin sits on site $x$ at time $t$. 
Given $L\in \N$, let then $\mathbf{p}_L$ and $p_L$ be given by
\begin{equation}\label{definition of p L kernels}
\mathbf{p}_L (t,x) \; = \; \sum_{y\in \mathrm B(0,L)} p(t,x+y) \qquad \text{and} \qquad p_L(t,x) \; = \; \frac{1}{|\mathrm B(0,L)|}\mathbf{p}_L(t,x) 
\end{equation}
for $(t,x) \in \R_+\times\Z$.
The quantity $\mathbf{p}_L(t,x)$ represents the probability that a free particle starting from the origin lies in the box of size $L$ centered at $x$ at time $t$.

Our first lemma furnishes a concentration bound: with high probability, 
the empirical density of $\eta(t)$ in a box of size $L$ does not deviate too much from the empirical density of the mean evolution $\overline{\eta}(t)$ in the same box. 
This result does not depend on $\gamma$. 
\begin{Lemma}\label{Lemma: large deviations}
There exists a constant $\const > 0$ such that, 
given $\eta\in \{ 0,1 \}^\Z$, $L\in \N$, $x \in \Z$, $t\ge 0$ and $a \ge 0$,
\begin{equation*}
P^{\eta} \big(\langle \eta (t) - \overline{\eta}(t) \rangle_{x,L}  \ge a \big) 
\; \le \; 
\ed^{-\const a^2 L}.
\end{equation*} 
\end{Lemma}

\begin{proof}
Given $\delta \ge 0$, it follows from Markov's inequality that
\begin{equation}\label{first thing in proof concentration inequality ssep}
P^{\eta} \big(\langle \eta (t) - \overline{\eta}(t) \rangle_{x,L}  \ge a \big) 
\; \le \;
\ed^{-\delta a} \cdot \ed^{-\delta \langle \overline{\eta}(t) \rangle_{x,L}}  \cdot E^\eta \Big( \ed^{\delta \langle \eta (t) \rangle_{x,L} } \Big).
\end{equation}
Let us first give an expression for $ \langle \overline\eta (t) \rangle_{x,L}$.
{For $x\in \Z$,  $\overline{\eta}_x(t)=P^\eta\left(\eta_x(t)=1\right)$. 
Note that $\{\eta_x(t)=1\}=\bigcup_{z\in\Z,\eta_z=1} A_z(t,x)$, where $A_z(t,x)$ is the event that a particle initially at $z$, sits at $x$ at time $t$. 
As these events are mutually disjoint due to the exclusion constraint, $\overline{\eta}_x(t)= \sum_{z\in \Z} p(t,x-z) \eta_z$, and we compute}
\begin{align}
\langle \overline{\eta}(t) \rangle_{x,L} 
\; &= \; 
\frac{1}{2L+1} \sum_{y\in \mathrm B(0,L)}\overline{\eta}_{x+y}(t)
\; = \; 
\frac{1}{2L+1} \sum_{y\in \mathrm B(0,L)} \sum_{z\in \Z} p(t,x+y-z) \eta_z
\nonumber\\
\; &= \; 
\sum_{z\in \Z} p_L(t,x-z) \eta_z,
\label{mean in terms of kernels}
\end{align}
as is seen from the definition \eqref{definition of p L kernels} of $p_L$.

Let us then work out the third factor in \eqref{first thing in proof concentration inequality ssep}. 
We will use Liggett's inequality to get rid of the exclusion constraint (see Proposition $1.7$ p.$366$ in \cite{Liggett}). 
Let us define the process $\theta$ that represents the collective motion of independent particles evolving on $\Z$.
So, let $\theta = (\theta(t))_{t\ge 0}$ be the process on $\N^\Z$ defined by the generator
\begin{equation*}
\overline{\mathcal L} f (\theta) \; = \; \sum_{x\in \Z} \theta_x \big( f (\theta^{x,x+1}) - 2 f (\theta) + f (\theta^{x,x-1}) \big),
\end{equation*}
with $\theta^{x,y} = (\dots , \theta_x - 1, \dots , \theta_y +1\dots)$ if $\theta = (\dots , \theta_x, \dots , \theta_y, \dots)$.
We assume that $\theta(0) = \eta$.
It is convenient to adopt the following interpretation: 
we say that there are $n\in \N$ particles at $x\in \Z$ at time $t\ge 0$ if and only if $\theta_x(t) = n$. 
Let us label all the particles, in an arbitrary way, by $k\in\N^*$. 
Let $X_k(t)$ be their position at time $t$.
At any time $t\ge 0$, the variables $(X_k(t))_{k\ge 1}$ are independent. 

While Liggett's inequality is stated for a finite number of particles, it is possible to use it here for the whole infinite system. 
Indeed, we can use it first for any truncated initial condition $\eta^N$, $N\geq 1$, where $\eta^N_i=0$ if $|i|>N$ and $\eta^N_i=\eta_i$ if $|i|\leq N$. 
We then couple these various initial conditions using the Harris graphical construction (see \cite{Harris}),
and finally use the monotone convergence theorem to extend Liggett's comparaison inequality to the whole system starting from $\eta$ (see for example \cite{Arratia} for a similar generalization). 
Remembering the definition \eqref{definition of p L kernels} of $\mathbf p_L$, we get
\begin{align}
E^\eta \Big( \ed^{\delta \langle \eta (t) \rangle_{x,L} } \Big)
\; \le \; 
E^\eta \Big( \ed^{\delta \langle \theta (t) \rangle_{x,L} } \Big)
\; &= \; 
E^\eta \Big( \ed^{\frac{\delta}{2L+1}\sum_{k\ge 1} \mathbf 1_{\mathrm B(x,L)}(X_k (t))} \Big)
\; = \; 
\prod_{k\ge 1} E^\eta \Big( \ed^{\frac{\delta}{2L+1}\mathbf 1_{\mathrm B(x,L)}(X_k (t))} \Big) \nonumber\\
\; &= \; 
\prod_{k\ge 1} \Big( \ed^{\frac{\delta}{2L+1}} \mathbf{p}_L(t,X_k(0) - x) + \big( 1 - \mathbf{p}_L(t,X_k(0) - x) \big) \Big) \nonumber\\
\; &= \; 
\prod_{z\in \Z} \Big(  \ed^{\frac{\delta}{2L+1}}  \mathbf{p}_L(t,z - x) + \big(1 - \mathbf{p}_L(t,z - x) \big) \Big)^{\eta_z} \nonumber\\
\; &\le \; 
\exp \Big( \big( \ed^{\delta/(2L+1)} - 1 \big) \sum_{z\in \Z} \mathbf{p}_L(t,z-x) \eta_z \Big)\nonumber\\
\; &= \; 
\exp \Big(  \big( \ed^{\delta/(2L+1)} - 1\big) (2L+1) \langle \overline{\eta}(t)\rangle_{x,L}  \Big),
\label{brol final concentration inequ Liggett}
\end{align}  
where the last expression follows from \eqref{mean in terms of kernels}.

Let us now come back to \eqref{first thing in proof concentration inequality ssep}. 
Assuming that $\delta/(2L+1) \le 1$, 
we conclude, using \eqref{brol final concentration inequ Liggett} and expanding $\ed^{\delta/(2L+1)} - 1$ in first order in $\delta/(2L+1)$, that we can find a constant $\Const < + \infty$ such that
\begin{equation*}
\ed^{-\delta \langle \overline{\eta}(t) \rangle_{x,L}}  \cdot E^\eta \Big( \ed^{\delta \langle \eta (t) \rangle_{x,L} } \Big)
\; \le \; 
\ed^{\frac{\Const \delta^2}{2L+1} \langle \overline \eta(t) \rangle_{x,L}}.
\end{equation*}
Therefore 
\begin{equation*}
P^{\eta} \big(\langle \eta (t) - \overline{\eta}(t) \rangle_{x,L}  \ge a \big) 
\; \le \;
\ed^{-a \delta + \frac{\Const \langle \overline{\eta}(t)\rangle_{x,L}}{2L+1} \delta^2} \qquad \text{for} \qquad 0 \,  \le \,  \delta  \, \le \,  2L+1.  
\end{equation*}
This inequality is optimized for $\delta = \min \{ \frac{a (2L+1)}{2 \Const \langle \overline \eta (t) \rangle_{x,L}} , 2L+ 1 \}$.
If $\delta =  \frac{a (2L+1)}{2 \Const \langle \overline \eta (t) \rangle_{x,L}}$, we find 
\begin{equation}\label{proof large deviation case 1}
P^{\eta} \big(\langle \eta (t) - \overline{\eta}(t) \rangle_{x,L}  \ge a \big) 
\; \le \;
\ed^{- \frac{a^2 (2L+1)}{4 \Const \langle \overline{\eta}(t)\rangle_{x,L}}}
\; \le \; 
\ed^{- \frac{a^2 (2L+1)}{4 \Const}}, 
\end{equation}
as $ \langle \overline{\eta} (t)\rangle_{x,L} \le 1$ for the simple exclusion process. 
If instead $\delta = 2L+1$, we obtain
\begin{equation*}
P^{\eta} \big(\langle \eta (t) - \overline{\eta}(t) \rangle_{x,L}  \ge a \big) 
\; \le \;
\ed^{- a (2L+1)( 1 - \frac{\Const \langle \overline{\eta} (t) \rangle_{x,L}}{a})}.
\end{equation*}
Because in this case $2L+1 \le  \frac{a (2L+1)}{2 \Const \langle \overline \eta (t) \rangle_{x,L}} $, this implies 
\begin{equation}\label{proof large deviation case 2}
P^{\eta} \big(\langle \eta (t) - \overline{\eta}(t) \rangle_{x,L}  \ge a \big) 
\; \le \;
\ed^{- a (2L+1)/2}.
\end{equation}
The Lemma is obvious for $a > 1$ and, for $a \le 1$, \eqref{proof large deviation case 1} is always larger than \eqref{proof large deviation case 2} as soon as $\Const \ge 1/2$. 
This gives the claim. 
\end{proof}

The next two lemmas furnish a control on the solution of the heat equation. 
The first of these makes precise the fact that, after a time $t$, the solution at $x$ is well approximated by the empirical density of the initial profile in a box of size $(\gamma t)^{1/2}$ around $x$.
\begin{Lemma}\label{Lemma: estimate heat equation}
There exists a constant $\Const < + \infty$ such that, 
given $\eta \in \{ 0,1\}^\Z$, $M,L \in \N$, $x\in \Z$ and $t \ge 0$, 
\begin{equation*}
\langle \overline{\eta}(t) \rangle_{x,M} \; \le \; \bigg( 1 + \frac{\Const L^2}{ \gamma t} \bigg) \sup \big\{ \langle \eta \rangle_{x,r }: r \ge L \big\}.
\end{equation*}
\end{Lemma}

\begin{proof}
Let us start by quoting a property of the heat equation. 
Let $f:\Z\mapsto \R$ be such that $f(x)=f(-x)$ for all $x\in \Z$, and such that, for all $x\ge 0$, it holds that $f(x) - f(x+1) \ge 0$. 
For $x\in \Z$ and $t\ge 0$, let also $f(t,x)$ be the solution of the initial value problem $f(0,x) = f(x)$ and $\partial_t f(t,x) = \gamma \Delta f(t,x)$. 
It is verified that, for all $t> 0$, we still have $f(t,x) = f(t,-x)$ for all $x\in \Z$, and $f(t,x) - f(t,x+1) \ge 0$ for all $x\ge 0$. 

This has the following consequence. 
It is seen from the definition \eqref{definition of p L kernels} of $p_M$ that $p_M(0,x) = p_M(0,-x)$ for all $x\in \Z$, that $p_M(0,x) - p_M(0,x+1)\ge 0$ for all $x\ge 0$, 
and that $p_M$ solves the heat equation $\partial_t p_M(t,x) = \gamma \Delta p_M(t,x)$ for $x\in \Z$ and $t\ge 0$.  
Therefore, for any function $u: \Z \rightarrow \R$, it holds that 
\begin{equation}\label{plateau decomposition of P kernel}
\sum_{z\in \Z} p_M(t,z-x) u(z) \; = \; \sum_{z\ge 0} \big( p_M(t,z) - p_M(t,z+1) \big)  (2z+1) \langle u \rangle_{x,z}, 
\end{equation}
where $p_M(t,z) - p_M(t,z+1) \ge 0$ for all $z\ge 0$. 

Let us now write for simplicity 
\begin{equation*}
R := \sup \big\{ \langle \eta \rangle_{x,r } : r \ge L \big\}.
\end{equation*}
As shown in \eqref{mean in terms of kernels}, it holds that $\langle \overline{\eta}(t) \rangle_{x,M} = \sum_{z\in\Z} p_M(t,z-x) \eta_z$,
so that, using \eqref{plateau decomposition of P kernel}, we find
\begin{equation}\label{first expression for the mean}
\langle \overline{\eta}(t) \rangle_{x,M} 
\; = \; 
\langle \overline{\eta}(t) \rangle_{x,M}  - R + R
\; = \;
 \sum_{z\ge 0} \big( p_M(t,z) - p_M(t,z+1) \big)  (2z+1) \big(\langle \eta \rangle_{x,z} - R\big) \; + \; R.
\end{equation}
If $z\ge L$, then 
\begin{equation*}
\langle \eta \rangle_{x,z} - R \; = \; \langle \eta \rangle_{x,z} - \sup \big\{ \langle \eta \rangle_{x,r } : r \ge L \big\} \; \le \;  0,
\end{equation*}
while, if $z < L$, then still
\begin{equation*}
(2z + 1) \big( \langle \eta \rangle_{x,z} - R \big)
\; \le \; 
(2z+1)  \langle \eta \rangle_{x,z}
\; = \; 
\sum_{y\in \mathrm B(0,z)} \eta_{x+y}
\;\le\;  
( 2 L + 1) R.
\end{equation*}
We see that we already obtain the result if $L=0$, so that we can further assume $L\ge 1$. 
Inserting these two estimates in \eqref{first expression for the mean}, we find
\begin{equation}\label{an expression for the mean at time t}
\langle \overline{\eta}(t) \rangle_{x,L} 
\; \le \; 
(2L+1)R\sum_{z=0}^{ L-1} \big( p_M(t,z) - p_M(t,z+1) \big) \; + \; R 
\; = \; 
\Big\{ 1 + (2L+1) \big( p_M(t,0) - p_M(t,L) \big) \Big\} R.
\end{equation}

For any $z\in \Z$ and $t>0$, expanding $p(t,z) = \ed^{\gamma \Delta t} \delta_0 (z)$ in the Fourier variables, and writing $\omega (\xi) = 2 (1 - \cos \xi)$, we find
$$ 
|p(t,z) - p(t,z+1)| \; = \; \Big| \int_{-\pi}^\pi \ed^{-\omega(\xi) \gamma t} \big(\ed^{i\xi z} - \ed^{i\xi (z+1)} \big) \, \frac{\dd \xi}{2\pi} \Big| \;\le\; \Const \int_\R \ed^{-c\xi^2 \gamma t} |\xi| \, \dd \xi \; \le \;  \frac{\Const'}{\gamma t},
$$
for some $c> 0$ and $\Const < \Const' < + \infty $.
Therefore
\begin{equation*}
p_M(t,0) - p_M(t,L) \; \le \; \frac{\Const L}{\gamma t}.
\end{equation*}
Inserting this estimate in \eqref{an expression for the mean at time t} furnishes the claim. 
\end{proof}

Our last lemma gives a bound that does not require to wait a time $\gamma t \ge L^3$ to hold:
\begin{Lemma}\label{Lemma: heat equation short times}
Given $\eta \in \{ 0,1\}^\Z$, $x\in \Z$, $t\ge 0$ and $L\in \N$, it holds that 
\begin{equation*}
\langle \overline{\eta}(t) \rangle_{x,L} \; \le \; \sup \big\{ \langle \eta \rangle_{y,r} : r\ge L, y \in \mathrm B(x,L) \big\}.
\end{equation*}
\end{Lemma}

\begin{proof}
For $L=0$, the lemma follows from Lemma \ref{Lemma: estimate heat equation}. We assume $L\ge 1$. 
To simplify writings, let us write 
\begin{equation*}
R \; := \; \sup \big\{ \langle \eta \rangle_{y,r} : r\ge L, y \in \mathrm B(x,L) \big\}.
\end{equation*}
We decompose 
\begin{equation*}
p(t,\cdot) \; = \; \tilde{p} (t,\cdot) + \sum_{k\ge L} p_k(t,\cdot),
\end{equation*}
with 
\begin{equation*}
p_k(t,\cdot) \; = \; \big(p(t,k) - p(t,k+1) \big)\mathbb I_{\mathrm B(0,k)}(\cdot), \qquad k \ge L,
\end{equation*}
and 
\begin{equation*}
\tilde p (t,z) \; =\;  p(t,z) - 
\sum_{k\ge L} p_k(t,z) \quad \text{if} \quad z\in \mathrm B(0,L-1) 
\qquad \text{and} \qquad 
\tilde p(t,z) \; = \; 0 \quad \text{otherwise}.
\end{equation*}

Let $y \in \mathrm B(x,L)$.
It holds that
\begin{equation*}
\overline{\eta}_y(t) 
\; = \; 
\sum_{z\in\Z} p(t,y-z) \eta_z
\; = \; 
\sum_{z\in \Z} \tilde{p} (t,y-z) \eta_z + \sum_{k\ge L} \sum_{z\in \Z} p_k(t,y-z) \eta_z.
\end{equation*}
Since, for $k\ge L$, 
\begin{equation*}
\sum_{z\in \Z} p_k(t,y-z) \eta_z 
\; = \; 
\big(p(t,k) - p(t,k+1) \big) \sum_{z\in \mathrm B(0,k)} \eta_{y+z}
\; = \; 
\big(p(t,k) - p(t,k+1) \big) (2k+1) \langle \eta \rangle_{y,k},
\end{equation*}
and since we assume $y \in \mathrm B(x,L)$, our hypotheses imply $\langle \eta \rangle_{y,k} \le R$, so that 
\begin{equation*}
\overline{\eta}_y(t) 
\; \le \; 
\sum_{z\in \mathrm B(0,L-1)} \tilde{p} (t,z) \eta_{z+y} + \sum_{k\ge L} \big(p(t,k) - p(t,k+1) \big) (2k + 1) R.
\end{equation*}

Therefore
\begin{align*}
\langle \overline{\eta}(t) \rangle_{x,L} 
\; &= \; 
\frac{1}{|\mathrm B(x,L)|} \sum_{y\in \mathrm B(x,L)}  \overline{\eta}_y(t) \\
\; &\le \; 
\sum_{z\in \mathrm B(0,L-1)} \tilde{p} (t,z) \frac{1}{|\mathrm B(x,L)|} \sum_{y\in \mathrm B(x,L)}  \eta_{z+y} + \sum_{k\ge L} \big(p(t,k) - p(t,k+1) \big) (2k + 1) R.
\end{align*}
If $z \in \mathrm B(0,L-1)$, we have $x+z \in \mathrm B(x,L)$, and so by hypothesis, 
\begin{equation*}
\frac{1}{|\mathrm B(x,L)|} \sum_{y\in \mathrm B(x,L)}  \eta_{z+y} \; = \; \frac{1}{|\mathrm B(0,L)|} \sum_{w\in \mathrm B(0,L)} \eta_{x+z+w} \; \le \; R.
\end{equation*}
We thus conclude that 
\begin{equation*}
\langle \overline{\eta}(t) \rangle_{x,L} 
\; \le \;
\Big\{ \sum_{z\in \mathrm B(0,L-1)} \tilde{p} (t,z) + \sum_{k\ge L} \big(p(t,k) - p(t,k+1) \big) (2k + 1) \Big\} R 
\; = \; R,
\end{equation*}
which is the claim.
\end{proof}

\subsection{Proof of Propositions \ref{Proposition: good set at a later time} and \ref{Proposition: conservation of good sets}}

\begin{proof}[Proof of Proposition \ref{Proposition: good set at a later time}]
We have
\begin{equation}\label{first estimate in proof good set at a later time}
P^{\eta} \big(x \notin G(\eta(t),J) \big) \; \le \; \sum_{J' \ge J} P^{\eta} \big( \langle \eta (t) \rangle_{x,J'} \, > \, (1+ \epsilon_{J}) \rho \big).
\end{equation} 
By Lemma \ref{Lemma: estimate heat equation}, and since by hypothesis $\gamma t \ge L^3$ and $x\in G(\eta,L)$, we find for all $J'\in \N$ that
\begin{equation*}
\langle \overline{\eta}(t) \rangle_{x,J'} 
\; \le \; 
\big( 1 + \Const /L \big) \sup \big\{ \langle \eta \rangle_{x,r} : r \ge L  \big\}
\; \le \; 
\big( 1 + \Const /L \big) (1 + \epsilon_L) \rho. 
\end{equation*}
Therefore, thanks to the hypothesis that $L (\epsilon_J - \epsilon_L)$ is large enough, we get 
\begin{equation}\label{second technical in proof good set}
(1+ \epsilon_{J}) \rho - \langle \overline{\eta}(t) \rangle_{x,J'} 
\; \ge \; 
\big( \epsilon_{J} - \epsilon_L - 2\Const / L \big) \rho 
\; \ge \; 
\frac{\rho}{2} (\epsilon_J - \epsilon_L) 
\; \ge \;  0.
\end{equation}
This last inequality allows us to use the concentration estimate stated in Lemma \ref{Lemma: large deviations}:
\begin{align}
P^{\eta} \Big( \langle \eta (t) \rangle_{x,J'} \, \ge \, (1+ \epsilon_{J}) \rho \Big)
\; &= \; 
P^{\eta} \Big( \big\langle \eta (t) - \overline{\eta}(t) \big\rangle_{x,J'} \, \ge \, (1+ \epsilon_{J}) \rho - \langle \overline{\eta}(t) \rangle_{x,J'} \Big) 
\nonumber\\
\; &\le \; P^{\eta} \Big( \big\langle \eta (t) - \overline{\eta}(t) \big\rangle_{x,J'} \, \ge \, \frac{\rho}{2} (\epsilon_J - \epsilon_L) \Big) 
\nonumber\\
\; &\le \; \exp \Big( -\const  J' \rho^2 (\epsilon_J - \epsilon_L)^2 \Big),
\label{large deviations in proof good set at a later time}
\end{align}
where we have used \eqref{second technical in proof good set} to get the first inequality. 
Inserting \eqref{large deviations in proof good set at a later time} in \eqref{first estimate in proof good set at a later time}, we obtain
\begin{equation*}
P^{\eta} \big(x \notin G(\eta(t),J) \big) \; \le \; \sum_{J' \ge J} \ed^{- c J' \rho^2 (\epsilon_J - \epsilon_L)^2}
\; \le \;
\Const \, \frac{\ed^ {-c J \rho^2 (\epsilon_J - \epsilon_L)^2}}{\rho^2 (\epsilon_J - \epsilon_L)^2},
\end{equation*}
since $\rho^2 (\epsilon_J - \epsilon_L)^2 \le 1$. This is the claim.
\end{proof}

\begin{proof}[Proof of Proposition \ref{Proposition: conservation of good sets}]
We write
\begin{equation*}
P^{\eta} \big( \langle \eta(t) \rangle_{x,L} \ge (1+\epsilon)\rho \big) 
\; = \;
P^{\eta} \big( \langle \eta(t) \rangle_{x,L} - \langle \overline{\eta}(t) \rangle_{x,L}  \ge (1+\epsilon)\rho - \langle \overline{\eta}(t) \rangle_{x,L} \big).
\end{equation*}
It follows from the hypotheses and from Lemma \ref{Lemma: heat equation short times} that 
\begin{equation*}
(1+\epsilon)\rho - \langle \overline{\eta}(t) \rangle_{x,L}  \; \ge \; (1+\epsilon)\rho - (1 + \epsilon_L) \rho \; = \; (\epsilon - \epsilon_L) \rho, 
\end{equation*}
so that, by Lemma \ref{Lemma: large deviations}, we find indeed $P^{\eta} \big( \langle \eta(t) \rangle_{x,L} \ge (1 + \epsilon)\rho \big) \,\le \, \exp (-c L \rho^2 (\epsilon -\epsilon_L)^2)$.
\end{proof}

\section{Drift for large $\gamma$}\label{Section: Drift to the right}

We here prove 
\begin{Theorem}\label{Theorem: lim inf to the right}
Assume that the drift condition \eqref{drift condition} holds. 
Then, there exists $v_*>0$ so that, for $\gamma$ large enough, there exists $v(\gamma) \ge v_*$ so that 
\begin{equation*}
\liminf_{t\to \infty} \frac{X_t}{t} \; = \; v (\gamma) \qquad P_0-a.s.
\end{equation*}
\end{Theorem}
\noindent
Remark that the ellipticity condition \eqref{ellipticity} is not required. 

Before we proceed to prove Theorem \ref{Theorem: lim inf to the right}, let us fix some parameters.  
First, for the whole section, we assume that $\alpha$, $\beta$ and $\rho$ are chosen so that the drift condition \eqref{drift condition} holds. 
Next, from now on, we assume that the sequence $(\epsilon_L)_{L\ge 0}$ introduced in Section \ref{Section: Dissipation of traps} to control the excess of density in boxes of size $L$, is given by
\begin{equation}\label{epsilon sequence}
\epsilon_L \; = \; \frac{1}{ 1 + \ln (L+1)} \qquad \text{for} \qquad L \in \N.
\end{equation}
With this choice, the sequence $(\epsilon_L)_{L\ge 0}$ satisfies two useful requirements:
First, in view of the definition \eqref{definition good set} of good sites, $\epsilon_L \gg L^{-1/2}$ is needed for a given site to be good with high probability for large $L$; 
this is the case with \eqref{epsilon sequence}, and the probability that $x\notin G(\eta,L)$ decays faster with $L$ than $\ed^{-c_\tau L^\tau}$ for any $\tau < 1$. 
Second, we will apply Proposition \ref{Proposition: good set at a later time} with $J$ and $L$ going together to infinity, in the ratio $L \sim J^2$. 
The bound \eqref{result Proposition good set at a later time} in Proposition \ref{Proposition: good set at a later time} is only meaningful for $(\epsilon_J - \epsilon_L)^2 \gg J^{-1}$; 
again $(\epsilon_L)_{L\ge 0}$ defined by \eqref{epsilon sequence} satisfies this condition, and the bound \eqref{result Proposition good set at a later time} becomes $ \mathcal O (\ed^{-c_\tau J^{\tau}})$ for any $\tau < 1$. 

Finally, we define a sequence $(\phi_L)_{L\ge 0}$, where $\phi_L$ will represent the size of the traps in a box of size $L$. 
Intuitively, typical regions of anomalous density in a box of size $L$ are of size $\ln L$. 
Nevertheless, for our estimates, we found it convenient to overestimate their size; we set 
\begin{equation}\label{phi function}
\phi_L \; = \;  L^{1/100}   \qquad \text{for} \qquad L \in \N.
\end{equation}
In the sequel, we will tacitly use the bound $|X_t - X_s| \le t-s$, valid for all $0 \le s \le t$.

\subsection{Outline of the proof}\label{subsection Outline of the proof large gamma}
The proof of Theorem \ref{Theorem: lim inf to the right} is divided into three steps. 

First, we show that, given any (arbitrarily large) time $T$, there exists $\gamma$ large enough so that the walker drifts to the right, for a large set of initial conditions on the environment.
This is the content of Lemma \ref{Lemma: Initial stpe large gamma} below. 
This is an easy result as we first fix $T$, and then chose $\gamma$ large enough so that the law of $X_T$ is well approximated by the law of a walk in a homogeneous environment, with drift given by \eqref{drift condition}.

Second, we keep $\gamma$ fixed, and we use a renormalization procedure to extend the previous result to arbitrarily long times $t$, and for a set of initial conditions which probability converges quickly to 1 as $t\to \infty$.
This is done in Proposition \ref{Proposition: Renormalization large gamma} below, at the cost of reducing slightly our lower bound on the drift 
($\mathsf v$ in  \eqref{result lemma initial step large gamma} becomes $\mathsf v/2$ in \eqref{content of drifs at a fixed time}).
Our scheme is eventually inspired by a method developed in \cite{Bricmont_Kupiainen_91} (see also \cite{Bricmont_Kupiainen_09} for a dynamical uniformly mixing environment), though our problem requires much less involved estimates.  
As we cannot go at once to arbitrarily long times, let us first consider the special case $t=T'$ with $T'= T(T+1) \sim T^2$, as will be done in the proof of Proposition \ref{Proposition: Renormalization large gamma} 
(the choice $T'\sim T^2$ is arbitrary to a large extend). 
Once this will be understood, the procedure will just be iterated to reach all times $T^{2^k}$ for any $k\ge 1$ (it is then not hard to see that all intermediate times can be reached as well).

Over a time $T'$, the walker will evolve in the ball $B(0,T')$ if it starts initially at the origin.  
Therefore, up to a small extra boundary of size $\sqrt{T'}$, we need to control the initial environment in $B(0,T')$. 
In the box $B(0,T')$, we allow for an initial environment $\eta$ such that all traps are of size no larger than $\phi_{T'}$ (see \eqref{phi function}), i.e.\@ such that $B(0,T') \subset G(\eta,\phi_{T'})$ (see \eqref{definition good set}). 
Incidentally, we notice that we already achieved one of our goals: 
the measure of the set of initial configurations that we need to exclude when we observe the walker over a time $T'$, 
is smaller than it was over a time $T$ (this set was the set of $\eta$ such that $B(0,T)\not\subset G(\eta,\phi_{T})$, see Lemma  \ref{Lemma: Initial stpe large gamma} below). 

We then decompose $X_{T'}$ as $X_{T'} = \sum_{k=1}^T (X_{kT} - X_{(k-1)T})$ (see Figure \ref{figure: Renormalization} with $t_n = r = T$ and $t_{n+1}=T'$). 
We would like to use our knowledge on the behavior of the walk over a time $T$ to control the behavior of each increment $X_{kT} - X_{(k-1)T}$. 
This is not possible for the first term $X_T - X_0$ though, since we only know that $B(0,T') \subset G(\eta, \phi_{T'})$, while we would need $B(0,T) \subset G(\eta, \phi_{T})$ (see Lemma \ref{Lemma: Initial stpe large gamma} below);
therefore, we just use the trivial bound $X_T - X_0 \ge -T$ (this is still fine since $T \ll T'$). 
However, with high probability with respect to the evolution of the environment, all points $y\in B(0,T)$ are such that $y\in G(\eta_{kT},\phi_{T})$ for all $1 \le k \le T-1$, 
so that we do have a good control on all the steps $X_{kT} - X_{(k-1)T}$ with $2 \le k \le T$. 
This is actually the main point of the argument; it is a consequence of the relaxation of traps as expressed by Proposition \ref{Proposition: good set at a later time}
and of the choice of suitable sequence $(\epsilon_L)_{L\ge 1}$ (see the comments after \eqref{epsilon sequence}).
Thanks to the quantitative estimate \eqref{result lemma initial step large gamma} in Lemma \ref{Lemma: Initial stpe large gamma} below, 
we deduce then a lower bound on the drift of $X_{T'}$ by reducing a bit the bound $\mathsf v$ in \eqref{result lemma initial step large gamma} (see \eqref{cn sequence}).
The iteration of this procedure to larger scales is straightforward. 

Finally, Theorem \ref{Theorem: lim inf to the right} is established thanks to a Borel-Cantelli type of argument.

\subsection{Initial step}\label{subsection: Initial step large gamma}

In the next Lemma, $T$ is both a time and spatial scale; 
the condition $\gamma \ge \phi_T^3$ and the bound \eqref{result lemma initial step large gamma} turn out to be convenient but are not optimal. 

\begin{Lemma}\label{Lemma: Initial stpe large gamma}
Let $T\in \N$ be large enough, and let $\gamma \ge \phi_T^3$. 
There exists a constant $\mathsf v > 0$ (that depends only on $\alpha$, $\beta$ and $\rho$) such that, given  $\eta\in \{ 0,1\}^\Z$ such that 
\begin{equation*}
\forall y \in \mathrm B(0,T), \qquad y \in G(\eta, \phi_T ),
\end{equation*} 
it holds that 
\begin{equation}\label{result lemma initial step large gamma}
P^{\eta} (X_T \le \mathsf v T ) \; \le \; \ed^{- \phi_T^{1/4}}.
\end{equation}
\end{Lemma}

\begin{proof}
Letting $\mathsf v > 0$ be a constant that we will determine later, we first write 
\begin{equation}\label{before main thing in proof Lemma initial step large gamma}
P^{\eta} (X_{T} \le \mathsf v T ) \;  \le \; P^{\eta} (X_T - X_1 \le \mathsf v T +1 ),
\end{equation}
as the hypotheses do not allow to determine whether the site $0$ is initially occupied or not. 

Let us define a few objects.  
Since we assume $\alpha \le \beta$, 
there exists $\delta > 0$ small enough so that \eqref{drift condition} will still be satisfied with any $\rho' \in [ 0, \rho + \delta]$ instead of $\rho$.
Let $\delta > 0$ be such that this holds. Up to an enlargement of the probability space we define a sequence of i.i.d.\@ random variables $(Y_k)_{k \ge 1}$ that are independent from both the exclusion process and the walker with distribution 
\begin{equation}\label{def of increments initial lemma large gamma}
P^{\eta} ( Y_k = 1) \; = \; \beta + (\alpha - \beta) (\rho + \delta), 
\qquad 
P^{\eta} ( Y_k = -1 ) \; = \; 1 - P^{\eta} ( Y_k = 1).
\end{equation}
Thanks to \eqref{drift condition} and our choice of $\delta$, it holds that $P^{\eta} ( Y_k = 1)  > 1/2$.
For $m\ge 2$, we define the events 
\begin{equation}\label{definition of J in initial step large gamma}
\mathcal E_{m-1} \; = \; \{ y \in G(\eta (m-1),J), \forall y \in \mathrm B(0,m) \}
\qquad \text{with}\qquad 
J = \lfloor \phi_T^{1/2} \rfloor,
\end{equation}
and $\mathcal E_0 = \{ y \in G(\eta (0),\phi_T), \forall y \in \mathrm B(0,1) \}$. 
By hypothesis, $P^{\eta} (\mathcal E_0) = 1$. 

We aim to show that 
\begin{equation}\label{main thing in proof Lemma initial step large gamma}
P^{\eta} (X_T - X_1 \le \mathsf v T +1 ) 
\; \le \;
P^{\eta} \Big( \sum_{k=1}^{T-1} Y_k \le \mathsf v T + 1 \Big) + \sum_{k=1}^{T-1} P^{\eta} (\mathcal E^c_{k-1}).
\end{equation}
Before deriving this expression, let us see that it implies the lemma for $\mathsf v$ small enough depending on the values of $\alpha$, $\beta$ and $\rho$.
Indeed for such a $\mathsf v$, by \eqref{def of increments initial lemma large gamma}, the first term in the right hand side of  \eqref{main thing in proof Lemma initial step large gamma} is seen to be bounded as 
\begin{equation}\label{first bound initialisation}
P^{\eta} \Big( \sum_{k=1}^{T-1} Y_k \le \mathsf v T + 1 \Big) \; \le \; \ed^{- \const T}
\end{equation}
for some constant $\const >0$.
The second term in the right hand side of \eqref{main thing in proof Lemma initial step large gamma} is then bounded by means of Proposition \ref{Proposition: good set at a later time}. 
Since $T$ is assumed to be large enough, and since $\gamma \ge \phi_T^3$, we deduce that 
\begin{align}
\sum_{k=1}^{T-1} P^{\eta} (\mathcal E^c_{k-1})
\; &\le \; \sum_{\stackrel{y\in \mathrm B({\rd 0},T),}{ s\in \{1, \dots , T-2 \} }}  P^{\eta} \big( y \notin G(\eta(s),J) \big) 
\; \le \; 
\Const \, \sum_{y,s} \frac{\exp \big( -c J \rho^2  (\epsilon_{J} - \epsilon_{\phi_{T}})^2 \big) }{\rho^2  (\epsilon_{J} - \epsilon_{\phi_{T}})^2 } 
\nonumber\\
\; &\le \; 
\Const \, T^2 \,  \frac{\exp \big( -c J \rho^2  (\epsilon_{J} - \epsilon_{\phi_{T}})^2 \big) }{\rho^2  (\epsilon_{J} - \epsilon_{\phi_{T}})^2 }
\; \le \;
\ed^{-\phi_T^{1/3}} ,
\label{second bound initialisation}
\end{align} 
where, to get the last inequality, we have used the explicit expressions \eqref{epsilon sequence} and \eqref{phi function} and \eqref{definition of J in initial step large gamma}, 
as well as the fact that $T$ is large enough.
We obtain the lemma by inserting the bounds \eqref{first bound initialisation} and \eqref{second bound initialisation} in \eqref{main thing in proof Lemma initial step large gamma}, 
and then \eqref{main thing in proof Lemma initial step large gamma} in \eqref{before main thing in proof Lemma initial step large gamma}. 

We are thus left with the proof of \eqref{main thing in proof Lemma initial step large gamma}. 
For this, we show that, for any $m\ge 1$ and any $a\in \R$, 
\begin{equation}\label{rewriting thing to show initial lemma large gamma}
P^{\eta} \big( X_{1+m} - X_1 \le a \big) 
\; \le \;
P^{\eta} \big( X_{1 + (m-1)} + Y_m - X_1 \le a \big) \, +  \, P^{\eta} (\mathcal E_{m-1}^c), 
\end{equation}
from which \eqref{main thing in proof Lemma initial step large gamma} follows by iteration using Fubini. 

Let us first deal with the case $m=1$. 
We need to show that $P^{\eta} (X_2 - X_1 \le a) \le P^{\eta} (Y_1 \le a)$, but for this it is enough to establish that $P^{\eta} (X_2 - X_1 = -1) \le P^{\eta} (Y_1 = -1)$.
It holds that 
\begin{align*}
P^{\eta} (X_2 - X_1 = -1)
\; &= \; 
\sum_{z\in \Z} \sum_{\sigma = 0,1} P^{\eta} (X_2 - X_1 = -1 | X_1 = z, \eta_z(1) = \sigma) P^{\eta} (X_1 = z,\eta_z(1)=\sigma)\\
\; &= \;
\sum_{z\in \Z} P^{\eta} (X_1 = z)\sum_{\sigma = 0,1} P^{\eta} (X_2 - X_1 = -1 | X_1 = z, \eta_z(1) = \sigma) P^{\eta} (\eta_z(1)=\sigma)\\
\; &= \;
\sum_{z\in \mathrm B({0},1)} P^{\eta} (X_1 = z)
\Big(
1 - \beta + (\beta - \alpha) P^{\eta} (\eta_z (1) = 1)
\Big).
\end{align*}
Since, by \eqref{def of increments initial lemma large gamma}, we have $P^{\eta} (Y_1 = -1) = 1 - \beta + (\beta - \alpha)(\rho + \delta)$, 
we will conclude by showing that $P^{\eta} (\eta_z (1) = 1) \le \rho + \delta$ for any $z\in \mathrm B({0},1)$. 
Let $z\in \mathrm B({0},1)$.
It holds that 
\begin{equation}\label{initial step large gamma local probability}
P^{\eta}\big( \eta_z(1) = 1 \big)
\; = \; \sum_{w \in \Z} \eta_w p(1,w-z) \; = \; \overline{\eta}_z(1),
\end{equation}
with $\overline\eta(t)$ defined in \eqref{mean evolution of the field}.
We can use Lemma \ref{Lemma: estimate heat equation}, with $M=0$, $L=\phi_T$, $t=1$ and $x=z$, to estimate the right hand side of \eqref{initial step large gamma local probability}.
Since $\gamma \ge \phi_T^3$ and $z\in G(\eta,\phi_T)$ by hypothesis, we conclude that if $T$ is large enough,
\begin{equation}\label{concrete estimate for initial step large gamma}
P^{\eta} \big( \eta_z(1) = 1\big) 
\; \le \; \Big( 1 + \frac{\Const \phi_T^2}{\gamma} \Big) \sup \{ \langle \eta \rangle_{z,r} : r\ge \phi_T \} 
\; \le \; \Big( 1 + \frac{\Const \phi_T^2}{\phi_T^3} \Big) (1+ \epsilon_{\phi_T}) \rho\\
\; < \; \rho +  \delta.
\end{equation}

Let us next consider the case $m>1$ in \eqref{rewriting thing to show initial lemma large gamma}:
\begin{multline*}
P^{\eta} (X_{1+m} - X_1 \le a)
\; = \; 
P^{\eta} (X_{1 +m} - X_m + X_m - X_1 \le a)\\
\; \le \;
P^{\eta} (\mathcal{E}_{m-1}^c)+\sum_{z_1,z_m \in \Z} 
 P^{\eta} \big( X_{1+m} - X_m \le a - (z_m-z_1) \big| X_1=z_1,X_m=z_m,\mathcal{E}_{m-1}\big)P^\eta(X_1=z_1,X_m=z_m,\mathcal{E}_{m-1}).
\end{multline*}
For each term of the sum, we proceed with the first factor exactly as for $m=1$. 
Replacing $L=\phi_T$ by $L=J$ in \eqref{concrete estimate for initial step large gamma}, we find 
\begin{equation*}
P^{\eta} \big( X_{1+m} - X_m \le a - (z_m-z_1) \big| X_1=z_1,X_m=z_m, \mathcal{E}_{m-1}\big)  \; \le \; P^{\eta} \big(Y_m \le a - (z_m - z_1) \big).  
\end{equation*} 
Altogether, we obtain
\begin{align*}
P^{\eta} (X_{1+m} - X_1 \le a)
\; &\le \;  \sum_{z_1,z_m \in \Z} P^{\eta} (Y_m \le a - (z_m - z_1)) P^{\eta}( X_1=z_1,X_m=z_m) \, + \, P^{\eta} (\mathcal E_{m-1}^c)\\
\; &= \; 
P^{\eta} (X_{1+(m-1)} + Y_m - X_1 \le a) \, + \, P^{\eta} (\mathcal E_{m-1}^c), 
\end{align*}
as desired. 
\end{proof}

\subsection{Renormalization procedure}\label{subsection: Renormalization procedure}

Let $\mathsf v$ be as in Lemma \ref{Lemma: Initial stpe large gamma}. 

\begin{Proposition}\label{Proposition: Renormalization large gamma}
Let $T$ be large enough 
and let $\gamma \ge \phi_T^3$.
Given $\eta\in \{ 0,1\}^\Z$ and $t\geq T$ such that 
\begin{equation*}
\forall y\in \mathrm B({0},t), \qquad  y \in G(\eta, \phi_t ),
\end{equation*}
it holds that 
\begin{equation}\label{content of drifs at a fixed time}
P^{\eta} \Big( X_t \le \frac{\mathsf v}{2} t \Big) \; \le \; \ed^{- \phi_t^{1/4}}.
\end{equation}
\end{Proposition}

\begin{figure}[h!]
\begin{center}
\includegraphics[width=0.8\textwidth]{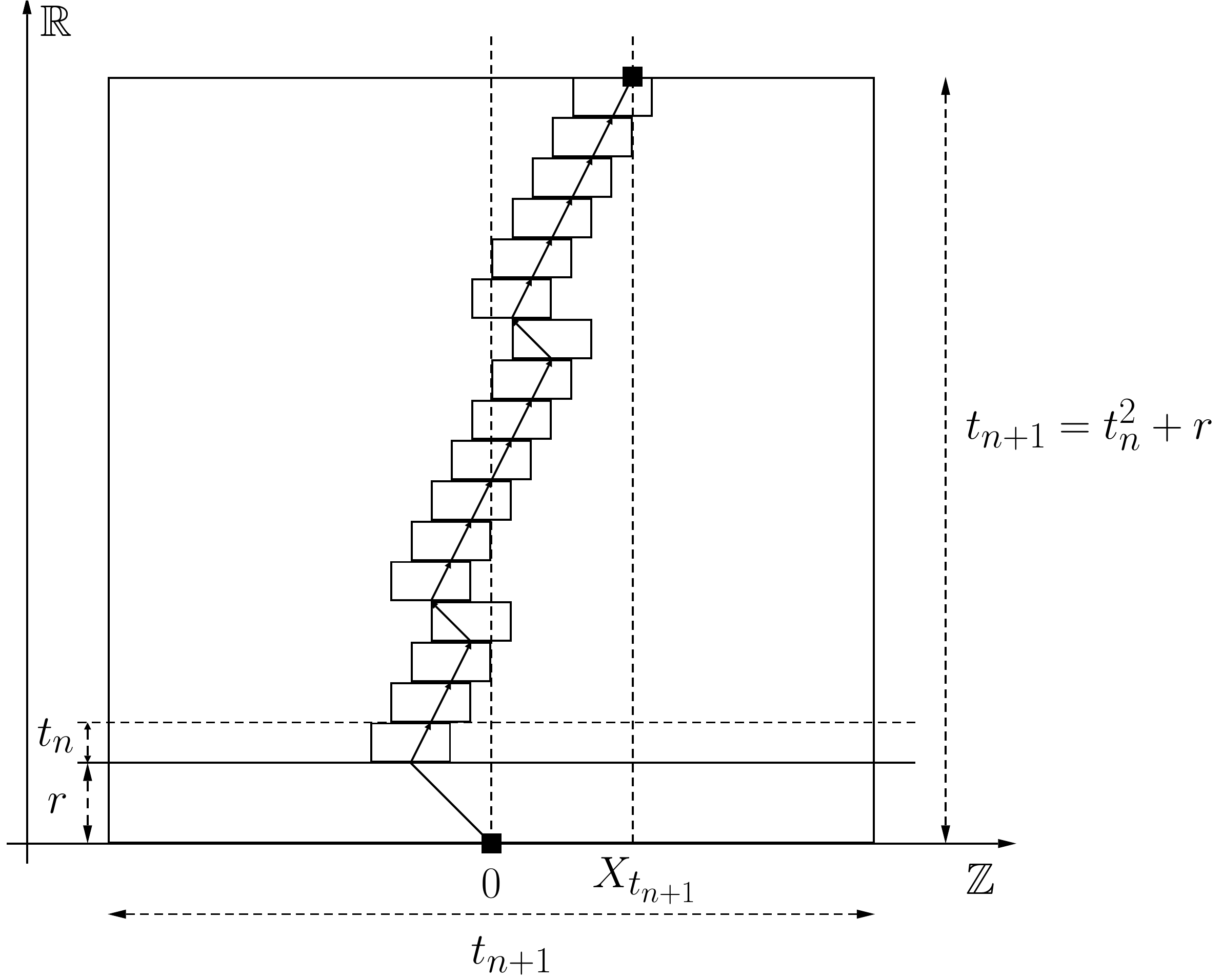}
\end{center}
\caption{
\label{figure: Renormalization}
Renormalization procedure used in the proof of Proposition \ref{Proposition: Renormalization large gamma}. 
Up to time $r$, we assume the worst case scenario ($X_r = -r$) since we lack a good control on the initial environment. 
With high probability, it is then possible to apply $t_n$ times our inductive hypothesis, providing a lower bound for the drift of the walker. 
In the rare cases ($4^{\mathrm th}$ and $11^{\mathrm{th}}$ space-time blocks from below on the figure) where $X_{kt_n} - X_{(k-1)t_n} < c_n t_n$, we assume again the worst case scenario. 
}
\end{figure}

\begin{proof}[Proof of Proposition \ref{Proposition: Renormalization large gamma}]
The main construction of the proof is illustrated on Figure \ref{figure: Renormalization}. 
We consider $T$ large enough so that the conclusions of Lemma \ref{Lemma: Initial stpe large gamma} hold for the time $T^{1/3}$.
We then define a sequence $(t_n)_{n\ge 0}$ with $t_0 \in [T^{1/3},T]$ and
\begin{equation}\label{deftn}
\textrm{for } n\geq 0,\qquad t_{n+1}\in[t_{n}^2,(t_n+1)^2]
\end{equation}
such that for some $N\ge 0$, $t_N=t$ (for example define recursively, from $t_N=t$ with some suitable $N$, $t_{k-1}=\lfloor \sqrt{t_k} \rfloor$ until reaching the interval $[T^{1/3},T]$).
We define also a sequence $(c_n)_{n\ge 0}\subset \R_+$ by
\begin{equation}\label{cn sequence}
\mathsf v \; \ge \; c_n \; = \; \bigg(3-\sum_{k=0}^n \frac{1}{2^k}\bigg) \frac{\mathsf v}{2} \; \ge \; \frac{\mathsf v}{2} \qquad (n\ge 0). 
\end{equation}
We prove by recurrence on $n\ge 0$ that, given $\eta\in \{ 0,1\}^\Z$ and $n\ge n_0$, 
if, for any $y\in \mathrm B({0},t_n)$, it holds that $y \in G(\eta, \phi_{t_n})$, then 
\begin{equation}\label{drift: to show by recurrence}
P^{\eta} (X_{t_n} \le c_n t_n) \; \le \; \ed^{-\phi_{t_n}^{1/4}}. 
\end{equation}
This will imply the claim. 

By Lemma \ref{Lemma: Initial stpe large gamma} and the hypotheses, \eqref{drift: to show by recurrence} holds true for $n=0$ since $\gamma \ge \phi^3_{T}$.
We now assume that \eqref{drift: to show by recurrence} holds for some $n\ge 0$, and we show that it implies it for $n+1$.  
We follow the same steps as in the proof of Lemma \ref{Lemma: Initial stpe large gamma}.
To simplify notations, let us write $t_n = \vart$ and $t_{n+1} = \vart'$, as well as $c_n = \varc$ and $c_{n+1} = \varc'$. 
As in Lemma \ref{Lemma: Initial stpe large gamma}, we first need to wait an initial step, 
as the information on the initial environment does not allow us to use our inductive hypothesis. 
It follows from the definition \eqref{deftn} of $(t_n)_{n\geq 1}$ that  $\vart' = r + (\vart - 1)\vart$ for some $\vart \le r < 4 \vart$. 
We write
\begin{equation}\label{a first estimate in the inductive step}
P^{\eta} (X_{\vart'} \le \varc' \vart' )
\; \le \; 
P^{\eta} (X_{\vart'} - X_{r} \le \varc' \vart' + r).
\end{equation}
If $n_0$ is large enough, waiting this time $r$ will suffice to dissipate possible traps with high probability. 

As in the proof of Lemma \ref{Lemma: Initial stpe large gamma} we define on the same probability space (enlarged if necessary) a sequence $(Y_k)_{k\ge 1}$ of i.i.d.\@ random variables independent from both the exclusion process and the walker with distribution  
\begin{equation*}
P^{\eta} (Y_k  = \varc \vart) \; = \; 1 - \ed^{-\phi_\vart^{1/4}}, \qquad
P^{\eta} (Y_k = -\vart) \; = \; \ed^{-\phi_\vart^{1/4}}.
\end{equation*}
Using our inductive hypothesis, we aim to show that 
\begin{equation}\label{two part equation in renormalization large gamma}
P^{\eta} (X_{\vart'} - X_{r} \le \varc' \vart' + r ) \; \le \; P^{\eta} \Big(\sum_{k=1}^{\vart-1} Y_k \le \varc' \vart' + r \Big)  \; + \;
\sum_{\substack{y\in\mathrm B({\rd 0},\vart'), \\ r\le s \le \vart' }} P^{\eta} \big( y\notin G(\eta(s),\phi_{\vart}) \big).
\end{equation}
To see this, we show that, for any $m\ge 1$ and any $a\in \R$, we have 
\begin{equation}\label{another way to: large gamma renormalization inequality for iid variables}
P^{\eta}  (X_{r + m\vart} - X_{r} \le a ) \; \le \; P^{\eta} (X_{r+(m-1)\vart} + Y_m - X_r \le a) \; + \; P^{\eta} (\mathcal E_{m-1}^c)
\end{equation}
where
\begin{equation*}
\mathcal E_{m-1} \; = \; \big\{ \forall y \in \mathrm B({0},r+m\vart), y \in G(\eta(r + (m-1)\vart),\phi_\vart) \big\}.
\end{equation*}
Since $\vart' = r + (\vart-1)\vart$, \eqref{two part equation in renormalization large gamma} follows from \eqref{another way to: large gamma renormalization inequality for iid variables} by iteration.

We prove now \eqref{another way to: large gamma renormalization inequality for iid variables}. For any $m \ge 1$, 
\begin{multline*}
P^{\eta}  (X_{r + m\vart} - X_{r} \le a)
\; = \;
P^{\eta} (X_{r+m\vart} - X_{r+(m-1)\vart} + X_{r+(m-1)\vart} - X_r \le a ) \\
\; \leq \; 
P^{\eta}(\mathcal{E}_{m-1}^c)+\sum_{z_0,z_{m-1}\in \Z} P^{\eta} \big( X_{r+m\vart} - X_{r+(m-1)\vart} \le a - (z_{m-1} - z_0) \big| A_{z_0,z_{m-1}}\big) P^{\eta} (X_{r+(m-1)\vart}=z_{m-1},  X_r = z_0).
\end{multline*}
with 
\begin{equation*}
A_{z_0,z_{m-1}} = \{X_{r+(m-1)\vart}=z_{m-1},  X_r = z_0,\mathcal{E}_{m-1}\}.
\end{equation*}
Using the inductive hypothesis each term of the sum can be controlled by 
\begin{equation*}
P^{\eta} \big( X_{r+m\vart} - X_{r+(m-1)\vart} \le a - (z_{m-1} - z_0) \big| A_{z_0,z_{m-1}} \big) 
\; \le \;
P^{\eta} \big(Y_m \le a - (z_{m-1} - z_0)  \big).
\end{equation*}
Therefore indeed
$P^{\eta}  (X_{r + m\vart} - X_{r} \le a) \; \le \; P^{\eta} ( Y_m + X_{r+(m-1)\vart} - X_r \le a ) + P^{\eta} (\mathcal E_{m-1}^c)$.
So \eqref{another way to: large gamma renormalization inequality for iid variables} and hence \eqref{two part equation in renormalization large gamma} are shown. 

We now proceed to bound each term in the right hand side of \eqref{two part equation in renormalization large gamma} separately. 
First, it follows from the definition \eqref{deftn} of $(t_n)_{n \geq 0}$ and the definition \eqref{cn sequence} of $(c_n)_{n \geq 0}$ that, if $T$ is large enough,
\begin{equation}\label{intermediate bound on the cn}
M \vart \; := \; E^\eta (Y_k) \; = \; \varc \vart \big( 1 - \ed^{-\phi_\vart^{1/4}} \big)  - \vart \ed^{-\phi_\vart^{1/4}}
\; \ge \; 
\frac{\varc + \varc'}{2} \vart.
\end{equation} 
Therefore
\begin{align}
P^{\eta} \bigg(\sum_{k=1}^{\vart-1} Y_k \le \varc' \vart' + r \bigg)
\; &\le \; 
P^{\eta}\bigg( \sum_{k=1}^{\vart-1} \big(Y_k - M \vart\big) \le (\varc' - M)(\vart - 1)\vart + 8 \vart \bigg)
\nonumber\\
\; &\le \; 
P^{\eta}\bigg( \sum_{k=1}^{\vart-1} \big(Y_k - M \vart\big) \le -\frac{\varc - \varc'}{2}(\vart - 1)\vart + 8 \vart \bigg)
\nonumber\\
\; &\le \; 
P^{\eta}\bigg( \sum_{k=1}^{\vart-1} \big(Y_k - M \vart\big) \le -\vart^{3/4} (\vart -1) \bigg)
\;\; \le \;\; 
\ed^{- c \vart^{1/2}} ,
\label{iterative estimate in iterative step}
\end{align}
where we have used the decomposition $\vart' = r + (\vart-1)\vart$, as well as the bounds $r \le 4 \vart$ and $\varc' \le 1$ to get the first inequality,
the estimate \eqref{intermediate bound on the cn} to get the second one, 
the definitions \eqref{deftn} of $(t_n)_{n \geq 0}$ and \eqref{cn sequence} of $(c_n)_{n \geq 0}$ to get the third one, 
and finally a classical concentration bound for sum of independent bounded variables ($|Y_k|\le \vart$ a.s.) to obtain the last one.

We then use Proposition \ref{Proposition: good set at a later time} to get a bound on the second term in the right hand side of \eqref{two part equation in renormalization large gamma}. 
By assumption, $y \in G(\eta, \phi_{\vart'})$ as soon as $y \in \mathrm B({0}, \vart')$, 
while, by the definition \eqref{phi function} of the sequence $(\phi_L)$, we have $\gamma r \ge \gamma \vart \ge \phi^3_{\vart'}$. 
Therefore
\begin{equation}\label{probabilistic estimate in iterative step}
\sum_{\substack{y\in\mathrm B(0,\vart'), \\ r\le s \le \vart' }} P^{\eta} \big( y\notin G(\eta(s),\phi_{\vart}) \big)
\; \le \; 
\Const \vart'^{2}
\frac{ \exp\big(- c   \phi_t \rho^2 (\epsilon_{\phi_{\vart}} - \epsilon_{\phi_{\vart'} })^2 \big)}
{\rho^2 (\epsilon_{\phi_{\vart}} - \epsilon_{\phi_{\vart'}})^2}
\; \le \; \ed^{-\phi_\vart^{1/3}},
\end{equation}
where the last bound is obtained using the explicit expressions 
\eqref{epsilon sequence} for  $(\epsilon_L)$, \eqref{phi function} for $(\phi_L)$, and \eqref{deftn} for $(t_n)_{n \geq 1}$.

The result is obtained by inserting \eqref{iterative estimate in iterative step} and \eqref{probabilistic estimate in iterative step} in \eqref{two part equation in renormalization large gamma}, 
and then \eqref{two part equation in renormalization large gamma} in \eqref{a first estimate in the inductive step}.
\end{proof}

\subsection{Proof of Theorem \ref{Theorem: lim inf to the right}}\label{subsection: Proof of Theorem lim inf to the right}

\begin{proof}[Proof of Theorem \ref{Theorem: lim inf to the right}]

Let us first show that there exists a constant $\Const < + \infty$ and $\alpha > 0$ such that
\begin{equation}\label{admissible points typical}
\P_\rho \left(\exists y \in B(0,t) \textrm{ s.t. } y\notin G(\eta,\phi_t)\right) \; \le \; \Const\, \ed^{-t^{\alpha}}. 
\end{equation}
Indeed, because the equilibrium measure is a Bernoulli product measure, we find
\begin{align*}
\P_\rho (\exists y \in B(0,t) \textrm{ s.t. } y\notin G(\eta,\phi_t))
\; &\le \; 
\sum_{y\in \mathrm B(0,t)} \P_\rho (y \notin G(\eta,\phi_t))
\; \le \; 
\sum_{y\in \mathrm B(0,t)} \sum_{L \ge \phi_t} \P_\rho \big( \langle \eta \rangle_{y,L} > (1 + \epsilon_{\phi_t}) \rho \big)\\
\; &\le \; 
2 \sum_{y\in \mathrm B(0,t)} \sum_{L \ge \phi_t} \ed^{- c L\epsilon^2_{\phi_t}\rho}
\; \le \;
\Const t \, \frac{\ed^{-c \rho \phi_t \epsilon^2_{\phi_t}}}{\epsilon^2_{\phi_t}\rho}.
\end{align*}
Relation \eqref{admissible points typical} is then obtained by inserting the values of $\phi_t$ and $\epsilon_{\phi_t}$ given by \eqref{epsilon sequence} and \eqref{phi function}.

By the Borel-Cantelli lemma, it then follows from Proposition \ref{Proposition: Renormalization large gamma} and \eqref{admissible points typical} 
that $\liminf_{t\rightarrow \infty} X_t \ge \mathsf v t/2$ $P_0-a.s.$,
where $\mathsf v$ as in Lemma \ref{Lemma: Initial stpe large gamma}. 
\end{proof}

\section{Drift for small $\gamma$}\label{Section: small gamma}
\label{sec:left}
We here prove
\begin{Theorem}\label{Theorem: drift for small gamma}
Assume that the condition \eqref{static transience} for transience to the right in a static environment holds. 
Then, for $\gamma$ small enough, there exists $v(\gamma)>0$ such that
\begin{equation*}
\liminf_{t\to \infty} \frac{X_t}{t} \; = \; v(\gamma) \qquad P_0-a.s.
\end{equation*}
\end{Theorem}

\noindent
Remark that the ellipticity condition \eqref{ellipticity} is not required. 
The cases where \eqref{static transience} holds while \eqref{ellipticity} is violated correspond to the cases $(\alpha = 1,\beta> 0)$ or $(\alpha >0,\beta=1)$. 
Since \eqref{static transience} still holds if the values or $\alpha$ or $\beta$ are lowered by a small enough amount, 
we conclude by a coupling argument that the hypothesis \eqref{ellipticity} can be added without loss of generality. 
We will thus assume that  \eqref{ellipticity} holds.

\subsection{Outline of the proof}
The overall strategy to prove Theorem \ref{Theorem: lim inf to the right} and Theorem \ref{Theorem: drift for small gamma} are completely analogous, but the actual implementation of the proofs differs in several respects. 

The main new conceptual difficulty appears when initializing the renormalization procedure. 
Indeed, Theorem \ref{Theorem: lim inf to the right} was obviously valid in the formal limit ``$\gamma = \infty$", and we could easily provide the initial step in Lemma \ref{Lemma: Initial stpe large gamma}.
Theorem \ref{Theorem: drift for small gamma} becomes wrong though, in the limit $\gamma = 0$. 
Two competing factors determine the behavior of the walk in the regime $\gamma \to 0$ ($\gamma > 0$):
though the evolution of the walker is better and better approximated by its evolution in a static environment as $\gamma \to 0$, implying transience to the right, the dissipation of traps gets slower and slower
(the asymptotic velocity of the walker goes to 0 as $\gamma \to 0$).
Quantitative estimates are needed to compute the resulting effect.  

In Lemma \ref{briquestat} below, we provide the needed specific estimates expressing that, if $\gamma=0$, there exists $\delta > 0$ so that $X_t \sim t^{\delta}$ if the hypothesis \eqref{static transience} holds.
Taking now $\gamma > 0$ small enough, we show in Lemma  \ref{Lemma: bloc for small gamma} below
that the behavior of the walker is well approximated by its behavior in a static environment as long as  $t\sim \gamma^{-1}$, so that $X_t \sim t^\delta$ for $t \sim \gamma^{-1}$.
It proves convenient to look at the evolution of $Y_t := X_{\gamma^{-1}t}$ instead of $X_{t}$ (we do not use the notation $Y_t$  through the proof): 
First, since the environment evolves over time scales of order $\gamma^{-1}$, corresponding to one ``step" $Y_t - Y_{t-1}$, the dissipation of traps seen by the walker $(Y_t)_{t\ge 0}$ does not get slowed as $\gamma \to 0$.  
Second, the expected drift of $(Y_t)_{t\ge 0}$, $E_0(Y_1)$, does not vanish in the limit $\gamma \to 0$ (it is actually of order $\gamma^{-\delta}$). 
Therefore, for $(Y_t)_{t\ge 0}$, we can proceed as in the previous section: 
Lemma  \ref{Lemma: bloc for small gamma} will play the role of  Lemma \ref{Lemma: Initial stpe large gamma}, to initialize the procedure, 
and Proposition \ref{Proposition: small gamma} the role of Proposition \ref{Proposition: Renormalization large gamma}, where the renormalization is worked out. 
Finally, as $\gamma$ will eventually be fixed, the result for $(Y_t)_{t\ge 0}$ implies the result for $(X_t)_{t\ge 0}$.  

There is however still one technical problem hidden in the above description. 
In the previous section, we used repeatedly the deterministic bound $X_t - X_s > - (t-s)$ for $s\le t$. 
However, for $(Y_t)_{t\ge 0}$, this becomes $Y_t - Y_s > - \gamma^{-1}(t-s)$.
As this bound clearly deteriorates as $\gamma \to 0$, we need to replace it by a probabilistic estimate.
This is the aim of Lemma \ref{Lemma: crude lower bound} below (see also the comments above this lemma). 
Unfortunately, the lack of a simple deterministic bound makes the whole proof considerably more heavy. 

Several constants will appear through the proof. 
Since it matters to select them in a given order, we list all of them already now (though some of them will only be formally introduced later): 
\begin{enumerate}
\item
As before, we let the sequence $(\epsilon_L)_L$ be given by \eqref{epsilon sequence}.

\item
We fix the parameters of the model, $\alpha$, $\beta$ and $\rho$, so that \eqref{ellipticity} and \eqref{static transience} hold.

\item
We fix an exponent $q\ge 2$, that appears in Lemma \ref{Lemma: bloc for small gamma} and Proposition \ref{Proposition: small gamma} below. 
It is used to quantify the probability of the exceptional event where the the walker does not drift to the right (see \eqref{result proposition gamma small} below) 

\item
We fix $\tau > 0$ small enough so that the environment can be accurately approximated by a static environment on time intervals of length $\gamma^{-1}\tau$.
This is used in the proofs of Lemmas \ref{Lemma: bloc for small gamma} and \ref{Lemma: crude lower bound} below.

\item
We fix $K > 0$ large enough so that environments with no trap of size larger or equal to $K \ln L$ in a given region of size $L$, are typical with respect to the equilibrium measure.
This is used  in Lemmas \ref{Lemma: bloc for small gamma} and \ref{Lemma: crude lower bound} and in Proposition \ref{Proposition: small gamma} below.  
Incidentally, taking $K$ large enough allows also to successfully apply Propositions \ref{Proposition: good set at a later time} and \ref{Proposition: conservation of good sets}. 

\item
We fix $\delta>0$ small enough so that,  in a static environment, the walker drifts a distance $t^\delta$ to the right in a time $t$ with high probability.
It is used in Lemmas \ref{Lemma: bloc for small gamma} and \ref{Lemma: crude lower bound} and in Proposition \ref{Proposition: small gamma} below.  

\item
Finally, we take $\gamma>0$ small enough. 
\end{enumerate}

\subsection{Static environment}\label{subsection: static}
In this section, we consider a random walk in a {static} random environment. We refer to \cite{Zeitouni} for background.
Given an environment $\omega \in[0,1]^\Z$, we define $S^\omega_x$ to be the law of the walker starting at time $0$ in $x \in \Z$, and evolving in the static environment $\omega$. 
We write $S^\omega$ for $S^\omega_0$. 
Given $\eta\in\{0,1\}^\Z$, we can define via \eqref{defomega} an environment $\omega(\eta)$. 
To simplify notations, we write $S^{\eta}$ for $S^{\omega(\eta)}$. 
Given a static environment $\omega$, we define the associated potential $V$ by
\begin{equation*}
V(0)=0 \qquad \textrm{and}\qquad\forall i\in\Z,\quad  V(i+1)-V(i)\; =\;\ln \frac{1-\omega_{i+1}}{\omega_{i+1}}.
\end{equation*}

\begin{Lemma}\label{briquestat}
Let $K$, and $\delta$ be positive numbers. 
Let $\tilde\rho \in [\rho,1[$ be small enough so that 
\begin{equation*}
\E^{\tilde\rho} \left( \ln \frac{1- \omega}{\omega} \right) \; < \; \frac{1}{2} \E^{\rho} \left( \ln \frac{1- \omega}{\omega} \right),
\end{equation*}
where $\E^{\rho}$ denotes the expectation with respect to $\P_{\rho}$.
\begin{enumerate}

\item
There exists $C>0$ so that, for $t$ large enough (depending on $\delta$), if the environment satisfies $\langle \eta \rangle_{-t^\delta/2,t^\delta/2} \le \tilde\rho$, then
\begin{equation}
\label{hypbs2}
S^\eta \left( \inf_{0 \le s \le t} X_s \le - t^\delta \right)
\; \le \;  e^{-Ct^{\delta}}.
\end{equation}

\item
For $t$ large enough (depending on $\delta$ and $K$), if the environment satisfies
\begin{equation*}
\forall x \in \mathrm B(0,2t^\delta), \qquad \langle \eta \rangle_{x,K \ln t^\delta} \; \le \; \tilde\rho,
\end{equation*}
then, there exists $C_K>0$ such that
\begin{equation}
\label{hypbs1}
S^\eta (X_t \le t^{\delta}) 
\; \le \; t^{-(1- C_K \delta)}.
\end{equation}
\end{enumerate}
\end{Lemma}
Note that the assumption of point $2$ implies the assumption of point $1.$

\begin{proof}
For simplicity, to avoid the use of integer parts, we assume that $t^\delta$ is an integer. 

We start with point $1.$
Under the assumption $\langle \eta \rangle_{-t^\delta/2,t^\delta/2} \le \tilde\rho$,
\begin{equation*}
V(-t^{\delta}) \;\ge\; -t^{\delta}\E^{\tilde{\rho}}\left(\ln\frac{1-\omega}{\omega}\right) \;\ge\; -t^{\delta}\frac{\E^{{\rho}}(\ln\frac{1-\omega}{\omega})}{2}.
\end{equation*}
Recall that (see e.g.\@ \cite{Zeitouni})
\begin{equation*}
S^{\eta}(T_{-t^{\delta}}<T_1)=\frac{1}{\sum_{i=-t^\delta}^{0}e^{V(i)}}\leq e^{t^{\delta}\frac{\E^{{\rho}}(\ln\frac{1-\omega}{\omega})}{2}}.
\end{equation*}
Therefore, using Markov's property at successive return times in $0$, we obtain
\begin{equation*}
S^\eta \left( \inf_{0 \le s \le t} X_s \le - t^\delta \right)\leq t e^{t^{\delta}\frac{\E^{{\rho}}(\ln\frac{1-\omega}{\omega})}{2}}.
\end{equation*}
This concludes the proof of \eqref{hypbs2}.

We turn to point $2$. 
We first control the probability that the walker has not exited the interval $[-2t^{\delta},2t^{\delta}]$ at time $t$. Let $T:=\min\{k\geq 0, \ X_k\notin[-2t^{\delta},2t^{\delta}]\}$. Define the environment $\tilde{\omega}$ for $i\in\Z$ by
\begin{equation*}
\tilde{\omega}_i=\begin{cases}\omega_i \quad \textrm{ if } i \geq -2t^\delta ,\\ 0\quad \textrm{ if } i< -2t^\delta . \end{cases}
\end{equation*}
By an obvious coupling of $S^{\eta}$ and $S^{\tilde\omega}$,
\begin{equation*}
S^{\eta}(T\geq t) \leq  S^{\tilde\omega}(T_{2t^{\delta}}\geq t),
\end{equation*}
where $T_x$ denotes the hitting time of $x\in \Z$.
Using a classical recurrence (see \cite{Zeitouni} p.$59$ for example), we find
\begin{equation*}
E_{S^{\tilde\omega}}(T_{2t^{\delta}})\leq \frac{1}{\alpha}(4t^{\delta})^2 e^{[V]_{-t^{\delta},t^{\delta}}},
\end{equation*}
where for $a,b\in \Z$, $[V]_{a,b}=\max\{V(j)-V(i),\ 1\leq i\leq j \leq b\}$. 
The assumption of point $2$ implies that any excursion above a minimum of $V$ has length at most $K\ln t^\delta$, and thus $[V]_{-2t^{\delta},2t^{\delta}}\leq K\ln(t^{\delta}) \ln(\frac{1-\alpha}{\alpha})$.
Finally, using Markov's inequality,
\begin{equation*}
S^{\eta}(T\geq t)\leq \frac{16}{\alpha}t^{2\delta-1+\delta K \ln(\frac{1-\alpha}{\alpha})}.
\end{equation*}

We turn to the probability that the walker is at the left of $t^\delta$ at time $t$:
\begin{equation*}
S^\eta(X_{t}\leq t^{\delta}) \leq S^\eta(T\geq t)+ S^\eta( \inf_{0 \le s \le t} X_s \le - 2t^\delta)+S^\eta(T_{2t^{\delta}}<t, X_{t} \leq t^\delta).
\end{equation*}
We have already controlled the first two terms. For the last one, we apply Markov's property at time $T_{t^\delta}$ and then we use \eqref{hypbs2} to derive 
\begin{equation*}
S^\eta(T_{2t^{\delta}}<t, X_{t} \leq t^\delta) \; \le \; t e^{\frac{t^{\delta}}{4}\E^{{\rho}}(\ln\frac{1-\omega}{\omega})} \;\le\; e^{-Ct^\delta}.
\end{equation*}
The largest of the three terms is thus the first one, at least for $t$ large enough. 
This concludes the proof of \eqref{hypbs1}.
\end{proof}

\subsection{Initial step}

The assumption $\gamma t \ge L^3$ in Proposition \ref{Proposition: good set at a later time} will read $t \ge K^3\ln^3 (\gamma^{-\delta}t)$, for $t$ replaced by $\gamma^{-1}t$ and $L$ replaced by $K\ln (\gamma^{-\delta}t)$, 
as it will be the case in the renormalization procedure (see Proposition \ref{Proposition: small gamma} below).
This procedure can thus only be initiated for times satisfying this bound, hence the choice of a window $\ln^4 (\gamma^{-1})\leq T\leq \ln^9(\gamma^{-1})$ in the next lemma:

\begin{Lemma}\label{Lemma: bloc for small gamma}
Let $q\ge 2$.
Let $K>0$ large enough then $\delta > 0$ small enough and then $\gamma > 0$ small enough. 
Let $\ln^4 (\gamma^{-1}) \leq T \leq \ln^9 (\gamma^{-1})$. 
Assume that the initial environment $\eta \in \{ 0,1 \}^\Z$ is such that 
\begin{equation*}
\forall x \in \mathrm B\left(0, (\gamma^{-\delta} T)^2 \right), \qquad x \in G \left(\eta, K \ln (\gamma^{-\delta}T)\right).
\end{equation*}
Then, 
\begin{equation*}
P^\eta(X_{\gamma^{-1}T} \; \leq \; \gamma^{-\delta} T) \; \leq \; \frac{1}{T^q} .
\end{equation*}
\end{Lemma}

\begin{proof}
Let $\tau>0$ be a real number that we will fix later. Let us first define an event $E$ relative to the exclusion process alone.
We set 
\begin{equation*}
E = \bigcap_{j=0}^{T/\tau - 1} (A_{j} \cap B_{j}) ,
\end{equation*}
with, for $j=0,\cdots,T/\tau-1$,
\begin{itemize}
\item 
$(\eta(t))_{t\ge 0} \in A_{j}$ if and only if
\begin{equation*}
\forall x \in \mathrm B\left(0, (\gamma^{-\delta} T)^2 \right), \qquad \langle \eta (t_j) \rangle_{x,L} \; \le \; (1 + \epsilon_{L^{1/2}}) \rho
\qquad \text{where} \qquad
t_j = j\tau, \quad L = K \ln (\gamma^{-\delta}T),
\end{equation*}

\item
 $(\eta(t))_{t\ge 0} \in B_{j,}$ if and only if, for all $x \in \mathrm B (0, (\gamma^{-\delta} T)^2 )$, the number of jumps of the exclusion process that moves particles in the time-space window $[\gamma^{-1}t_j,\gamma^{-1} t_{j+1}] \times  [x - L,x +L]$ is bounded by $\tau^{1/2} (2L)$. 
 \end{itemize}
We then decompose 
\begin{equation}\label{decomposition initial step small gamma}
P^\eta(X_{\gamma^{-1}T} \; \leq \; \gamma^{-\delta} T)
\; \le \;
P^\eta(X_{\gamma^{-1}T} \; \leq \; \gamma^{-\delta} T | E) \, + \, P^\eta (E^c).
\end{equation}

Let us estimate the first term in the right hand side of \eqref{decomposition initial step small gamma}.
Let $\omega$ be an environment such that the corresponding $(\eta(t))_{t\ge 0}$ satisfies $(\eta(t))_{t\ge 0} \in E$. 
Let $x_j = j\gamma^{-\delta}\tau$ ($0 \le j \le T/\tau - 1$).
A coupling argument shows that
\begin{equation}\label{small gamma initial step bound as product}
P^\omega_{0,0}(X_{\gamma^{-1}T} \, > \, \gamma^{-\delta} T )
\; \ge \; 
\prod_{j=0}^{T/\tau - 1} P^\omega_{\gamma^{-1}t_j,x_j} \left( X_{\gamma^{-1}t_{j+1}} - X_{\gamma^{-1}t_j} > \gamma^{-\delta}\tau \right).
\end{equation}
For any $0 \le j \le T/\tau - 1$, we define $\tilde\eta$ by 
\begin{equation*}
\tilde\eta_j \; = \; \max_{t_j \le s \le t_{j+1}} \eta (\gamma^{-1}s). 
\end{equation*}
On $A_j\cap B_j$, it holds that 
\begin{equation*}
\langle \tilde\eta_j \rangle_{x_j,L}  \; \le \; (1 + \epsilon_{L^{1/2}} + \tau^{1/2}) \rho \; := \; \tilde\rho. 
\end{equation*}
For $\tau$ and $\gamma$ small enough, it holds that $\mathbb E^{\tilde\rho} ((1-\omega)/\omega) < \frac{1}{2}\mathbb E^{\rho} ((1-\omega)/\omega)$. 
It then follows from the second part of Lemma \ref{briquestat} with $t=\gamma^{-1}\tau$, that for $\gamma$ small enough,
\begin{multline}\label{first use of brique stat}
P^\omega_{\gamma^{-1}t_j,x_j} \left( X_{\gamma^{-1}t_{j+1}} - X_{\gamma^{-1}t_j} > \gamma^{-\delta}\tau \right)
\; \ge \;
S^{\tilde\eta_j}_{x_j} \left(  X_{\gamma^{-1}\tau} - x_j > \gamma^{-\delta}\tau \right) \\
\; \ge \; 
S^{\tilde\eta_j}_{x_j} \left(  X_{\gamma^{-1}\tau} - x_j > (\gamma^{-1} \tau)^\delta \right) 
\;\ge\;
1 - \gamma^{1 - C_{K} \delta}.
\end{multline}
Inserting this estimate in \eqref{small gamma initial step bound as product}, we find that, for $\delta>0$ small enough and $\gamma>0$ small enough, 
\begin{equation*}
P^\omega_{0,0}(X_{\gamma^{-1}T} \; > \; \gamma^{-\delta} T )
\; \ge \; 
\left( 1 - \gamma^{1 - C_{K}\delta} \right)^{T/\tau}
\; \ge \; 
\exp \left(- (T/\tau) \gamma^{1 - C_K\delta}/2\right).
\end{equation*}
Therefore, thanks to the hypothesis $\ln^4(\gamma^{-1}) \leq T \leq \ln^9(\gamma^{-1})$, it also holds that 
\begin{equation}\label{small gamma initial step first thing}
P^\eta(X_{\gamma^{-1}T} \;  \le \; \gamma^{-\delta} T | E)
\; \le \; 
1 -\exp \left(- (T/\tau) \gamma^{1 - C_K\delta}/2\right) 
\; \le \; T^{-(q+1)}. 
\end{equation}

We next come to the second term in the right hand side of  \eqref{decomposition initial step small gamma}:
\begin{equation}\label{small gamma initial step second part}
P^\eta (E^c) 
\; \le \; 
\sum_{j=0}^{T/\tau - 1} P^\eta \left( (A_{j} \cap B_{j})^c \right)
\; \le \; 
\sum_{j=0}^{T/\tau - 1} \left( P^\eta (A_{j}^c) + P^\eta (B_{j}^c|A_{j}) \right).
\end{equation}
First, applying Proposition \ref{Proposition: conservation of good sets} we find that for $\gamma$ small enough and $j=0,\cdots, T/\tau-1$,
\begin{equation*}
P^\eta (A_{j}^c) 
\; \le \; 
\sum_{x\in \mathrm B (0, (\gamma^{-\delta} T)^2)} P^{\eta}(\langle \eta(t_j)\rangle_{x,L} > (1 + \epsilon_{L^{1/2}})\rho)
\; \le \; C (\gamma^{-\delta} T)^2 \ed^{-c L (\epsilon_{L^{1/2}} - \epsilon_L)^2}
\; \le \; (\gamma^{-\delta}T)^{2 - c K}.
\end{equation*}
Next, by a classical concentration bound, for $\tau$ small enough, there exists $c(\tau)>0$ so that, for $j=0,\cdots, T/\tau-1$,
\begin{equation*}
 P^\eta (B_{j}^c|A_{j}) 
\; \le\;  
C (\gamma^{-\delta} T)^2 \ed^{-c(\tau) L} \; \le \; (\gamma^{-\delta}T)^{2 - c K}.
\end{equation*}
Inserting these two last bounds in \eqref{small gamma initial step second part}, we conclude that, if $K$ is large enough,
\begin{equation}\label{small gamma initial bloc before end}
P^\eta (E^c) 
\; \le \; 
T^{-(q+1)}.
\end{equation}

The result is obtained by inserting \eqref{small gamma initial step first thing} and \eqref{small gamma initial bloc before end} in \eqref{decomposition initial step small gamma}. 
\end{proof}

\subsection{Rough lower bound for intermediate times}

Lemma \ref{Lemma: bloc for small gamma} derived above does not yet allow to initiate the renormalization procedure described in the proof of Proposition \ref{Proposition: small gamma} below.
Indeed, while it ensures that the walker moves a distance $\gamma^{-\delta}t$ to the right over a time $\gamma^{-1}t$, with probability $1-1/t^q$ for good initial environments, 
we only know that the walker does not move more than a distance  $\gamma^{-1}t$ to the left, with probability $1/t^q$ for good initial environments. 
Since Lemma \ref{Lemma: bloc for small gamma} is only valid for a time  $t\leq \ln^9(\gamma^{-1})$, we cannot yet exclude that $E^\eta (X_{\gamma^{-1}T})< 0$. 
For intermediate times $t$ such that $T\le t \le \gamma^{-1}$, with $T$ as in Lemma \ref{Lemma: bloc for small gamma}, 
the next lemma furnishes a better lower bound than the deterministic bound $X_s \ge -s$ ($s\ge 0$). 
In particular, once combined with Lemma~\ref{Lemma: crude lower bound}, Lemma~\ref{Lemma: bloc for small gamma} ensures now that the walker drifts to the right over a time $\gamma^{-1}T$. 
We will use the full power of Lemma~\ref{Lemma: crude lower bound}, i.e.\@ when $T< t \le \gamma^{-1}$, for the first iterations in the proof of Proposition~\ref{Proposition: small gamma}, where the same type of difficulty shows up. 


\begin{Lemma}\label{Lemma: crude lower bound}
Let $K>0$ large enough, then $\delta > 0$ small enough, and then $\gamma >0$ small enough. 
Assume that $T$ is chosen as in Lemma \ref{Lemma: bloc for small gamma}, i.e.\@ $\ln^4(\gamma^{-1}) \leq T \leq \ln^9(\gamma^{-1})$, and let $T \le t \le \gamma^{-1}$.
Assume that the initial state $\eta\in \{ 0,1\}^\Z$ is such that  
\begin{equation}\label{small gamma crude lower bound hypothesis}
\forall x \in \mathrm B\left(0, (\gamma^{-\delta} t)^2 \right), \qquad x \in G \left(\eta, K \ln (\gamma^{-\delta}t)\right).
\end{equation}
Then for all $1 \le s \le t$,  
\begin{equation*}
P^{\eta} \left(X_{\gamma^{-1}s} \le - \gamma^{-\delta} s\right) \; \le \;  \ed^{ -s^{\delta/2}} .
\end{equation*}
\end{Lemma}

\begin{proof}
We first notice that the hypothesis \eqref{small gamma crude lower bound hypothesis} implies 
\begin{equation}\label{small gamma rewriting hypothesis crude lower bound}
\forall x \in \mathrm B\left(0, (\gamma^{-\delta} t)^2 \right), \qquad x \in G \left(\eta, (\tau \gamma^{-1})^\delta/2\right), 
\end{equation}
provided that $(\epsilon_L)_{L\ge 0}$ is no longer given by \eqref{epsilon sequence}, but is such that 
\begin{equation}\label{small gamma local definition of epsilon sequence}
\epsilon_{(\tau \gamma^{-1})^\delta/2} \; = \; \frac{1}{1 + \ln \left( 1 + \ln (\gamma^{-\delta}t) \right)}.
\end{equation}
For the present proof, we assume that \eqref{small gamma local definition of epsilon sequence} holds instead of \eqref{epsilon sequence}, and we use directly the hypothesis \eqref{small gamma rewriting hypothesis crude lower bound} instead of \eqref{small gamma crude lower bound hypothesis}. 

Fix $1\leq s\leq t$ and define an event $E$ relative to the exclusion process alone.
We set 
\begin{equation*}
E = \bigcap_{j=0}^{s/\tau - 1} (A_{j} \cap B_{j}),
\end{equation*}
with
\begin{itemize}
\item  
$(\eta (t))_{t\ge 0} \in A_{j}$ if and only if
\begin{equation*}
\langle \eta (\gamma^{-1}t_j) \rangle_{x_j-L,L} \; \le \; (1 + \epsilon_{L^{1/2}}) \rho \qquad \text{where} \qquad
t_j = j \tau, \quad
x_j = -k (\tau\gamma^{-1})^\delta, \quad
L = (\tau \gamma^{-1})^\delta/2, 
\end{equation*}

\item  
$(\eta(t))_{t\ge 0} \in B_{j}$ if and only if the number of jumps of the exclusion process that moves particles in the time-space window $[x_j - 2L,x_j] \times [\gamma^{-1}t_j, \gamma^{-1} t_{j+1}]$ is bounded by $\tau^{1/2} (2L)$. 
\end{itemize}
We then decompose 
\begin{equation}\label{decomposition rough lower bound small gamma}
P^\eta \left( X_{\gamma^{-1}s} \le - \gamma^{-\delta}s  \right)
\; \le \;
P^\eta \left( X_{\gamma^{-1}s} \le - \gamma^{-\delta}s |\, E \right) \, + \, P^\eta (E^c).
\end{equation}

Let us estimate the first term in the right hand side of \eqref{decomposition rough lower bound small gamma}.
Let $\omega$ be an environment such that the corresponding $(\eta(t))_{t\ge 0}$ satisfies $(\eta(t))_{t\ge 0} \in E$.
A coupling argument shows that 
\begin{equation}\label{small crude lower bound real first thing}
P^\omega_{0,0} \left( X_{\gamma^{-1}s} \le - \gamma^{-\delta}s\right)
\; \le \; 
\sum_{j=0}^{s/\tau - 1} P^\omega_{\gamma^{-1}t_j,x_j} \left(  X_{\gamma^{-1}t_{j+1}} - X_{\gamma^{-1}t_j} \le -2L  \right).
\end{equation}
For any $0 \le j \le s/\tau - 1$, we define $\tilde\eta$ by 
\begin{equation*}
\tilde\eta_j \; = \; \max_{t_j \le r \le t_{j+1}} \eta (\gamma^{-1}r). 
\end{equation*}
It holds that 
\begin{equation*}
\langle \tilde\eta_j \rangle_{x_j-L,L}  \; \le \; (1 + \epsilon_{L^{1/2}} + \tau^{1/2}) \rho \; := \; \tilde\rho. 
\end{equation*}
For $\tau$ small enough and $\gamma$ small enough, it holds that $\E^{\tilde\rho} ((1-\omega)/\omega) < \frac{1}{2}\E^{\rho} ((1-\omega)/\omega)$. 
By a coupling argument and the first part of Lemma \ref{briquestat}, we obtain, for some $c>0$, 
\begin{equation*}
P^\omega_{\gamma^{-1}t_j,x_j} \left(  X_{\gamma^{-1}t_{j+1}} - X_{\gamma^{-1}t_j} \le -2L \right)
\; \le \; 
S^{\tilde\eta_j}_{x_j} \left( X_{\gamma^{-1}\tau} - x_j \le - 2L \right) 
\; \le \;  \ed^{-cL}.
\end{equation*}
Therefore, inserting this estimate in \eqref{small crude lower bound real first thing}, it holds that $P^\omega_{0,0} \left( X_{\gamma^{-1}s} \le - \gamma^{-\delta}s\right) \le (s/\tau)\ed^{-cL}$, so that 
\begin{equation}\label{conclusion small crude lower bound real first thing}
P^\eta \left( X_{\gamma^{-1}s} \le - \gamma^{-\delta}s |\, E \right) \;\le\; \frac{s}{\tau} \ed^{-cL}.
\end{equation}

For the second term in the right hand side of \eqref{decomposition rough lower bound small gamma}, we have too
\begin{equation}\label{small gamma crude lower bound second thing}
P^\eta (E^c) 
\; \le \; 
\sum_{j=0}^{s/\tau - 1} P^\eta \left( (A_{j} \cap B_{j})^c \right)
\; \le \; 
\sum_{j=0}^{s/\tau - 1} \left( P^\eta (A_{j}^c) + P^\eta (B_{j}^c|A_{j}) \right) 
\; \le \;
\frac{s}{\tau} \ed^{-c L}, 
\end{equation}
since, applying Proposition \ref{Proposition: conservation of good sets} we find $P^\eta (A_{j}^c) \le \ed^{-c L}$ if $K$ is large enough, while, if $\tau$ is chosen small enough, a classical concentration bound furnishes $ P^\eta (B_{j}^c|A_{j}) \le \ed^{-cL}$.

Inserting \eqref{conclusion small crude lower bound real first thing} and \eqref{small gamma crude lower bound second thing} into \eqref{decomposition rough lower bound small gamma}, we obtain for $\gamma$ large enough,
\begin{equation*}
P^\eta \left( X_{\gamma^{-1}s} \le - \gamma^{-\delta}s  \right)
\; \le \; \Const \gamma^{-1} \tau^{-1} \ed^{- cL} \;\le \; \ed^{-\gamma^{-\delta/2}} \; \le \; \ed^{-s^{\delta/2}}.
\end{equation*}
since $s\le \gamma^{-1}$.
\end{proof}

\subsection{Renormalization procedure}\label{section: renormalization small gamma}

We use Lemma \ref{Lemma: bloc for small gamma} as initial step, Proposition \ref{Proposition: good set at a later time} to guarantee the dissipation of possible traps with high probability, and Lemma \ref{Lemma: crude lower bound} as lower bound for intermediate times, to derive

\begin{Proposition}\label{Proposition: small gamma}
Let $q\ge 2$.
Let $K$ be large enough, then $\delta > 0$ small enough and finally $\gamma > 0$ small enough. 
Fix $t \ge \ln^9 (\gamma^{-1})$ and assume that the initial environment $\eta\in \{ 0,1\}^\Z$ is such that  
\begin{equation*}
\forall x \in \mathrm B\left(0,(\gamma^{-\delta}t)^2\right), \qquad x \in G \left(\eta, K \ln (\gamma^{-\delta}t) \right).
\end{equation*}
Then, 
\begin{equation}\label{result proposition gamma small}
P^{\eta} (X_{\gamma^{-1}t} \le \gamma^{-\delta}t/2) \; \le \;  1/t^q.
\end{equation}
\end{Proposition}

\begin{proof}
We take $K< + \infty$, $\delta > 0$ and $\gamma > 0$ such that Lemmas \ref{Lemma: bloc for small gamma} and \ref{Lemma: crude lower bound} hold. 
Given $t\geq \ln^9 (\gamma^{-1})$, we define a sequence $(t_n)_{n\ge 0}$ with $t_0 \in [\ln^4\gamma^{-1},\ln^9\gamma^{-1}]$ and for $n\geq 0$, $t_{n+1}\in[t_{n}^2,(t_n+1)^2]$, such that for some $N\ge 0$, $t_N=t$. 
Recall the definition of $(c_n)_{n\ge 0}$ in \eqref{cn sequence} with $\mathsf v = 1$ so that the sequence is decreasing from $1$ to $1/2$. 
We prove by recurrence that, given $\eta\in \{ 0,1\}^\Z$ and $n\ge 0$, if, for any $x\in \mathrm B(0,(\gamma^{-\delta}t_n)^2)$, it holds that $x \in G(\eta, K \ln \gamma^{-\delta}t_n )$, then 
\begin{equation}\label{drift for small gamma: to show by recurrence}
P^{\eta} (X_{\gamma^{-1}t_n} \le c_n \gamma^{-\delta} t_n) \; \le \; t_n^{-q}. 
\end{equation}
This will imply the claim. 

By Lemma \ref{Lemma: bloc for small gamma} and by the hypotheses, \eqref{drift for small gamma: to show by recurrence} holds true for $n=0$.
We now assume that \eqref{drift for small gamma: to show by recurrence} holds for some $n\ge 0$, and we show that it implies it for $n+1$.  
To simplify notations, let us write $t_n = {\vart}$ and $t_{n+1} = {\vart}'$, as well as $c_n = \varc$ and $c_{n+1} = \varc'$. 
It holds that  ${\vart}' = r + ({\vart} - 1){\vart}$ for some ${\vart} \le r < 4{\vart}$. 
The cases ${\vart}' \le \gamma^{-1}$ and ${\vart}' > \gamma^{-1}$ are treated differently.  
We first write a bound analogous to \eqref{a first estimate in the inductive step} in the proof of Proposition \ref{Proposition: Renormalization large gamma}:
\begin{align}
P^{\eta} (X_{\gamma^{-1}{\vart}'}  \le \varc' \gamma^{-\delta}{\vart}' )
\; &\le \; 
P^{\eta} \left( X_{\gamma^{-1}{\vart}'} - X_{\gamma^{-1}r} \le \varc' \gamma^{-\delta} {\vart}' + \gamma^{-\delta}r \right)
\, + \, 
\ed^{-{\vart}^{\delta/2}}
& \text{if} \quad {\vart}' \le \gamma^{-1}, 
\label{a first estimate in the inductive step for small gamma small t} \\
P^{\eta} (X_{\gamma^{-1}{\vart}'}  \le \varc' \gamma^{-\delta}{\vart}' )
\;& \le \; 
P^{\eta} \left( X_{\gamma^{-1}{\vart}'} - X_{\gamma^{-1}r} \le \varc' \gamma^{-\delta} {\vart}' + \gamma^{-1}r \right)
& \text{if} \quad{\vart}' > \gamma^{-1},
\label{a first estimate in the inductive step for small gamma large t}
\end{align}
where \eqref{a first estimate in the inductive step for small gamma small t} is derived from Lemma \ref{Lemma: crude lower bound} while the bound $X_{\gamma^{-1}r}\geq -\gamma^{-1}r$ in \eqref{a first estimate in the inductive step for small gamma large t} is deterministic.

We need to evaluate the the right hand side of either of these bounds. 
For this, we define on the same probability space (enlarged if necessary) a sequence $(Y_k)_{k\ge 1}$ of i.i.d.\@ random variables independent from both the exclusion process and the walker with distribution: 
\begin{align}
&P^{\eta} \left(Y_k  = \varc \gamma^{-\delta} {\vart} \right) \; = \; 1 - {\vart}^{-q}, \qquad P^{\eta} \left(Y_k = -\gamma^{-\delta}{\vart}\right) \; = \;  {\vart}^{-q}
& \text{if} \quad {\vart}' \le \gamma^{-1}, \label{small gamma Y intermediate t}\\
& P^{\eta} \left(Y_k  = \varc \gamma^{-\delta} {\vart} \right) \; = \; 1 - {\vart}^{-q}, \qquad P^{\eta} \left(Y_k = -\gamma^{-1}{\vart} \right) \; = \;  {\vart}^{-q}
&  \text{if} \quad {\vart}' > \gamma^{-1}. \label{small gamma Y large t}
\end{align}
For any integer $m \ge 0$, let us also define the events 
\begin{align*}
\mathcal D_m \; &= \; 
\left\{
X_{\gamma^{-1}(r+m{\vart})} \ge - \gamma^{-\delta}{\vart}'
\right\}, \\
\mathcal E_m \; &= \; 
\left\{ 
\forall x \in \mathrm B\left(0,(\gamma^{-\delta}{\vart}')^2 \right), \, x \in G\Big(\eta\left(\gamma^{-1}(r+m{\vart})\right),K \ln (\gamma^{-\delta}{\vart})\Big) 
\right\}.
\end{align*}
Using our inductive hypothesis, we aim to show that, for ${\vart}' \le \gamma^{-1}$, 
\begin{multline}\label{three part equation in renormalization small gamma intermediate t}
P^{\eta} \left( X_{\gamma^{-1}{\vart}'} - X_{\gamma^{-1}r} \le \varc' \gamma^{-\delta} {\vart}' + \gamma^{-\delta}r \right) 
\; \le \; 
P^{\eta} \left( \sum_{k=1}^{{\vart}-1} Y_k \le \varc' \gamma^{-\delta} {\vart}' + \gamma^{-\delta}r \right) \\
\, + \,
\sum_{1 \le m \le {\vart}-1} P^\eta \left( X_{\gamma^{-1}(r+m{\vart})} - X_{\gamma^{-1}(r+(m-1) {\vart})} \le - \gamma^{-\delta}{\vart} \big| \mathcal D_{m-1},\mathcal E_{m-1} \right)\\
\, + \,
\sum_{1\le m \le {\vart}-1} \big(  P^\eta  (\mathcal D^c_{m-1}) + P^{\eta} (\mathcal E^c_{m-1}) \big)
\end{multline}
with $(Y_k)_{k\ge 1}$ given by \eqref{small gamma Y intermediate t},
and for ${\vart}' > \gamma^{-1}$, 
\begin{equation}\label{three part equation in renormalization small gamma large t}
P^{\eta} \left( X_{\gamma^{-1}{\vart}'} - X_{\gamma^{-1}r} \le \varc' \gamma^{-\delta} {\vart}' + \gamma^{-1}r \right) 
\; \le \; 
P^{\eta} \left( \sum_{k=1}^{{\vart}-1} Y_k \le \varc' \gamma^{-\delta} {\vart}' + \gamma^{-1}r \right)  \, + \,  \sum_{1\le m \le \vart-1}  P^{\eta} (\mathcal E^c_{m-1})
\end{equation}
with $(Y_k)_{k\ge 1}$ given by \eqref{small gamma Y large t}.
One sees that the first of these bounds, valid for ${\vart}' \le \gamma^{-1}$, involves two extra terms in comparison with the second one, valid for  ${\vart}' >\gamma^{-1}$. 
This comes form the fact that, in the first case, it is not always so that, after a time $\gamma^{-1}(r+(m-1) {\vart})$ (for some $1 \le m \le \vart - 1$), the walker is on a site where we have a control on the environment, while it is always so in the  second case (the initial environment is controlled in a box of size $(\gamma^{-\delta}\vart')^2$).

The bounds \eqref{three part equation in renormalization small gamma intermediate t} and \eqref{three part equation in renormalization small gamma large t} are shown in an analogous way; 
as the latter is easier, we focus on the derivation of \eqref{three part equation in renormalization small gamma intermediate t}.
Let us thus assume ${\vart}' \le \gamma^{-1}$.
Let us establish that, for any $m\ge 1$ and any $a\in \R$, we have 
\begin{multline}\label{another way to: small gamma renormalization inequality for iid variables}
P^{\eta}  (X_{\gamma^{-1}(r + m{\vart})} - X_{\gamma^{-1}r} \le a ) 
\; \le \; 
P^{\eta} (X_{\gamma^{-1}(r+(m-1){\vart})} + Y_m - X_{\gamma^{-1}r} \le a)\\
+ \, P^{\eta} \left( X_{\gamma^{-1}(r + m{\vart})} - X_{\gamma^{-1}(r + (m-1){\vart})} \le - \gamma^{-\delta}{\vart} \big| \mathcal D_{m-1},\mathcal E_{m-1} \right)
\, + \,  P^{\eta} (\mathcal D_{m-1}^c) \,+ \, P^{\eta} (\mathcal E_{m-1}^c) .
\end{multline}
Since ${\vart}' = r + ({\vart}-1){\vart}$, \eqref{three part equation in renormalization small gamma intermediate t} follows from \eqref{another way to: small gamma renormalization inequality for iid variables} by iteration and Fubini theorem.
For $m\geq 1$:
\begin{multline}
\label{small gamma induction antepenultimate term}
P^{\eta}  (X_{\gamma^{-1}(r + m{\vart})} - X_{\gamma^{-1}r} \le a)  \; \le\;   \\
P^{\eta} (X_{\gamma^{-1}(r+m{\vart})} - X_{\gamma^{-1}(r+(m-1){\vart})} + X_{\gamma^{-1}(r+(m-1){\vart})} - X_{\gamma^{-1}r}  \le a | \mathcal D_{m-1},\mathcal 
E_{m-1}) \, + \, P^\eta (\mathcal D_{m-1}^c) + P^\eta (\mathcal E_{m-1}^c). 
\end{multline}
The first term in the right hand side is expressed as 
\begin{multline}\label{small gamma induction penultimate term}
\sum_{z \geq - \gamma^{-\delta}{\vart}'} 
P^\eta \left( 
X_{\gamma^{-1}(r+m{\vart})} - X_{\gamma^{-1}(r+(m-1){\vart})} \le a - z + X_{\gamma^{-1}r} \big| \mathcal D_{m-1}, \mathcal E_{m-1}, X_{\gamma^{-1}(r+(m-1){\vart})}=z 
\right)\\
P^\eta (X_{\gamma^{-1}(r+(m-1){\vart})}=z | \mathcal D_{m-1},\mathcal E_{m-1}).
\end{multline}
Our inductive hypothesis at scale $n$ and Markov's property imply that 
\begin{multline}\label{small gamma induction last term}
P^\eta \left( 
X_{\gamma^{-1}(r+m{\vart})} - X_{\gamma^{-1}(r+(m-1){\vart})} \le a - z + X_{\gamma^{-1}r} \big| \mathcal D_{m-1}, \mathcal E_{m-1}, X_{\gamma^{-1}(r+(m-1){\vart})}=z 
\right)
\; \le \; \\
P^\eta (Y_{m} \le a - z + X_{\gamma^{-1}r}) \, + \,  P^{\eta} (X_{\gamma^{-1}(r+m{\vart})} - X_{\gamma^{-1}(r+(m-1){\vart})} \le - \gamma^{-\delta}{\vart} | \mathcal D_{m-1},\mathcal E_{m-1}, X_{\gamma^{-1}(r+(m-1){\vart})}=z ).
\end{multline}
Inserting \eqref{small gamma induction last term} into \eqref{small gamma induction penultimate term}, and then \eqref{small gamma induction penultimate term} into \eqref{small gamma induction antepenultimate term} leads to \eqref{another way to: small gamma renormalization inequality for iid variables}, hence to \eqref{three part equation in renormalization small gamma intermediate t}.

We now proceed to bound each term in the right hand side of \eqref{three part equation in renormalization small gamma intermediate t} and \eqref{three part equation in renormalization small gamma large t} separately.
Let us start with \eqref{three part equation in renormalization small gamma intermediate t}.
To deal with the first term in the right hand side of \eqref{three part equation in renormalization small gamma intermediate t}, we define for any integer $1 \le p \le t-1$, the event
\begin{equation*}
A_p \; = \; \{\exists \ 1 \le   k_1< \dots < k_p \le t-1:Y_{k_j} < 0 \; \text{for all}\; 1 \le j \le p \}. 
\end{equation*}
We have
\begin{equation*}
P^\eta (A_p) \; \le \; \Big( \frac{{\vart}}{{\vart}^2} \Big)^p = \frac{1}{{\vart}^p}. 
\end{equation*}
We take $p=3(q+1)$ and we show that 
\begin{equation}\label{induction small gamma the first bound to be collected}
P^{\eta} \left( \sum_{k=1}^{{\vart}-1} Y_k \le \varc' \gamma^{-\delta} {\vart}' + \gamma^{-\delta}r \right)
\; \le \; P^{\eta} \left( \sum_{k=1}^{t-1} Y_k \le \varc' \gamma^{-\delta} {\vart}' + \gamma^{-\delta}r , A_p^c \right) \, + \, P^\eta (A_p) \; \le \; \frac{1}{{\vart}^p} \; \le \; \frac{1}{{(\vart}')^{(q+1)}} , 
\end{equation}
provided that $\gamma$ was taken small enough. 
Indeed, the first term in the right hand side of \eqref{induction small gamma the first bound to be collected} vanishes since, on $A_p^c$, it holds that 
\begin{equation*}
\sum_{k=1}^{{\vart}-1} Y_k \; \ge \; 
\varc \gamma^{-\delta} {\vart} ({\vart} - 1-p) - p \gamma^{-\delta} {\vart}  \; > \; \varc' \gamma^{-\delta} {\vart}' + \gamma^{-\delta}r,
\end{equation*}
as indeed the last inequality reads $(\varc - \varc') \vart > C(p)$ (which holds true since $\varc - \varc' = 2^{-(n+2)}$ while $\vart \ge \ed^{c 2^n}$ for some $c> 0$).
To deal with the second term in the right hand side of \eqref{three part equation in renormalization small gamma intermediate t}, we apply Lemma \ref{Lemma: crude lower bound}:
\begin{equation}\label{induction small gamma the second bound to be collected}
\sum_{1 \le m \le {\vart}-1} P^\eta \left( X_{\gamma^{-1}(r+m{\vart})} - X_{\gamma^{-1}(r+(m-1) {\vart})} \le - \gamma^{-\delta}{\vart} \big| \mathcal D_{m-1},\mathcal E_{m-1} \right)
\; \le \;
({\vart}-1) \ed^{-{\vart}^{\delta/2}} \; \le \; \frac{1}{({\vart}')^{q+1}}. 
\end{equation}
Similarly,
\begin{equation}\label{induction small gamma the third bound to be collected}
\sum_{1\le m \le {\vart}-1}   P^\eta  (\mathcal D^c_{m-1})  \; \le \; ({\vart}-1) \ed^{-{\vart}^{\delta/2}} \; \le \; \frac{1}{({\vart}')^{q+1}}. 
\end{equation}
Finally, to bound for any $m\geq 0$, $P^{\eta}(\mathcal E^c_m)$, we apply Proposition \ref{Proposition: good set at a later time}. 
The hypothesis reads here $\gamma (\gamma^{-1}(r+m{\vart})) \ge (K \ln (\gamma^{-\delta}{\vart}))^3$, and is satisfied since, for $\gamma$ smal enough, $r+m{\vart} \ge {\vart} \ge  \ln^4 \gamma^{-1}$.
Therefore, taking $K$ large enough, 
\begin{multline}\label{induction small gamma the fourth bound to be collected}
\sum_{1\le m \le {\vart}-1}   P^\eta  (\mathcal E^c_{m-1}) 
\; \le \; 
\sum_{\substack{x\in\mathrm B(0,(\gamma^{-\delta}{\vart}')^2), \\ 1\le m \le {\vart}-1 }} P^{\eta} \left( x\notin G(\eta(\gamma^{-1}(r+m {\vart})), K \ln \gamma^{-\delta}{\vart}) \right) \\
\; \le \;
({\vart}-1)(\gamma^{-\delta} {\vart}')^2 \ed^{-cK \ln \gamma^{-\delta}{\vart}} \; \le \; \frac{1}{({\vart}')^{q+1}}.
\end{multline}
For \eqref{three part equation in renormalization small gamma large t}, both the equivalent of \eqref{induction small gamma the first bound to be collected} and of \eqref{induction small gamma the fourth bound to be collected} remain valid. Inserting (\ref{induction small gamma the first bound to be collected}-\ref{induction small gamma the fourth bound to be collected}) into \eqref{three part equation in renormalization small gamma intermediate t} and then \eqref{three part equation in renormalization small gamma intermediate t} into \eqref{a first estimate in the inductive step for small gamma small t} for ${\vart}' \le \gamma^{-1}$, or inserting the equivalent of \eqref{induction small gamma the first bound to be collected} and \eqref{induction small gamma the fourth bound to be collected} into \eqref{three part equation in renormalization small gamma large t} and then \eqref{three part equation in renormalization small gamma large t} into \eqref{a first estimate in the inductive step for small gamma large t} for ${\vart}' > \gamma^{-1}$, yields the result. 
\end{proof}

\subsection{Conclusion of the proof}

\begin{proof}[Proof of Theorem \ref{Theorem: drift for small gamma}]
If $K$ is large enough, we have, writing simply $\epsilon$ for $\epsilon_{K \ln \gamma^{-\delta}t}$,
\begin{align}\label{admissible points typical small gamma}
&\P_\rho \left( 
\exists x \in \mathrm B\left(0, (\gamma^{-\delta} t)^2\right) \, : \, x \notin G\left(\eta, K \ln (\gamma^{-\delta}t)\right) 
\right) 
\; \le \; 
\sum_{x\in \mathrm B\left(0, (\gamma^{-\delta} t)^2\right)} \P_\rho \left( x \notin G\left(\eta, K \ln (\gamma^{-\delta}t)\right)  \right) \nonumber\\
\; &\le \; 
\sum_{x\in \mathrm B\left(0, (\gamma^{-\delta} t)^2\right)} \sum_{L \ge K \ln (\gamma^{-\delta}t)}
\P_\rho \left( \langle \eta \rangle_{x,L} > (1 + \epsilon) \rho \right) 
\; \le \; 
\sum_{x\in \mathrm B\left(0, (\gamma^{-\delta} t)^2\right)} \sum_{L \ge K \ln (\gamma^{-\delta}t)} \ed^{-c L \epsilon^{2}} \nonumber\\
\; & \le \; 
\Const (\gamma^{-\delta}t)^2  \epsilon^{-2} \ed^{-c K \ln (\gamma^{-\delta}t) \epsilon^{2} }
\; \le \; \Const' (\gamma^{-\delta}t)^2 (\gamma^{-\delta}t)^{-c' K} \; \le \; t^{3-c'K}
\end{align}
once $t$ is large enough. This last term can be bounded by $1/t^2$ that is summable for $K$ large enough.
Therefore, since 
\begin{multline*}
P_0 (X_t < \gamma^{-\delta}t /2) \; \le \; 
P_0  \left( \exists x \in \mathrm B (0, (\gamma^{-\delta} t)^2) \, : \, x \notin G (\eta, K \ln (\gamma^{-\delta}t) \right) \\
+
P_0 \left(X_t < \gamma^{-\delta}t /2 \; \big| \;  \forall x \in \mathrm B (0, (\gamma^{-\delta} t)^2 ) \, : \, x \in G (\eta, K \ln (\gamma^{-\delta}t) \right),
\end{multline*}
it follows from the Borel-Cantelli lemma, from Proposition \ref{Proposition: small gamma} and from \eqref{admissible points typical small gamma} that, 
for $K$ large enough, $\delta> 0$ small enough and $\gamma >0$ small enough, 
\begin{equation*}
\liminf_{t\rightarrow \infty} X_t \; \ge \; \frac{\gamma^{-\delta} t}{2} \qquad P_0 - a.s.
\end{equation*}
\end{proof}

\section{Proof of Theorem \ref{Theorem: Drift to the right} and Theorem \ref{Theorem: Drift to the left}}\label{Section: Proof of law of large numbers}
Theorem \ref{Theorem: Drift to the right} is deduced from Proposition \ref{Proposition: Renormalization large gamma} and Theorem \ref{Theorem: lim inf to the right} exactly in the same way as Theorem \ref{Theorem: Drift to the left} is deduced from Proposition \ref{Proposition: small gamma} and Theorem \ref{Theorem: drift for small gamma}. 
We only show Theorem \ref{Theorem: Drift to the right}, and so we fix $\gamma$ large enough so that the conclusions of Proposition \ref{Proposition: Renormalization large gamma} and Theorem \ref{Theorem: lim inf to the right} are in force.

Let us first introduce a few definitions and notations. 
In Section \ref{subsec:model}, the law $\P$ of the environment was built from the exclusion process. 
In this section, it is more convenient to build it from the interchange process on $\Z$ together with an independent collection of particles of two different types, so that each site is occupied by exactly one particle. 
We use the following definitions: 
\begin{enumerate}
\item
Let us consider a collection of spatially independent Poisson clocks $(U(t,x))_{t\in\R,x\in\Z}$ with parameter $\gamma$, called updates. 
We construct a process $(\xi (t,i))_{t\in \R,i\in\Z}$, where, for $i\in\Z$ and $t\in {\R}$, $\xi(t,i)$ denotes the position at time $t$ of particle $i$:
$\xi (0,i)=i$ and, if the clock $x$ rings at time $t$, then the particle at $x-1$ and the particle at $x$ exchange their positions.
Remark that the time is indexed by $\R$: it is convenient for technical reasons, although not necessary to define the model, see Lemma \ref{lemma: particle too much to the right} for example.



\item
For $t \ge 0$ and $x\in\Z$, $\mu(t,x)$ is the unique $i$ such that $\xi(t,i)=x$ that is the label of the particle that is in $x$ at time $t$. 
Note that $\mu$ is a function of $\xi$. 

\item
We consider also a family of type of particles $(\nu(i))_{i\in \Z}$, which is a product of independent Bernoulli, with parameter $\rho$. 
Let $\P_2$ be its law. 
\end{enumerate}
The environment $\omega$ is viewed as a function of $\xi$ and $\nu$:
\begin{align*}
\omega(\xi,\nu) (t,i)= \begin{cases} \alpha &\mbox{if } \nu(\mu(t,i))=1, \\
		                                       \beta & \mbox{if } \nu(\mu(t,i))=0. \end{cases}
\end{align*}
We let the reader check that the law of the environment $\P$ defined in Section \ref{subsec:model} is the push forward probability of $\P_1 \otimes \P_2$ through this function. 
We denote the space where lives $\xi$ by $\Xi$, and the set of  path that supports the process $X$ by $\mathcal{P}$.

We will use the ellipticity of the environment to control the probability that the walker executes some displacements independently from any information we could have about the environment. 
Define this minimal probability by
\begin{equation}
\label{kappa}
\kappa = \min\{\alpha,1-\alpha,\beta,1-\beta\}.
\end{equation}
By assumption \eqref{ellipticity}, we have $\kappa>0$. 

We now start the proof of Theorem \ref{Theorem: Drift to the right}. 
Our aim is to construct a renewal structure, i.e.\@ to cut the random path followed by the walker into pieces that are independent under $P_0$ 
(see \cite{Avena dos Santos Vollering} for the first use of this method in a similar context, see also \cite{Berard Ramirez, Hilario et al}). 
For any time-space point $(t,x)\in \R\times \Z$, we define 
\begin{align}
T^-_{t,x}&=\{(s,y), s\leq t,y\leq x+\frac{\mathsf v}{4}(s-t)\},\\
T^+_{t,x}&=\{(s,y), t\leq s , y\geq x+\frac{\mathsf v}{4}(s-t)\},
\end{align}
where $\mathsf v$ is the constant appearing in Proposition \ref{Proposition: Renormalization large gamma} (see Figure \ref{figure: Renewal time}).
The position of the rightmost visited particle at time $n\geq 0$ is $M(n):=\sup\{\xi(n,\mu), \ \mu \in \mathcal{V}_n\}$, with $\mathcal{V}_n=\{\mu(i,X_i),\ i\leq n-1\}$ the set of labels of visited particles at time $n$. 
Our goal is to define a sequence of random times, called {renewal times}, that satisfy 
\begin{align}
\label{re1}&\forall n\leq \tau, \quad (n,X_n)\in T^-_{\tau,X_{\tau}},\\
\label{re2}& M(\tau)< X_\tau,\\
\label{re4}&\forall n\geq \tau, \quad (n,X_n)\in T^+_{\tau,X_{\tau}},\\
\label{re5}& \forall (n,x) \in T^+_{\tau,X_{\tau}},\quad \xi(\tau,\mu(n,x))\geq X_{\tau}
\end{align}
(see Figure \ref{figure: Renewal time}).
The meaning of \eqref{re1} and \eqref{re4} is clear. 
Condition \eqref{re2} means that, at time $\tau$, all particles visited by the walker before time $\tau$, are behind the walker.
Finally, \eqref{re5} means that all the particles at the left of $X_\tau$ at time $\tau$ will never enter into the cone $T^+_{\tau, X_\tau}$. 
These four conditions together imply in particular that after the time $\tau$, the walker will only visits particles that have not yet been visited at time $\tau$.

\begin{figure}[h!]
\begin{center}
\includegraphics[width=0.8\textwidth]{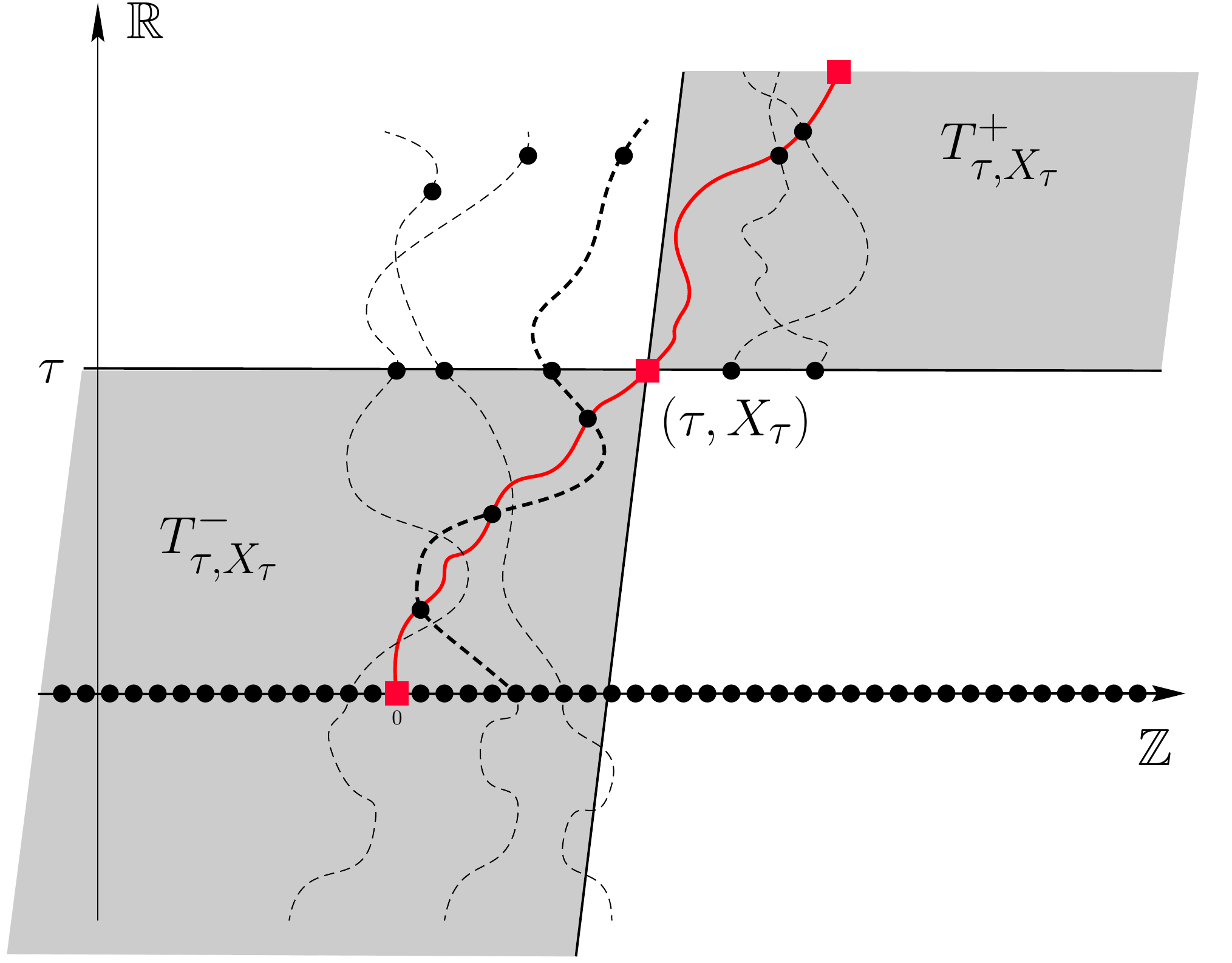}
\end{center}
\caption{
\label{figure: Renewal time}
Renewal time $\tau$. 
The slope of the oblique line is $(\mathsf v/4)^{-1}$. 
Conditions \eqref{re1} and \eqref{re4}: the walker (in red) lives in $T^{-}_{\tau,X_\tau}$ before time $\tau$ and in $T^{+}_{\tau,X_\tau}$ after time $\tau$ . 
Conditions \eqref{re1} and \eqref{re5}: all particles visited by the walker before time $\tau$ sit at its left at time $\tau$, 
and no particle at the left of $X_{\tau}$ at time $\tau$ enters $T^{+}_{\tau,X_\tau}$.
}
\end{figure}

Our goal is to prove that, $P_0-$a.s., there are infinitely many renewal times, and to show that the first and second moments of $\tau$ are finite, from which Theorem \ref{Theorem: Drift to the right} is readily deduced (see Section \ref{renewal: last subsection}). 
For this, our first step is to show that, almost surly, there is a positive fraction of the points on the trajectory $(n,X_n)_{n\ge 0}$, called candidates, satisfying the conditions (\ref{re1}-\ref{re2}) (see Proposition \ref{l:densite} below). 
Since the walker drifts ballistically in a diffusive environment, this claim is surely very reasonable, but it is however not straightforward:
While it follows directly form the ballistic behavior that a positive fraction of times satisfy \eqref{re1}, it is not obvious to prove that a positive fraction of these times  satisfy \eqref{re2}.
The proof of Proposition \ref{l:densite} constitues the most technical part of the work (see the heuristic considerations below Proposition \ref{l:densite}), 
and we stress that this problem could not be solved by the methods of \cite{Avena dos Santos Vollering, Berard Ramirez, Hilario et al}.
We next turn to  (\ref{re4}-\ref{re5}).
In Section \ref{subsection: beginning two other conditions}, we carry over the main computation that allows us to see that these conditions can be satisfied with positive probability
(the ellipticity of the walk, $\kappa > 0$, see \eqref{kappa}, is used here).
We are then able to construct an infinite sequence of renewal points (see Sections \ref{sub:Building the first renewal point}-\ref{sub:Building renewal points}) and,
along the way, we show the properties that make the interest of the renewal times (they cut the trajectory into independent pieces). 
Finally, the moments of $\tau$ are bounded by combining the bounds of Proposition \ref{Proposition: Renormalization large gamma} and Proposition \ref{l:densite} (see Proposition \ref{moment tau} in Section \ref{renewal: last subsection}).


\subsection{Existence of a density of points $(n,X_n)$ satisfying (\ref{re1}-\ref{re2})}
Given an environment, a point $(n,X_n)$ on a trajectory $(m,X_m)_{m\ge 0}$ is said to be a {candidate} if it satisfies \eqref{re1} and \eqref{re2}: 
$$\forall m\leq n, \quad (m,X_m)\in T^-_{n,X_n}, \qquad \text{and} \qquad M(n)< X_n.$$
We prove the existence of a positive fraction of candidate points on $P_0-$almost every trajectory: 
\begin{Proposition}\label{l:densite}
Let $q \ge 2$.
There exists some $c>0$ such that, for $n$ large enough, 
\begin{equation}\label{densite quantitatif}
P_0 (\sharp\{i\leq n, \ (i,X_i) \textrm{ is a candidate}\} \leq cn) \; \le \; \frac{1}{n^q}.
\end{equation}
In particular, by the Borel-Cantelli lemma, 
\begin{equation}\label{densite}
\liminf_{n\to+\8} \frac{\sharp\{i\leq n, \ (i,X_i) \textrm{ is a candidate \} }}{n}>c,\quad P_0-a.s.
\end{equation}
\end{Proposition}

Let us first describe the idea of the proof. 
For this, we start by introducing a slightly weaker notion than that of candidate. 
Let $l \ge 1$ and $n\in\N$. Given an environment, a point $(n,X_n)$ on a trajectory $(m,X_m)_{m\ge 0}$ is said to be $l-${good}, or simply good, if
\begin{align}
&\forall m\le n, \quad (m,X_m)\in T^-_{n,X_n} \label{good 1}\\
&M(n) < X_n + l. \label{good 2}
\end{align}
We will prove Proposition \ref{l:densite} in two steps: we first establish the result with ``candidate" replaced by ``$l-$good" (for some large enough $l$), and then we show that Proposition \ref{l:densite} follows from this intermediate statement 
(this second step is straightforward and we do not comment upon it here). 

Let us take some large integer $n$. 
On the one hand, by Proposition \ref{Proposition: Renormalization large gamma}, $X_n \ge \mathsf v n/2$ with high probability, for some $\mathsf v>0$. 
Therefore, it follows from basic geometric considerations that, with the same probability, a positive fraction of the points on the trajectory $(m,X_m)_{0\le m \le n}$ satisfy \eqref{good 1}. 
On the other hand, given a deterministic path $(j,Y_j)_{j\ge 0}$, with $Y_{j+1} - Y_j = \pm 1$, it holds that any given point $(k,Y_k)$ satisfying \eqref{good 1}, satisfies also \eqref{good 2} with a probability at least $1 - \ed^{-cl}$, 
where $c$ is some constant independent from the path and the point (see  Lemma \ref{lemma: particle too much to the right} below).
One may therefore try to establish that, with very high probability, any given path $(j,Y_j)_{j\ge 0}$ independent of the environment and having a density of points satisfying \eqref{good 1}, has also a density of good points. 
More precisely, since the number of such paths, differing from one another before the time $n$, is roughly bounded by $2^n$, we would need to establish that the probability that a given such path has less than $cn$ good points decays faster to $0$ than $2^{-n}$, provided that $c$ has been chosen small enough.
If the events that different points on the path are good were independent, we would indeed conclude by the above that this probability decays like $\ed^{-I(l)n}$, where $I(l)$ grows to infinity as $l$ grows to infinity, so that the result would follow by taking $l$ large enough. 

The lack of independence forces us to adapt this strategy. 
Let us assume that a point $(k,Y_k)$ satisfies \eqref{good 1} but not \eqref{good 2}.
We define a random variable $d$ to quantify how much \eqref{good 2} failed; roughly, $d$ measures the distance to the right of $Y_k$ travelled by the particles that made \eqref{good 2} to fail (see \eqref{def de d} for a precise definition). 
The main observation is that, when considering a next point $(k',Y_{k'})$  satisfying \eqref{good 1}, with $k' > k$, then, if $(k, Y_k + d) \in T^-_{k',Y_{k'}}$, we can estimate the probability that \eqref{good 2} is satisfied for $(k',Y_{k'})$, independently of our knowledge about $(k,Y_k)$.
This observation is implemented as follows: 
First, in order to decrease the cardinal of possible paths, we consider paths of blocks of size $l$ instead of paths of points. 
This helps since, while the number of such paths is now bounded by $2^{n/l}$, the probability that a given block is bad (see \eqref{bad part} below) still behaves like $1 - l^2 \ed^{-cl} \sim 1 - \ed^{-c'l}$. 
Then, each time a bad block is seen, we estimate how bad it was, i.e.\@ we estimate $E$, as defined in \eqref{min of maximal distance} below. 
The probability that $E$ exceeds a certain amount $k \ge 0$ decays exponentially with $k$ (see Lemma \ref{lemma: particle too much to the right} bellow).

Before starting the proof of Proposition \ref{l:densite}, let us state an elementary result relative to the updates of the environment alone (recall that the updates $U(t,x)$ are defined for all $t\in \R$).  
Given an environment, we say that a space-time point $(t,x)$ is $l-${bad}, or simply bad, if there exists a particle $i\in\Z$ and a time $s \leq t$ such that 
\begin{equation}\label{bad part}
 (s,\xi(s,i))\in T^-_{t,x}\quad \textrm{ and }\quad \xi(t,i)\geq l+x. 
\end{equation}
Remark that, contrary to the notion of good point (see (\ref{good 1}-\ref{good 2})), the notion of bad point does not involve the trajectory of the walker.
Given $(t,x)$ a bad point, $A_{t,x}$ denotes the set of particles satisfying \eqref{bad part}. 
Given a particle $i\in A_{t,x}$ we define 
\begin{align*}
&s(i)=\max\{s\leq t:\ (s,\xi(s,i))\in T^-_{t,x}\},\\
&\lambda(i)=\min\{k:\ \textrm{for all } s(i)\leq s \leq t,\ (s,\xi(s,i))\in T^-_{t,x+k}\}. 
\end{align*}
Finally we define the variable $d_{t,x}$ by
\begin{equation}
\label{def de d}
d_{t,x}= \inf\{\lambda(i), i\in A_{t,x}\}
\end{equation}
if $(t,x)$ is bad, and $d_{t,x}=0$ if $(t,x)$ is not bad. 
We stress that the variables $d_{t,x}$ are deterministically equal at $0$ or larger than $l$, measurable with respect to $\xi$, and identically distributed (but off course not independent).
\begin{Lemma}\label{lemma: particle too much to the right}
There exists $c > 0$ so that, for any $k\geq l$,
\begin{equation}\label{eq:cone}
\P_1(d_{0,0}\geq k)\leq \ed^{-c k}.
\end{equation}
\end{Lemma}
\begin{proof}
We first establish the exponential decay for large $k$.
For $k\geq l$, using a union bound,
\begin{align}
\label{debut d}
\P_1(d_{0,0}\geq k)&\; \le \; \sum_{i\geq l}\P_1( \exists s\leq 0 \textrm{ s.t.\@ } \xi(s,i)\notin T^-_{0,k-1},\ \exists t\leq s \textrm{ s.t.\@ } \xi(t,i)\in T^-_{0,0}).
\end{align}
For each $i\geq l$, $(\xi(s,i))_{s\leq 0}$ is a simple continuous time random walk so that there exists some constant $c_1>0$ so that $\P_1(\exists\ t\leq 0 \textrm{ s.t.\@ } \xi(t,i)\in T^-_{0,0})\leq \ed^{-c_1 i}$. 
The sum \eqref{debut d} is decomposed in two parts. 
First, there exists $c_2 > 0$ such that, for $k$ large enough,
\begin{equation*}
\sum_{i\geq k}\P_1( \exists s\leq 0 \textrm{ s.t.\@ } \xi(s,i)\notin T^-_{0,k-1},\ \exists t \leq s \textrm{ s.t.\@ } \xi(t,i)\in T^-_{0,0})
\; \le \; 
\sum_{i\geq k}\P_1(\exists t\leq 0 \textrm{ s.t.\@ } \xi(t,i)\in T^-_{0,0})
\; \le \; 
\ed^{-c_2 k}.
\end{equation*}
Second, let us consider an index $l\leq i\leq k-1$. 
Using Markov's property at the hitting time of the complementary of $T^-_{0,k-1}$ together with the fact the interchange process is considered under its invariant measure, we obtain
\begin{equation*}
\P_1( \exists s\leq 0 \textrm{ s.t.\@ } \xi(s,i)\notin T^-_{0,k-1},\ \exists t\leq s \textrm{ s.t.\@ } \xi(t,i)\in T^-_{0,0}) \; \le \; \P_1( \exists t\leq 0 \textrm{ s.t.\@ } \xi(t,k-1)\in T^-_{0,0}),
\end{equation*}
so that 
\begin{align*}
\sum_{l\leq i\leq k-1}\P_1( \exists s\leq 0 \textrm{ s.t.\@ } \xi(s,i)\notin T^-_{0,k-1},\ \exists t\leq s \textrm{ s.t.\@ } \xi(t,i)\in T^-_{0,0})\leq k \ed^{-c_1 k}.
\end{align*}
We thus obtain \eqref{eq:cone} for $k$ large enough.

As $\P_1(d_{0,0}\geq k)$ is non-increasing with $k\geq l$, it remains to prove that $\P_1(d_{0,0}\geq l)<1$ to complete the proof, or equivalently that $\P_1(d_{0,0}=0)>0$. 
We already know that for some $K$ large enough $\sum_{i\geq K}\P_1(\exists t\leq 0\ \text{s.t.}\ \xi(t,i)\in T^-_{0,0})<1$. We consider a time $s<0$ so that $-\frac{\mathsf v}{4}s > K$. 
Using Markov's property at time $s$ together with the fact the interchange process is considered under its invariant measure, we obtain
\begin{align*}
\P_1(d_{0,0}=0) &\;\ge \; \P_1\left(U(s,l)=U(0,l),\ \{\exists  i\geq l\ \exists t\leq s  \textrm{ s.t.\@ } \xi(s,i)\geq l \textrm{ and }\xi(t,i)\in T^-_{0,0}\}^c\right)\\
 &\; \ge \; \P_1\left(U(s,l)=U(0,l)\right)\  \left(1- \P_1(\exists i\geq K \ \exists t\leq 0 \textrm{ s.t.\@ } \xi(t,i)\in T^-_{0,0})\right) \; > \; 0, 
\end{align*}
where $U(s,l)=U(0,l)$ means that the clock between the sites $l-1$ and $l$ has not rung during the time interval $[0,s]$. 
That concludes the proof.
\end{proof}

\begin{proof}[Proof of Proposition \ref{l:densite}]
The proof is made of two steps. 

In a first step, we fix $l \ge 1$ large enough, and we prove that there exists some constant $c_0>0$ such that, for all $n$ large enough, 
\begin{equation}\label{inter}
P_0 (\sharp\{i\leq n, \ (i,X_i) \textrm{ is $l-$good}\} < c_0 n) \; \le \; \frac{1}{n^{q+1}}, 
\end{equation}
with $l-$good as defined in (\ref{good 1}-\ref{good 2}).
We observe that 
\begin{equation*}
P_0(\sharp\{i\leq n, \ (i,X_i) \textrm{ is good}\}\leq c_0 n)
\; \le \;   
P_0 \left( \sharp\{i\leq n, \ (i,X_i)\textrm{ is good}\}\leq c_0 n,X_n\geq \frac{\mathsf v}{2}n \right) + P_0\left( X_n\leq \frac{\mathsf v}{2}n\right) . 
\end{equation*}
By Proposition \ref{Proposition: Renormalization large gamma}, the second term is bounded by $ P_0(X_n\leq \frac{\mathsf v}{2}n) \le\ed^{-\phi_n^{1/4}} \le 1/n^{q+2}$, where the second inequality is valid for $n$ large enough. 
We thus need a bound on the first term. 

Let us consider a deterministic trajectory $(j,Y_j)_{0 \le j\le n}$ such that  $Y_0=0$, $Y_{j+1} - Y_j = \pm 1$ for all $0 \le j\le n-1$, and $Y_n \ge \frac{\mathsf v}{2}n$. 
We first describe how to create two trajectories of boxes of size $l$, called $\mathcal R$ and $\mathcal T$ below, starting from the trajectory of points $(j,Y_j)_{0 \le j \le n}$. 
Our constructions are illustrated on Figure  \ref{figure: Block construction}. 
For $j \geq 0$, we define 
\begin{equation*}
t_j \; = \; \inf \left\{k\geq 0 \textrm{ s.t.\@ } Y_k\geq j+\frac{\mathsf v}{4}k \right\}, 
\end{equation*}
and we remark that, on the event $\{X_n\geq \frac{\mathsf v}{2}n\}$, $t_i\leq n$ for all $i\leq \frac{\mathsf v}{4}n$. 
We define a space-time parallelogram, or box, by its opposite sides,
\begin{equation*}
\left[(0,0),\left(n,\frac{\mathsf v}{4} n\right)\right] \qquad \text{and} \qquad \left[\left(0,\frac{\mathsf v}{4} n\right),\left(n, \frac{\mathsf v}{2} n\right)\right],
\end{equation*}
and we cut it into boxes of size $l$: writing $ M = \lfloor \frac{n}{l}\rfloor$ and $N=\lfloor \frac{\mathsf vn}{4l} \rfloor$, 
we consider $M\times N$ boxes $C(i,j)$, with $i\in\{1,\dots,M\}$ and $j\in\{1,\dots, N\}$, where $C(i,j)$ is the box defined by the two opposite sides 
\begin{equation*}
\left[\left((i-1)l,(j-1)l+\frac{\mathsf v}{4}(i-1)l\right),\left(il,(j-1)l + \frac{\mathsf v}{4}il\right) \right] 
\quad \text{and} \quad 
\left[\left((i-1)l,jl+\frac{\mathsf v}{4}(i-1)l\right),\left(il,jl + \frac{\mathsf v}{4}il\right) \right].
\end{equation*}
We denote by $\mathcal R \subset \{ 1, \dots , M \} \times \{ 1, \dots , N \}$ the subset of indices $(i,j)$, such that $(i,j) \in \mathcal R$ if and only if the box $C(i,j)$ contains at least one point $(t_k,X_{t_k})$ for some $k \le \frac{\mathsf v}{4} n$.
We observe the following:
\begin{enumerate}
\item
For any $1 \le j \le N$, there exists at least one $i\in \{ 1, \dots , M\}$ such that $C(i,j) \in \mathcal R$.
\item
For any $1 \le i \le M$, there are at most two $j \in \{ 1 , \dots , N \}$ such that $C(i,j) \in \mathcal R$, and in that case the two indices $j$ are consecutive.
\end{enumerate}
It makes thus sense to consider the subset $\mathcal T\subset \mathcal R$ such that $(i,j) \in \mathcal T$ if $j$ is odd and if $i$ is the smallest number so that $(i,j) \in \mathcal R$. 
The cardinal of $\mathcal T$ is $N/2$ (assuming $N$ even, the other case being analogous) and the set $\mathcal T$ can be described as 
\begin{equation*}
\mathcal T = \left\{ (i_1,1), (i_2, 3), \dots , (i_{N/2},N/2) \right\} \quad \text{with} \quad 1 \le i_1 < i_2 < \dots < i_{N/2} \le M. 
\end{equation*}

\begin{figure}[h!]
\begin{center}
\includegraphics[width=0.8\textwidth]{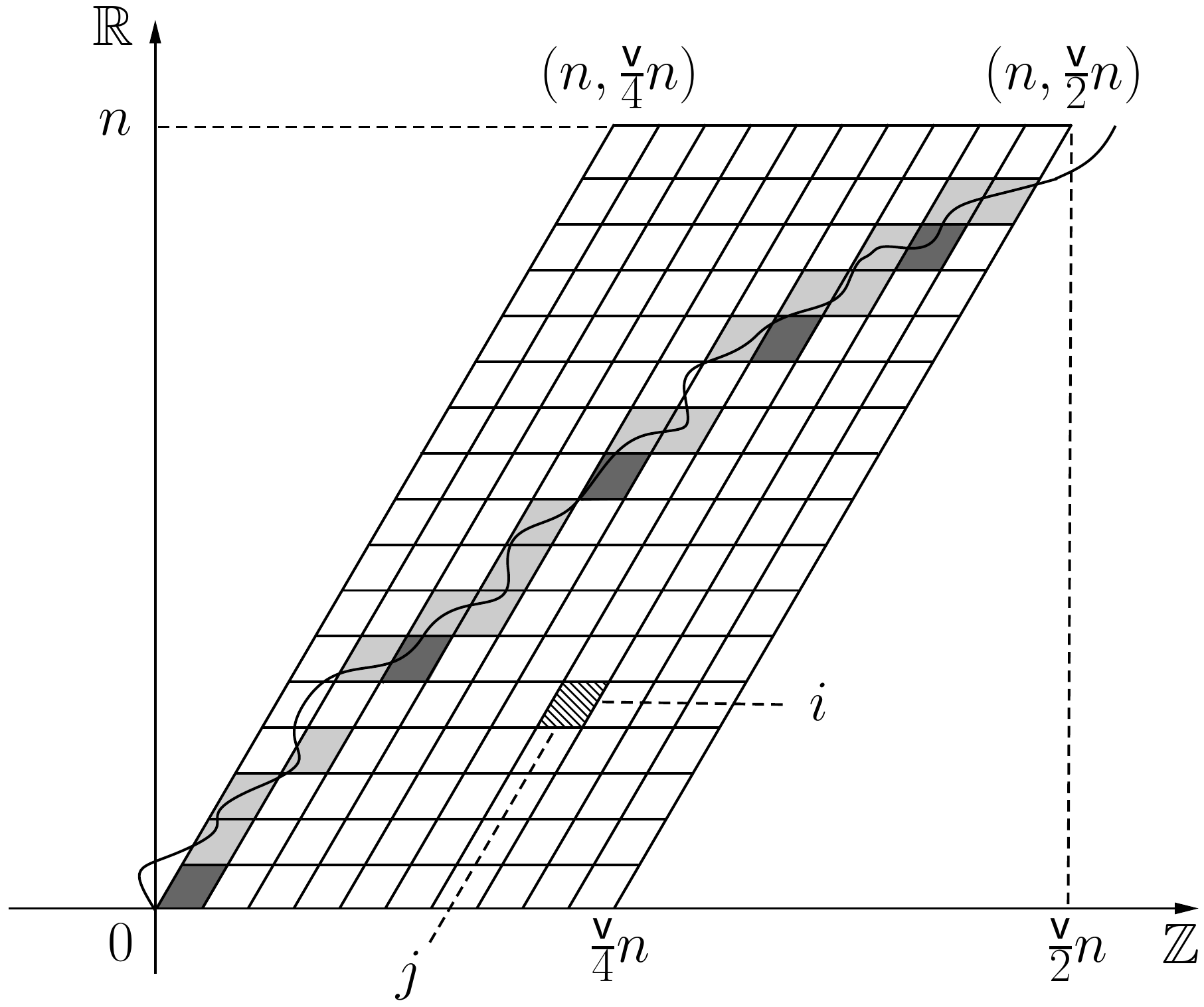}
\end{center}
\caption{
\label{figure: Block construction}
Construction of the trajectories of blocks, $\mathcal R$ and $\mathcal T$, from a deterministic trajectory $(k,Y_k)_{0\le k \le n}$ satisfying $Y_0=0$ and $Y_n \ge \frac{\mathsf v}{2}n$.
The trajectory $\mathcal R$ is marked by the shaded boxes, and the trajectory $\mathcal T \subset \mathcal R$ by the darker boxes.  
A box $C(i,j)$ is represented by the hatched region. }
\end{figure}

Recall now the definition \eqref{bad part} of bad points. 
A block $C(i,j)$ is said to be {bad} if at least one point in $C(i,j)$ is bad. 
For any block $C(i,j)$, we define the variable $E(i,j)$ by $E(i,j)=0$ if the block is not bad, and by
\begin{equation}\label{min of maximal distance}
E(i,j)=\inf\{m\geq 1:\ \exists (t,x)\in C(i,j) \textrm{ bad s.t.\@ } \exists p\in A_{t,x}\textrm{ s.t.\@ } x+\lambda (p) \in \tilde{T}^-_{i,j+m} \}
\end{equation}
if the block is bad, where 
\begin{equation*}
\tilde{T}^-_{i,j}=T^-_{il,jl + \frac{\mathsf v}{4}il}.
\end{equation*}
We note that $E(i,j) \ge 1$ if and only if the block $C(i,j)$ is bad. 
The crucial observation is that, for any $m\ge 1$, the event $\{E(i,j)=m\}$ is measurable with respect to
\begin{equation}\label{independence lemma percolation}
\sigma\left( U (t,x) :\ (t,x)\in \tilde{T}^-_{i,j+m} \setminus \tilde{T}^-_{i,j-1}\right).
\end{equation}
Moreover, using a union bound together with \eqref{eq:cone}, we get that there exists $c_1>0$ such that, for any $m\geq 1$,
\begin{equation}\label{eq:blck}
\P(E(i,j)\geq m)\leq e^{-c_1ml}.
\end{equation}

Let us now assume that $0 < c_0 \le \frac{\mathsf v}{16l}$. We have 
\begin{equation*}
\left\{\sharp\{i\leq n, \ (i,X_i)\text{ is good}\}\leq c_0 n,X_n\geq \frac{\mathsf v}{2}n\right\}\; \subset \; \left\{ \sharp\{ (i,j)\in \mathcal T ((i,X_i)_{1\le i \le n}) \textrm{ s.t. } C(i,j) \textrm{ is bad}\} \geq \frac{N}{4}\right \}, 
\end{equation*}
where the notation $\mathcal T ((i,X_i)_{1\le i \le n})$ means that the trajectory of blocks $\mathcal T$ is built from the trajectory $(i,X_i)_{1 \le i \le n}$. 
There exists $c_2>0$ (independent of $n$) so that the number of $\mathcal T$ trajectories is bounded by $\ed^{c_2\frac{n}{l}}$. Therefore
\begin{multline*}
P_0\left(\left\{\sharp\{i\leq n, \ (i,X_i)\text{ is good}\} \le c_0 n,X_n\geq \frac{\mathsf v}{2}n \right\} \right) \\
\;\le \; 
\ed^{c_2\frac{n}{l}} \max_{\mathcal T} P_0 \left(\sharp\{ (i,j)\in \mathcal T \textrm{ s.t. } C(i,j) \textrm{ is bad}\} \geq \frac{N}{4}\right),
\end{multline*}
where the maximum runs over all block trajectories $\mathcal T$ built from trajectories $(j,Y_j)_{0\le j \le n}$. 
Moreover, there exists $c_3> 0$ (independent of $n$) so that, given  $\mathcal T$, there are at most $\ed^{c_3\frac{n}{l}}$ ways to extract half of the blocks of $\mathcal T$. Thus   
\begin{equation*}
P_0 \left(\sharp\{ (i,j)\in \mathcal T \textrm{ s.t. } C(i,j) \textrm{ is bad}\} \geq \frac{N}{4}\right)
\; \le \; 
\ed^{c_3\frac{n}{l}} \max_{\mathcal T'\subset \mathcal T,\sharp \mathcal T'=N/4} P_0 \left(\forall (i,j)\in \mathcal T',\ C(i,j) \text{ is bad}\right)
\end{equation*}
(we assume that $N/4$ is an integer for the ease of notations, the other cases are analogous).

To get a bound on the maximum in this last expression, we need some extra notations. 
Let $j_1 < \dots < j_{N/4}$ be the set of points such that $(i,j_k) \in \mathcal T'$ for some unique $i$. 
We denote by $\Phi$ a partition of $\{ j_1 , \dots , j_{N/4}\}$ in non-empty intervals, by which we mean non-empty sets of the type $[a,b] \cap \{ j_1, \dots , j_{N/4}\}$, for $[a,b]$ an interval of $\R$.
We denote by $|\Phi|$ the number of sets in $\Phi$, 
by $\phi_k$ the sets of $\Phi$, 
by $|\phi_k| \ge 1$ the cardinal of each set,
by $l_k$ the smallest integer in the set $\phi_k$, 
by $r_k$ the largest integer in the set $\phi_k$. 
Moreover, to lighten the notations, let us simply write $C(j_k)$ (resp.\@ $E(j_k)$) for $C(i(j_k),j_k)$ (resp.\@ $E(i(j_k),j_k)$),
where $i(j_k)$ is the number such that $(i (j_k),j_k)\in \mathcal T'$ for a given $j_k \in \{ j_1, \dots, j_{N/4} \}$. 
Then 
\begin{align*}
&\P \left(\forall (i,j)\in \mathcal T',\ C(i,j) \text{ is bad}\right) 
\; = \;
\P \left(C(j_1), \dots , C(j_{N/4}) \text{ are bad} \right) \\
\;& \le \; 
\sum_{\Phi} \P \left( r_1-l_1+1\leq E(l_1) \leq l_2-l_1-1 , r_2-l_2+1\leq E(l_2)  \leq  l_3-l_2-1 , \dots , E(l_{|\Phi|}) \geq r_{|\Phi|}-l_{|\Phi|}+1\right)\\
\;& = \;
\sum_{\Phi} \P \left(  r_1-l_1+1\leq E(l_1) \leq l_2-l_1-1\right) \P\left( r_2-l_2+1\leq E(l_2)  \leq l_3-l_2-1 \right) \cdots \P \left(E( l_{|\Phi|}) )\geq r_{|\Phi|}-l_{|\Phi|}+1 \right)  \\
\;& \leq \;
\sum_{\Phi} \P \left(  r_1-l_1+1\leq E(l_1) \right) \P\left( r_2-l_2+1\leq E(l_2) \right) \cdots \P \left(E( l_{|\Phi|})) \geq r_{|\Phi|}-l_{|\Phi|}+1 \right) \\
\; & \le \;
\sum_{\Phi} \ed^{- c_1 l \sum_{i=1}^{|\Phi|} |\phi_i|}
\; \le \; 2^{N/4} \ed^{-c_1 l N/4}, 
\end{align*}
where the equality in front of the third line follows from \eqref{independence lemma percolation}, where the inequality in front of the fourth line follows from \eqref{eq:blck} together with the rough estimates $|r_k-l_k+1|\geq |\phi_k|$ for any $0\leq k\leq |\Phi|$,  and where the last estimates follows from the fact that the number of partitions $\Phi$ in intervals is bounded by $2^{N/4}$. 
Since $N \ge c_5 n/l$, we obtain finally that, for some constant $c_6 < + \infty$, 
\begin{equation*}
P_0\left(\sharp\{i\leq n, \ (i,X_i)\textrm{ is good}\}\leq c_0 n,X_n\geq \frac{\mathsf v}{2}n \right) \; \le \; \ed^{c_6\frac{n}{l}}e^{-c_5n}. 
\end{equation*}
By taking $l$ large enough, this is bounded by $1/ n^{q+2}$ for $n$ large enough, from where \eqref{inter} follows.

We now turn to the second step of the proof, and derive the result from \eqref{inter}. 
Let $g_1$ be the first good time, i.e. the first time such that $(g_1,X_{g_1})$ is good, and define then, by iteration on $i \ge 1$, $g_{i+1}$ as the first good time after $g_i+l$. 
By \eqref{inter}, the sequence $(g_i)_{i\ge 1}$ is almost surely infinite. 
Remark that $(g_i)_{i\geq 1}$ are stopping times with respect to the filtration
\begin{equation}\label{def filtration}
\mathcal{F}_k=\sigma((X_i)_{i\leq k},(\xi(u,x))_{u\leq k, x\in\Z}),\qquad k\geq 1.
\end{equation}
There exists a constant $\epsilon > 0$ such that, for any $i\geq 1$,
\begin{align*}
P_0(g_{i}+l &\textrm{ is a candidate}|\mathcal{F}_{g_{i-1}+l})\\
&\;\ge\; P_0(U(g_i+l,X_{g_i}+l)=U(g_i,X_{g_i}+l),X_{g_i+1}-X_{g_i}=1,\cdots=X_{g_i+l+1}-X_{g_i+l}=1|\mathcal{F}_{g_{i-1}+l})\\
&\;\ge\; \kappa^{l+1}P_0(U(l,0)=U(0,0)) \; \ge \; \epsilon,
\end{align*}
where we have used the ellipticity to bound the conditional probability that the walker does $l+1$ steps to the right. 
Let us denote by $(Z_i)_{i\ge 1}$ a sequence of variables with values in $\{0,1 \}$ such that $Z_i = 1$ if $g_i + l$ is candidate, and $Z_i = 0$ otherwise. 
The above implies that $P_0 (Z_k = 1 | Z_1, \dots , Z_{k-1}) \ge \epsilon$. 
Therefore, for $c_7 > 0$ small enough, there exists $c_8 > 0$ so that 
\begin{equation}\label{proof lemma candidates estimate sum blabla}
P_0 \left(\sharp \{ n \le g_i + l : (n,X_n) \text{ is candidate} \} \le c_7 i \right) 
\; \le \; 
P_0 \left(\sum_{k=1}^i Z_k \le c_7 i \right) 
\; \le \; 
\ed^{-c_8 i}.
\end{equation}

We now compute 
\begin{multline*}
P_0 \left(\sharp \{ i \le n : (i,X_i) \text{ candidate} \} \le cn \right) 
\; \le \; \\
P_0 \left(\sharp \{ i \le n : (i,X_i) \text{ candidate} \} \le cn ,  \sharp \{ i \le n : (i,X_i) \text{ good} \} \geq c_0 n \right) +  P_0 \left(\sharp \{ i \le n : (i,X_i) \text{ good} \} < c_0 n \right) .
\end{multline*}
The second term is bounded by $1 / n^{q+1}$ thanks to \eqref{inter}. 
For the first one, we observe that on the event $\{ \sharp \{ i \le n : (i,X_i) \text{ good} \} \geq c_0 n \}$, there exists $c_9 > 0$ such that $g_{c_9 n} \le n$. 
Therefore, if $c>0$ is taken so that $c \le c_7 c_9$, we obtain
\begin{multline*}
P_0 \left(\sharp \{ i \le n : (i,X_i) \text{ candidate} \} \le cn ,  \sharp \{ i \le n : (i,X_i) \text{ good} \} \geq c_0 n \right)
\; \le \;\\
P_0 (\sharp \{ i \le g_{c_9 n} : (i,X_i) \text{ candidate} \} \le cn )
\; \le \; 
\ed^{- c_8 c_9 n }
\end{multline*}
thanks to \eqref{proof lemma candidates estimate sum blabla}.
\end{proof}

\subsection{Proof that $(0,X_0)$ satisfies (\ref{re4}-\ref{re5}) with positive probability}\label{subsection: beginning two other conditions}
We prove here that, with positive probability, the walker lives in $T^+_{0,0}$ and does not visit any of the particles that were initially at its left.
We introduce
\begin{align}
D&=\inf\left\{n\geq 0 \textrm{ s.t. } X_n < X_0+\frac{\mathsf v}{4}n\right\},\\
\label{defFF} F&=\inf\left\{n\geq 0,\ \textrm{ s.t. }\exists\ x\geq X_0+ \frac{\mathsf v}{4}n,\ \xi(0,\mu(n,x))<X_0\right\},
\end{align}
with the convention that $\inf \emptyset=+\8$. Remark that $F$ is defined as a discrete time and that $D$, $F$ are stopping times with respect to the filtration $(\mathcal{F}_k)_{k\geq 0}$ defined in \eqref{def filtration}. 
The variable $D$ can be considered as a function of $X$, and $F$ as a function of $\xi$ and $X_0$ only. 
Let $H$ be the infimum of these two stopping times:
\begin{equation*}
H=D\wedge F.
\end{equation*}

We claim that 
\begin{equation}\label{bascule}
P_0(H=+\8)=P_0(D=F=+\8)>0.
\end{equation}

\begin{proof}[Proof of \eqref{bascule}]
In this proof, as $X_0=0$ a.s., we consider $F$ as a function of $\xi$ only, i.e.\@ $F=F(\xi,0)$.
As $\liminf X_n/n>\mathsf v$ $P_0-$a.s., we deduce by monotonicity that
\begin{equation}\label{eq:limL}
\lim_{L\to +\8} P_0 \left(X_n\geq \frac{\mathsf v}{4}(n-L),\ \forall n\geq 0\right) \; = \; 1.
\end{equation}
Using the same type of computation as for \eqref{eq:cone}, we obtain that 
\begin{equation*}
\P_1(F=+\8)>0,
\end{equation*}
and we choose $L$ large enough so that (recall \eqref{eq:limL})
\begin{equation}
\label{moitie}
P_0\left(X_n\geq \frac{\mathsf v}{4}(n-L),\ \forall n\geq 0\right) \; >\; 1-\frac{\P_1(F=+\8)}{2}.
\end{equation}
Finally, for $L'$ large enough so that $L'-v/4L'\geq L$,
\begin{align*}
P_0(H=+\8)&\geq \P \times P^{\omega}_{0,0}(F=+\8,X_1-X_0=1,\cdots=X_{L'}-X_{L'-1}=1,X_n\geq \frac{\mathsf v}{4}n,\ \forall n \geq L')\\
		&\geq \E\left(F=+\8\ ,P^{\omega}_{0,0}(X_1-X_0=1,\cdots=X_{L'}-X_{L'-1}=1)P^{\omega}_{L',L'}(X_n\geq \frac{\mathsf v}{4}(n-L),\ \ \forall n\geq 0)\right)\\
		&\geq \kappa^{L'}\E\left(F=+\8\ ,P^{\omega}_{L',L'}(X_n\geq \frac{\mathsf v}{4}(n-L),\ \ \forall n\geq 0)\right),
\end{align*}
where we have used ellipticity to get the last line, and where $\kappa$ is defined in \eqref{kappa}.
As the law of the environment is invariant by translation,  $\P \times P^{\omega}_{L',L'}(X_n\geq \frac{\mathsf v}{4}(n-L),\ \forall n\geq 0)$ is equal to $P_0(X_n\geq \frac{\mathsf v}{4}(n-L),\ \forall n\geq 0)$, so that finally, using $\eqref{moitie}$,
\begin{equation*}
P_0(H=+\8)\geq \kappa^{L'} \frac{\P_1(F=+\8)}{2} >0.
\end{equation*}
\end{proof}

\subsection{Building the first renewal point}
\label{sub:Building the first renewal point}
We define, for $n\geq 0$, the shift $\theta_n$ on the space $ \Xi\times \mathcal{P}$ by  $\theta_n (\xi,x)=(\xi',x')$, where
\begin{align*}
\xi'(t,y)&=\xi(t+n,y), \quad (t,y)\in\R^+\times \Z,\\
x'_i&=x_{n+i},\qquad i\geq 0,
\end{align*}
and we consider the increasing sequence of stopping times with respect to the filtration $(\mathcal{F}_k)_{k\geq 1}$  (see \eqref{def filtration}) defined by
\begin{align*}
S_0&=1,\qquad \textrm{ and for } k\geq 1,\\
R_{k}&=\inf\{n\geq S_{k-1} \textrm{ s.t. } (n,X_n) \textrm{ is a candidate} \},\\
S_{k}&=H\circ \theta_{R_k} + R_k.
\end{align*}
Define
\begin{equation}
\label{eq:renouv}
K=\inf\{k\geq 1 \textrm{ s.t. } R_k<+\8 \textrm{ and } S_k=+\8 \}.
\end{equation}
 We claim that 
 \begin{equation}
 \label{Kfini}
 P_0-a.s.,\qquad  K<+\8,
 \end{equation}
so that $\tau:=R_K$ is well defined and moreover $\tau$ is a renewal time in the sense that it satisfies \eqref{re1}-\eqref{re5}. \\
\begin{proof}[Proof of \eqref{Kfini}]
We first deduce from Lemma \ref{l:densite} that
 \begin{equation*}
  P_0-a.s.,\qquad \forall k\geq 0,  \qquad \{S_k<+\8\}\subset \{R_{k+1}<+\8\},
 \end{equation*}
so that
\begin{equation}
\label{recK}
P_0(R_{k+1}<+\8)=P_0(R_{k}<+\8)-P_0(R_{k}<+\8,S_{k}=+\8).
\end{equation}
To compute the last term of \eqref{recK}, we use first that $P^\omega$ is Markovian:
\begin{align}
\label{cle}
P_0(R_{k}&<+\8,S_{k}=+\8)=\sum_{x\in\Z , i\geq 0}P_0(R_{k}=i,X_{R_{k}}=x,S_k=+\8)\\
&=\sum_{x\in\Z , i\geq 0}\E_1\left( \E_2 \left( P^{\omega}_{0,0}(R_{k}=i,X_{R_k}=x) P^{\omega}_{i,x}(H(\theta_i\xi,X)=+\8) \right) \right), \nonumber
\end{align}
where $\E_1$ (resp. $\E_2$) denote the expectation with respect to $\P_1$ (resp. $\P_2$).
For any fixed $\xi$, let $V_{i,x}$ be the set of labels of particles that are strictly at the left of $x$ at time $i$, i.e.
\begin{equation*}
V_{i,x}=\{\mu(i,y),\ y< x\}.
\end{equation*}
Note that given $\xi$, $P^{\omega}_{0,0}(R_{k}=i,X_{R_k}=x)$ is measurable with respect to $\sigma(\nu(\mu), \mu\in V_{i,x})$ while $P^{\omega}_{i,x}(H(\theta_i\xi,X)=+\8)$ is measurable with respect to $\sigma(\nu(\mu), \mu \notin V_{i,x})$. These two variables are thus independent under $\P_2$:
\begin{equation*}
 \E_2\left( P^{\omega}_{0,0}(R_{k}=i,X_{R_k}=x)P^\omega_{i,x}(H=+\8) \right)=\E_2 \left(  P^{\omega}_{0,0}\left(R_{k}=i,X_{R_k}=x\right)\right)\E_2\left( P^{\omega}_{i,x}\left(H(\theta_i\xi,X)=+\8 \right)\right).
\end{equation*}
Finally  $\E_2  \left( P^{\omega}_{0,0}(R_{k}=i,X_{R_k}=x)\right)$ is measurable with respect to $\sigma(U(s,y), s\leq i, y\in\Z)$ and $\E_2\left(P^{\omega}_{i,x}(H(\theta_i\xi,X)=+\8) \right)$ is measurable with respect to $\sigma(U(s,y), s\geq i, y\in\Z)$. These two variables are thus independent under $\P_1$ and finally
\begin{equation*}
P_0(R_{k}<+\8,S_{k}=+\8)=\sum_{x\in\Z , i\geq 0} P_0(R_{k}=i,X_{R_k}=x) P_{i,x}(H(\theta_i\xi,X)=+\8).
\end{equation*}
As the law of environment is invariant by space-time translation, for all $x\in\Z$ and $i\geq 0$, $P_{i,x}(H(\theta_i\xi,X)=+\8)=P_{0}(H=+\8)$ so that
\begin{equation*}
P_0(R_{k}<+\8,S_{k}=+\8)=P_0(R_{k}<+\8)P_0(H=+\8),
\end{equation*}
and going back to \eqref{recK},
\begin{equation}
\label{fin}
P_0(R_{k+1}<+\8) \leq P_0(R_{k}<+\8)(1-P_0(H=+\8)).
\end{equation}
By iteration, we conclude the proof of \eqref{Kfini}.
\end{proof}

\subsection{Defining a sequence of renewal points by iteration}\label{sub:Building renewal points}
Remark that $\tau_1$ is a function of $X$ and $\xi$ so that, in order to iterate the construction, we study the law of these two processes after a time $\tau_1$. 
That is the purpose of the next proposition.
\begin{Proposition}
\label{renew1}
The process and the environment after the first renewal time, $(X_{\tau_1+n}-X_{\tau_1})_{n\geq 0},(\xi(\tau_1+t,x))_{t\geq 0, x\in\Z} $, 
are independent from $(X_{n\wedge \tau_1})_{n\geq 0},(\xi(t\wedge \tau_1,x))_{t\geq 0,x\in\Z}$ and have same law as $(X_n)_{n\geq 0},(\xi(t,x))_{t\geq 0,x\in\Z}$ under $P_0(\cdot | H=+\8)$.
\end{Proposition}

\begin{proof}
This type of proof is quite usual (see e.g.\@ \cite{SZ} for the case of a static environment). 
We adapt it explicitly to our case in order to be exhaustive.
We define 
\begin{equation*}
\mathcal{G}_1=\sigma \left( (X_{n \wedge \tau_1})_{n\geq 0}, (\xi(s,x))_{s\leq \tau_1,x\in\Z} \right).
\end{equation*}
We have to prove that for any bounded functions $\phi_1,\phi_2$,
\begin{align} 
\label{goal}
E_0(\phi_1((X_{\tau_1+n}-X_{\tau_1})_{n\geq 0})&\phi_2(\xi(\tau_1+t,x)_{x\in\Z,t\geq 0})| \mathcal{G}_1)\nonumber\\
&=E_0(\phi_1((X_{t})_{t\geq 0})\phi_2((\xi(t,x))_{x\in\Z,t\geq 0})| H=+\8).
\end{align}
Consider the variables $\psi_1((X_{n\wedge \tau_1})_{n\geq 0})$ and $\psi_2( (\xi(s,x))_{s\leq \tau_1,x\in\Z})$, where $\psi_1$ and $\psi_2$ are bounded functions. 
If $Z$ is some process and $t$ some time  (possibly random), the process stopped at time $t$ is denoted $Z^t$: for any $s\geq 0$, $Z^t_s=Z_{s\wedge t}$. 
Using the same arguments as in the proof \eqref{cle}-\eqref{fin}, we deduce that (we write explicitly the arguments of the functions only when they change from line to line)
\begin{align*}
&E_0\left( \phi_1\phi_2 \psi_1\psi_2   \right)\\
&=\sum_{k,x,n}E_0\left(\phi_1(\theta_n X-X_{n})\phi_2(\theta_n \xi)\psi_1(X^n) \psi_2( \xi^n) , R_k=n, X_{R_k}=x, H\circ \theta_n=+\8 \right)\\
&= \sum_{k,x,n}\E^1\left( \phi_2 \psi_2 \E^2[  E^{\omega}_{0,0}(\psi_1, R_k=n, X_{R_k}=x) E^\omega_{n,x}(\phi_1(X),H(\theta_n\xi,X)=+\8)]\right)\\
&= \sum_{k,x,n}\E^1\left( \psi_2\E^2[ E^{\omega}_{0,0}(\psi_1, R_k=n, X_{R_k}=x) ] \phi_2 \E^2  [ E^\omega_{n,x}(\phi_1(X) ,H(\theta_n\xi,X)=+\8)]\right)\\
&= \sum_{k,x,n}\E^1\left( \psi_2 \E^2[  E^{\omega}_{0,0}(\psi_1, R_k=n, X_{R_k}=x) ] \right) \E^1\left(\phi_2 \E^2  [E^\omega_{n,x}((\phi_1(X),H(\theta_n\xi,X)=+\8)]\right)\\
&= \sum_{k,x,n}\E^1\left( \psi_2 \E^2[  E^{\omega}_{0,0}(\psi_1, R_k=n, X_{R_k}=x) ] \right) E_{n,x}(\phi_1\phi_2(\theta_n\xi),H(\theta_n\xi,X)=+\8) \\
&=E_{0}(\phi_1\phi_2|H=+\8) \sum_{k,x,n}\E^1\left( \psi_2 \E^2[ E^{\omega}_{0,0}(\psi_1, R_k=n, X_{R_k}=x) ] \right) P_0(H=+\8)\\
&=E_0(\phi_1\phi_2|H=+\8) E(\psi_1(X^{\tau_1})\psi_2(\xi^{\tau_1})),
\end{align*}
where the last line is obtained from the previous one by taking $\phi_1=\phi_2=1$ in the same computation. This concludes the proof of \eqref{goal} and thus the proof of Proposition \ref{renew1}.
\end{proof}

As $P_0(H=+\8) >0$, $\tau$ is also defined and finite a.s. under $P_0(\cdot|H=+\8)$. 
We can thus define $\tau_2$, the second renewal time, by $\tau_2=\tau_1+\tau_1((X_{n+\tau_1}-X_{\tau_1})_{n\geq 0},(\xi(n+\tau_1,x))_{t\geq 0,x\in\Z})$
and, by iteration, we define in the same way an increasing sequence of renewal times $(\tau_k)_{k\geq 1}$  that are finite $P_0-$a.s.

The interest of this construction lies in the following
\begin{Proposition}
\label{prop:ccl}
Under $P_0$, $\left( (X_{\tau_{k}+\cdot}-X_{\tau_k})_{0\leq t\leq \tau_{k+1}-\tau_k}\ , \tau_{k+1}-\tau_k  \right)_{k\geq 1}$ are i.i.d. with the same law as $\left( (X_t)_{0\leq t\leq \tau_1}, \tau_1 \right)$ under $P_0(\cdot | H=+\8)$.
\end{Proposition}
The proof follows by induction from Proposition \ref{renew1}.
\begin{Corollary}
Under $P_0$, $(X_{\tau_{k+1}}-X_{\tau_{k}},\tau_{k+1}-\tau_k)_{k\geq 1}$ are positive i.i.d. random variables with the same law as $(X_{\tau_1},\tau_1)$ under $P_0(\cdot | H=+\8)$. 
\end{Corollary}

\subsection{Control on the moments of $\tau_1$ and conclusion of the proof}\label{renewal: last subsection}

\begin{Proposition}
\label{moment tau}
For $n\geq 0$ large enough,
\begin{equation*}
P_0(\tau_1>n)\leq \frac{1}{n^3}.
\end{equation*}
In particular, as $P_0(H=+\8)>0$, this implies
\begin{equation*}
E_0(\tau_1|H=+\8)<+\8 \quad \textrm{and} \quad E_0(\tau_1^2|H=+\8)<+\8.
\end{equation*}
\end{Proposition}
\begin{proof}
First observe that for any $n\in\N$,
\begin{equation}
\label{dec1}
P_0(\tau>n)\leq P_0(\tau>n,K<\ln^2 n)+P_0(K\geq \ln^2 n),
\end{equation}
with $K$ defined in \eqref{eq:renouv}. As $K$ is distributed under $P_0$ like a geometric random variable with success parameter $P_0(H=+\8)>0$ ( see \eqref{bascule}), we obtain that for $n$ large enough
\begin{equation*}
P_0(K\geq \ln^2 n)\leq \frac{1}{n^4}.
\end{equation*}
In order to deal with the first term in the right had side of \eqref{dec1}, we decompose
\begin{multline}
\label{deux termes}
P_0(\tau>n,K<\ln^2 t)\leq P_0\left(\tau>n, \{\forall\ 1\leq k<K, S_k-R_k\leq \sqrt{n}\},K\leq \ln^2 n\right)\\+P_0\left(\tau>n,\{\exists\ 1 \leq k<K, S_k-R_k> \sqrt{n}\},K\leq \ln^2 n\right),
\end{multline}
and again we have to control two terms. For the first one, note that 
\begin{equation*}
P_0\left(\tau>n, \{\forall\ 1 \leq k<K, S_k-R_k\leq \sqrt{n}\},K\leq \ln^2 n\right)\leq P_0(\sharp\{i\leq n, \ (i,X_i) \textrm{ is a candidate }\} \leq \sqrt{n} \ln^2 n),
\end{equation*}
so that, using Proposition \ref{l:densite}, we obtain that for $n$ large enough,
\begin{equation*}
P_0\left(\tau>n, \{\forall\ 1 \leq k<K, S_k-R_k\leq \sqrt{n}\},K\leq \ln^2 n\right)\leq \frac{1}{n^4}.
\end{equation*}

For the second term in \eqref{deux termes}, we observe that
\begin{equation*}
P_0\left(\tau>n,\{\exists\ 1 \leq k<K, S_k-R_k> \sqrt{n}\},K\leq \ln^2 n\right)\leq \sum_{k=1}^{\ln^2 n}P_0\left( R_k<+\8, \sqrt{n}< S_k-R_k<+\8 \right).
\end{equation*}
We study each term of the sum in the same way. Fixing some integer $k$ such that $1\leq k\leq \ln^2 n$,
\begin{align*}
P_0(R_{k}&<+\8,\sqrt{n}<S_{k}-R_k<+\8)=\sum_{x\in\Z , i\geq 0}P_0(R_{k}=i,X_{R_{k}}=x,\sqrt{n}<S_{k}-R_k<+\8)\\
&=\sum_{x\in\Z , i\geq 0}\E_1\left( \E_2 \left( P^{\omega}_{0,0}(R_{k}=i,X_{R_k}=x) P^{\omega}_{i,x}(\sqrt{n}< H(\theta_i\xi,X)<+\8) \right)\right). \nonumber
\end{align*}

Given $\xi$ and $\nu$, for all $i\geq 0$ and $x\in\Z$,
\begin{equation}
\label{twoterms}
P^{\omega}_{i,x}(\sqrt{n}< H(\theta_i\xi,X)<+\8)\leq P^{\omega}_{i,x}(\sqrt{n}< D<+\8) 1_{\{F(\theta_i\xi,x)=+\8\}} + 1_{\{ \sqrt{n}< F(\theta_i\xi,x)<+\8\}}.
\end{equation}
We consider separately these two terms. 
For the first one, observe that given $\xi$, $P^{\omega}_{0,0}(R_{k}=i,X_{R_k}=x)$ is measurable with respect to $\sigma(\nu(\mu), \mu\in V_{i,x})$ while $P^{\omega}_{i,x}(\sqrt{n}< D<+\8)$ is measurable with respect to $\sigma(\nu(\mu), \mu \notin V_{i,x})$. These two variables are thus independent under $\P_2$:
\begin{multline*}
 \E_2\left( P^{\omega}_{0,0}(R_{k}=i,X_{R_k}=x)P^\omega_{i,x}(\sqrt{n}< D<+\8)\right)\\=\E_2 \left( P^{\omega}_{0,0}(R_{k}=i,X_{R_k}=x)\right)\E_2\left(P^{\omega}_{i,x}(\sqrt{n}< D<+\8 )\right).
\end{multline*}
Observe that $\E_2  \left( P^{\omega}_{0,0}(R_{k}=i,X_{R_k}=x)\right)$ is measurable with respect to $\sigma(U(s,y), s\leq i, y\in\Z)$ and $\E_2\left(P^{\omega}_{i,x}(\sqrt{n}< D<+\8)\right) 1_{\{F(\theta_i\xi,x)=+\8\}}$ is measurable with respect to $\sigma(U(s,y), s\geq i, y\in\Z)$. These two variables are thus independent under $\P_1$ and
\begin{multline*}
\sum_{x\in\Z , i\geq 0}\E_1\left( \E_2 \left( P^{\omega}_{0,0}(R_{k}=i,X_{R_k}=x)\right)\ \E_2\left(P^{\omega}_{i,x}(\sqrt{n}< D<+\8)\right) 1_{\{F(\theta_i\xi,x)=+\8\}} \right)\\
=\sum_{x\in\Z , i\geq 0} P_0(R_{k}=i,X_{R_k}=x) P_{i,x}(\sqrt{n}< D<+\8, F(\theta_i\xi,x)=+\8).
\end{multline*}
As the law of environment is invariant by space-time translation, for all $x\in\Z$ and $i\geq 0$,
\begin{equation*}
 P_{i,x}(\sqrt{n}< D<+\8, F(\theta_i\xi,x)=+\8)=P_{0}(\sqrt{n}< D<+\8, F=+\8),
\end{equation*}
so that
\begin{multline*}
\sum_{x\in\Z , i\geq 0}\P \left( P^{\omega}_{0,0}(R_{k}=i,X_{R_k}=x) P^{\omega}_{i,x}(\sqrt{n}< H(\theta_i\xi,X)<+\8) \right)\\=P_0(R_{k}<+\8)P_0(\sqrt{n}< D<+\8, F=+\8).
\end{multline*}
Using Proposition \ref{Proposition: Renormalization large gamma} and \eqref{admissible points typical}, we obtain that, for $n$ large enough,
\begin{equation*}
P_0(\sqrt{n}< D<+\8, F=+\8) \;\le\; P_0(\exists k\geq \sqrt{n}, X_k\leq X_0+\frac{\mathsf{v}}{4}k)
\; \le \; \sum_{k=\sqrt{n}}^{+\8}C e^{-k^{\alpha}}+e^{-\phi_k^{1/4}} \;\le\; \frac{1}{n^4}.
\end{equation*}
We turn to the second term in \eqref{twoterms}.  As $1_{\{ \sqrt{n}< F(\theta_i\xi,x)<+\8\}}$ is measurable with respect to $\sigma(U(s,y), s\geq i, y\in\Z)$  and $ \E_2 \left( P^{\omega}_{0,0}(R_{k}=i,X_{R_k}=x) \right)$ is measurable with respect to $\sigma(U(s,y), s\leq i, y\in\Z)$, these two variables are independent. Thus, using also the invariance of $P_0$ under space-time translations,
\begin{equation*}
\sum_{x\in\Z , i\geq 0}\E_1 \left(  1_{\{ \sqrt{n}< F(\theta_i\xi,x)<+\8\}} \E_2 \left( P^{\omega}_{0,0}(R_{k}=i,X_{R_k}=x) \right) \right) = P_0(R_k<\8)P_0(\sqrt{n}< F<+\8)
\end{equation*}
Using the same type of computations as for \eqref{eq:cone}, we obtain that, for $n$ large enough,
\begin{equation*}
P_0(\sqrt{n}< F<+\8)\leq \frac{1}{n^4}.
\end{equation*}
That concludes the proof.
\end{proof}
We are now ready to conclude the proof of Theorem \ref{Theorem: Drift to the right}. 
For random walks in static random environments, it is well-known that the existence of a finite second moment for $\tau_2-\tau_1$ implies a law of large numbers and an annealed central limit theorem (see \cite{SZ} and \cite{SZN}).
The proof carries over to our case, and we recall it here for sake of completeness. 

We start with the law of large numbers, i.e.\@ point $1$ of Theorem \ref{Theorem: Drift to the right}. 
For $n\geq 0$, $k(n)$ denotes the label of the "renewal slab" that contains $n$, i.e.\@ the unique integer such that $\tau_{k(n)} \leq n < \tau_{k(n)+1}$. We can thus control the walker via
\begin{equation}
\label{encadre}
 \frac{X_{\tau_{k(n)}}}{ \tau_{k(n)+1}}  \leq \frac{X_n}{n} \leq \frac{X_{\tau_{k(n)+1}}}{ \tau_{k(n)}}.
\end{equation}
Rewrite the right term as
\begin{equation}
\label{prodlgn}
 \frac{X_{\tau_{k(n)+1}}}{ \tau_{k(n)}} =  \frac{X_{\tau_1}+\sum_{i=2}^{k(n)+1}\Delta_{i}}{ k(n)}  \frac{k(n)}{ \tau_1+\sum_{i=2}^{k(n)}(\tau_{i+1}-\tau_i)},
\end{equation}
so that, using Proposition \ref{prop:ccl} and the law of large numbers, it is seen to converge $P_0$-a.s.\@ to 
\begin{equation}
\label{defV}
v(\gamma)=\frac{E_0(X_{\tau_1}|H=+\8)}{E_0(\tau_1|H=+\8)},
\end{equation}
which is always well defined and positive thanks to \eqref{bascule} and Proposition \ref{moment tau}. 
Using the same decomposition as in \eqref{prodlgn} to study the left term in \eqref{encadre}, we obtain the law of large numbers stated in point $1$ of Theorem \ref{Theorem: Drift to the right}.

We turn to the proof of point $2$ of Theorem \ref{Theorem: Drift to the right}, mainly following \cite{SZN}. For $j\geq 1$, define
\begin{equation*}
Z_j =X_{\tau_{j+1}}-X_{\tau_{j}} -(\tau_{j+1}-\tau_j) v(\gamma).
\end{equation*}
These variables are i.i.d (Proposition \ref{renew1}), centered, \eqref{defV} and admit a finite second moment (Proposition \ref{moment tau}). It thus follows from Donsker's theorem that 
\begin{equation}
\label{Prems}
 \left( \frac{\sum_{i=1}^{\lfloor nt \rfloor} Z_j}{\sqrt{n}} \right)_{t\geq 0}
\end{equation}
converges in law to a Brownian motion  with variance $E_0(Z_1^2)$, that is positive as $Z_1$ is not $P_0-$a.s.\@ constant.
As a consequence of Proposition \ref{prop:ccl}, the law of large numbers and Dini's theorem,
\begin{equation*}
P_0-a.s.,\quad \forall T>0,\quad \sup_{0\leq t \leq T} \left| \frac{k(\lfloor tn \rfloor)}{n}-\frac{t}{E_0(\tau_2-\tau_1)}\right|.
\end{equation*}
Therefore, we deduce from \eqref{Prems} that
\begin{equation}
\label{Deuze}
 \left( \frac{\sum_{i=1}^{k(\lfloor nt \rfloor)} Z_j} {\sqrt{n}} \right)_{t\geq 0}
\end{equation}
converges in law to a Brownian motion  with variance $E_0\left( (Z_1)^2 \right)/E_0(\tau_2 -\tau_1)$.
Finally, observe that $P_0-$a.s., for all $T>0$,
\begin{equation*}
\sup_{0\leq t \leq T} \left| \frac{X_{nt}-ntv(\gamma)}{\sqrt{n}}- \frac{\sum_{j=1}^{k(\lfloor nt \rfloor)} Z_j} {\sqrt{n}}\right| 
\leq (v(\gamma)+1)\max_{0\leq k\leq \lfloor nT \rfloor}\frac{\tau_{k+1}-\tau_k}{\sqrt{n}},
\end{equation*}
with the convention $\tau_0=0$. Let us prove that the right hand side converges to $0$ as $n\to \infty$ in probability, using Proposition \ref{prop:ccl} and Proposition \ref{moment tau}. Indeed, for any $\epsilon>0$,
\begin{align*}
P_0\left( \max_{0\leq k\leq \lfloor nT \rfloor} \tau_{k+1}-\tau_k \geq \epsilon \sqrt{n}\right)&\leq P_0\left( \tau_{1}\geq \epsilon \sqrt{n} \right)+ (nT+1) P_0\left( \tau_{2}-\tau_1 \geq \epsilon \sqrt{n} \right)\\
& \leq P_0\left( \tau_{1}\geq \epsilon \sqrt{n} \right)+ \frac{nT+1}{\epsilon^3 n^{3/2}} 
\end{align*}
and the right hand side converges to $0$ when $n$ goes to $+\8$. 
We have thus proven that the Skorohod distance between  $\left( \frac{X_{nt}-ntv(\gamma)}{\sqrt{n}} \right)_{t\geq 0}$ and $\left( \frac{\sum_{i=1}^{k(\lfloor nt \rfloor)} Z_j} {\sqrt{n}} \right)_{t\geq 0}$ goes to $0$ in $P_0-$probability. 
We deduce from the convergence in law of the latter that 
\begin{equation*}
 \left( \frac{X_{nt}-ntv(\gamma)}{\sqrt{n}} \right)_{t\geq 0}
\end{equation*}
converges in law to a Brownian motion  with variance $\sigma^2:=E_0\left( (X_{\tau_2}-X_{\tau_1})^2 \right)/E_0(\tau_2 -\tau_1)$.



\vspace{0.5cm}

\textsc{
Fran\c cois Huveneers,
CEREMADE  - UMR CNRS 7534 - 
Universit\' e de Paris-Dauphine,
Place du Mar\' echal De Lattre De Tassigny,
75775 PARIS CEDEX 16, FRANCE.}\\
\textsc{Email:}~\texttt{huveneers@ceremade.dauphine.fr}

\vspace{0.5cm}

\textsc{
Fran\c cois Simenhaus,
CEREMADE  - UMR CNRS 7534 - 
Universit\' e de Paris-Dauphine,
Place du Mar\' echal De Lattre De Tassigny,
75775 PARIS CEDEX 16, FRANCE.}\\
\textsc{Email:}~\texttt{simenhaus@ceremade.dauphine.fr}

\end{document}